\documentclass[8pt,svgnames,oneside]{book}
\usepackage[english]{babel}
\usepackage[T1]{fontenc}
\usepackage[margin=3cm]{geometry}

\usepackage{titlesec}
\usepackage[mode=image|tex]{standalone}
\usepackage{graphicx}
\usepackage{eso-pic}
\usepackage{amsmath,amsfonts,amssymb,amscd,amsthm}
\usepackage{mathrsfs,dsfont}
\usepackage{enumitem}
\usepackage{xfrac}
\usepackage[breakable]{tcolorbox}
\usepackage[colorlinks, citecolor=blue, linkcolor=purple]{hyperref}

\usepackage{tikz}
\usepackage{tikz-3dplot}
\usepackage{tkz-euclide}
\usepackage[siunitx]{circuitikz}
\usepackage{pgfplots}
\usepgfplotslibrary{patchplots}
\usepgfplotslibrary{fillbetween}
\usepgfplotslibrary{colorbrewer}
\pgfplotsset{compat=1.18}
\usetikzlibrary{calc,patterns,angles,quotes}
\usetikzlibrary{backgrounds, intersections}
\usetikzlibrary{babel, automata}

\newlength{\drop}

\newcommand{\R}{\mathbb{R}}

\newcommand{\N}{\mathbb{N}}
\newcommand{\Z}{\mathbb{Z}} 
\newcommand{\E}{\mathbb{E}} 

\newcommand{\Ball}{\mathbb{B}} 
\newcommand{\Sph}{\mathbb{S}}

\providecommand{\dom}{\mathop{\rm dom}\nolimits}
\providecommand{\gph}{\mathop{\rm gph}\nolimits}

\providecommand{\rint}{\mathop{\rm rint}\nolimits}
\providecommand{\ext}{\mathop{\rm ext}\nolimits}

\providecommand{\co}{\mathop{\rm co}\nolimits}
\providecommand{\cco}{\mathop{\overline{\rm co}}\nolimits}

\providecommand{\argmin}{\mathop{\rm argmin}}

\newcommand{\ind}{\mathds{1}}

\providecommand{\wto}{\mathop{\rightharpoonup}\nolimits}
\providecommand{\tto}{\mathop{\rightrightarrows}\nolimits}

\usepackage[round]{natbib}

\definecolor{deepcarmine}{rgb}{0.66, 0.13, 0.24}
\newtcolorbox{warning}{colback=white,
  colframe=deepcarmine!20,fonttitle=\bfseries,coltitle=black,
  title={Warning}}

\definecolor{forestgreenWeb}{rgb}{0.13, 0.55, 0.13}
\newtcolorbox{note}[1]{colback=white,
  colframe=forestgreenWeb!20,fonttitle=\bfseries,coltitle=black,
  title={#1}}

\newtcolorbox{formulation}[1]{colback=white,
  colframe=DarkOrchid!20,fonttitle=\bfseries,coltitle=black,
  title={#1}}

\definecolor{UOHblue}{cmyk}{1, 0.4, 0, 0}

\titleformat{\chapter}[display]{\normalfont\Large}
{\titlerule\vspace{10pt}\scshape\chaptername\
  \Huge\bfseries\thechapter}
{5pt}{\titlerule[1pt]\vspace{1pt}\titlerule\vspace{20pt}\Huge\bfseries}
[\vspace{30pt}]

\titleformat{\section}[runin]{\Large\bfseries}
{\thesection}{1em}{}[\hspace{1em}\hrulefill\\\medskip\\]

\titleformat{\subsection}{\large\bfseries} {\thesubsection}{1em}{}

\newtheorem{theorem}{Theorem}[section]
\newtheorem{definition}[theorem]{Definition} 
\newtheorem{lemma}[theorem]{Lemma}
\newtheorem{corollary}[theorem]{Corollary} 
\newtheorem{proposition}[theorem]{Proposition}

\theoremstyle{definition}
\newtheorem{example}[theorem]{Example}
\AtEndEnvironment{example}{\hfill$\Diamond$}
\newtheorem{problem}{Problem}[chapter]

\begin{document}
 

\def\codigo{Lecture Notes}
\def\ramo{Introduction to Bilevel Optimization: A perspective from Variational Analysis}
\def\autores{David Salas}
\def\semestre{Fall}
\def\year{2024}


\begin{titlepage}
  \drop=0.12\textheight
  \centering
  \vspace*{\drop}
  
  \scshape
  {\large Toulouse School of Economics\\ Master in Mathematics and Economic Decision}\\[\baselineskip]
  
  \rule{\textwidth}{1.6pt}\vspace*{-\baselineskip}\vspace*{2pt}
  \rule{\textwidth}{0.4pt}\\[\baselineskip]
  
  {\LARGE\bf \codigo \\[\baselineskip] \ramo}\\[0.2\baselineskip]

  \rule{\textwidth}{0.4pt}\vspace*{-\baselineskip}\vspace{3.2pt}
  \rule{\textwidth}{1.6pt}\\[\drop]

  {\Large\autores}\\[1.5\drop]
  \begin{minipage}{0.49\textwidth}
  \centering
  {\includegraphics[width=0.5\textwidth]{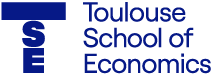}
  }
  \end{minipage}
  \begin{minipage}{0.49\textwidth}
  \centering
  {\includegraphics[width=0.6\textwidth]{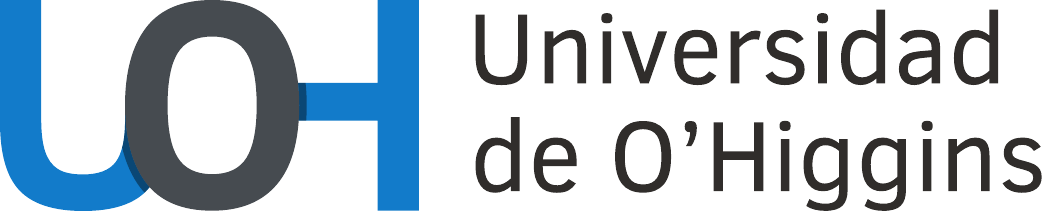}}
  \end{minipage}\\[\drop]
  {\large\semestre-\year}\par
  \null

\end{titlepage}


\pagestyle{plain}
\newpage
\cleardoublepage \vspace*{\fill}
\chapter*{Preface}
\parskip = 10pt
Game Theory can be described as the branch of mathematics that studies the interactions of multiple agents whose decisions affect each others. The field gain huge popularity after the groundbreaking contribution of John Nash \citep{Nash1950}. However, the study of economical equilibia between agents can be traced back to mid 19th century, with the duopoly model of Antoine Agustin Cournot \citep{cournot1838}. The model of Cournot considered two similar companies competing in a noncooperative market, and reaching an economical equilibrium due to rationality.

In 1934 (still 16 years prior to the contributions of Nash), Heinrich von Stackelberg proposed in \citep{Stackelberg1934} a variant of Cournot competition model, by considering that the two companies were different: one of them had much more market power than the other. The idea of Stackelberg was to assume that the bigger company, known as the \textit{leader}, knew somehow the decision problem of the smaller one, known as the follower, and therefore was capable of anticipate her reaction. The leader then takes this information into account, solving her decision problem by including the anticipated reaction given by the (parametric) decision problem of the follower.

This is the origin of bilevel programming: an optimization problem (the leader's problem) that takes into account the parametric solution of a surrogate optimization problem (the follower's problem). The subject resurfaced in the 70's and has gained popularity since, due to its numerous applications in game theory, economics, engineering, operations. Several monographs are already available  \citep{Dempe2002Foundations,DempeEtAl2015BilevelApplications,DempeZemkoho2020BilevelAdvances,AusselLalitha2017GNEPandBilevel,BeckSchmidt2023GentleIntro}, and the subject is an active field of research.

In this lecture notes, our aim is to give a quick introduction to the field of bilevel programming from the perspective of parametric optimization and variational analysis. In Chapter \ref{Chapter02:Model}, we will develop the "Cournot versus Stackelberg" discussion, present the model of bilevel programming and the fundamental case of linear bilevel programming, and review some modern applications.  Chapter \ref{Chapter03:Existence} is devoted to the fundamental problem of existence of solutions. We provide a quick overview of fundamentals in set-valued analysis and parametric optimization, and then present the main existence results for bilevel optimization. Chapter \ref{Chapter04:Algorithms} is devoted to the elementary methods to solve bilevel programming problems. Here we pay particular attention to the linear setting, providing hardness results and presenting some concrete algorithms to solve them. Finally, Chapter \ref{Chapter05:Extensions} explores some state-of-the-art extensions of the Stackelberg model. 

These lecture notes are by no means a comprehensive presentation of all the contributions in the field. It is rather, what I expect to be, a pedagogical introduction. Bilevel programming is at the crossroad of many fields, such as Variational Analysis, Mathematical Programming, Computers Science and Operations Research. The presentation here is strongly influenced by my background on Set-Valued Analysis and Nonconvex/Nonsmooth Optimization. For an alternative point of view, closer to Mathematical Programming and Operations Research, I strongly recommend to have a look to the lecture notes of Yasmine Beck and Martin Schmidt \citep{BeckSchmidt2023GentleIntro}, publicly available at \url{https://optimization-online.org/2021/06/8450/}.

As a disclaimer, these lecture notes might have many errors, and some notions such as duality and computational complexity are explained in a very superficial manner. If you as a reader spot any typos or conceptual mistakes, please do not hesitate in contact me. I would definitely like to improve this notes with the time.

\vfill
\begin{flushright}
\textit{David Salas Videla}\\
\textit{September 2024}
\end{flushright}
\vfill



\cleardoublepage
\parskip=10pt
\tableofcontents
\cleardoublepage

\parskip = 2pt
\chapter{The Model of Bilevel Optimization}
\label{Chapter02:Model}

\quad In this chapter we will present the model of Bilevel programming: an optimization problem that takes into account the solutions of another (parametric) surrogate optimization problem. This idea was first introduced by H. von Stackelberg in 1934 \citep*{Stackelberg1934}. The first comprenhensive monograph on the subject (to the best of my knowledge) is due to S. Dempe \citep*{Dempe2002Foundations}. Most of the elements of this chapter can be found there.

\section{Stackelberg Duopoly}
\label{ch02:sec:Stackelberg}

In 1938, A. A. Cournot introduced his celebrated model for economic equilibrium of two companies (duopoly), nowadays known as the Cournot competition model \citep*{cournot1838}. The problem that Cournot studied was as follows:

\begin{formulation}{Cournot Dupololy}
  Two companies $E_1$ and $E_2$ produce the same (fractional) good, at marginal cost $c>0$. Each company $E_i$ ($i=1,2$) must decide its production $q_i \in [0,+\infty)$. All the production $q=q_1+q_2$ is consumed, but the selling price varies depending on $q$. We consider that the price is given by $p(q) = p_0 -\alpha q$, with $p_0>c$. The companies aim to maximize their revenues. The final formulation is given by
  
  \begin{equation}\label{ch02:eq:CournotFormulation}
    \begin{array}{|c|c|}
        \hline
        &\\
        \begin{array}{cl}
            \displaystyle\max_{q_1}   & (p(q_1+q_2)-c)q_1\\
            \text{s.t.} & q_1\geq 0.
        \end{array}
        &
        \begin{array}{cl}
            \displaystyle\max_{q_2}   & (p(q_1+q_2)-c)q_2\\
            \text{s.t.} & q_2\geq 0.
        \end{array}\\
        
        &\\
        \hline
        \text{Company $E_1$}&\text{Company $E_2$}\\
        \hline
    \end{array}
\end{equation}
\end{formulation}

Note that in each of the problems of \eqref{ch02:eq:CournotFormulation}, the problem of the company is affected by the decision of the other one (as it is the trend in Game Theory). The fundamental idea of Cournot was to understand the problem of each company as a parametric optimization problem, parametrized by the decision of the other one. Then, one can compute what is nowadays known as the \textit{best response function}. For company $E_1$, the best responde is given by
\begin{equation}\label{ch02:eq:BestResponseCournot}
    B_1:q_2\in [0,+\infty)\mapsto \argmin_{q_1\geq 0}\{ -(p_0-\alpha(q_1+q_2) - c)q_1 \}. 
\end{equation}
Let us fix $q_2\geq 0$, and compute $B_1(q_2)$. To do so, we can consider the Lagrangian of Company $E_1$'s problem in \eqref{ch02:eq:CournotFormulation}, that is,
\[
\mathcal{L}(p,\mu) = -(p_0 - \alpha(q + q_2) - c)q - \mu q.
\]
Then, the first-order optimality conditions of the problem are given by
\begin{subequations}\label{ch02:eq:KKTCournotE1}
\begin{align}
    \frac{\partial\mathcal{L}}{\partial q} = -(p_0 - \alpha(q + q_2) - c) + \alpha q + \mu &= 0\label{ch02:eq:KKTCournotE1-dL}\\
    \mu&\geq 0\\
    q&\geq 0\\
    \mu q &= 0.
\end{align}
\end{subequations}
Then, solving $q$ in \eqref{ch02:eq:KKTCournotE1-dL} and applying the other constraints in \eqref{ch02:eq:KKTCournotE1}, we deduce that 
\[
q^* = \frac{p_0 - \alpha q_2-c +\mu}{2\alpha} = \begin{cases}
    0\quad&\text{ if } p_0 - \alpha q_2-c<0,\\
    \\
    \displaystyle\frac{p_0 -c- \alpha q_2}{2\alpha}\quad&\text{ otherwise.}
\end{cases}
\]
That is, $q=\max\left\{0,\frac{p_0 -c- \alpha q_2}{2\alpha}\right\}$. Since $E_1$'s problem in \eqref{ch02:eq:CournotFormulation} is concave in $q_1$, then $q_1 = q$ is its (unique) solution. Therefore,
\[
B_1(q_2) = \max\left\{0,\frac{p_0 - c-  \alpha q_2}{2\alpha}\right\}.
\]
By a symmetry argument, for a given $q_1\geq 0$, we can define the best response of company $E_2$ as
\[
B_2(q_1) = \max\left\{0,\frac{p_0 -c- \alpha q_1}{2\alpha}\right\}.
\]
The Cournot equilibrium is given by the pair $(q_1^*,q_2^*)\in\R_+^2$ such that each decision of one of the companies is the best response to the decision of the other one, that is $(q_1^*,q_2^*) = (B_1(q_2^*), B_2(q_1^*))$. 

Now, since $p_0>c$, one can easily verify that $(0,0)$ can not be a Cournot equilibrium. Let us suppose then that $q_1^*>0$. Then, $q_1^* = (p_0 - c - \alpha q_2^*)/2\alpha$ and we can write:
\begin{align*}
    q_2^* &= \max\left\{0, \frac{p_0 -c -\alpha\left(\frac{p_0 - c - \alpha q_2^*}{2\alpha}\right) }{2\alpha}\right\}\\
    &= \max\{ 0, \frac{p_0 - c}{4\alpha} + \frac{1}{4}q_2^* \}\\
    &= \frac{p_0 - c}{4\alpha} + \frac{1}{4}q_2^*.
\end{align*}
Thus, by solving $q_2^*$ in the above equation, we get that $q_2^* = \frac{p_0 - c}{3\alpha}$. Then, we obtain that  $q_1^* = \frac{p_0 - c}{3\alpha}$ as well. Noting that, by symmetry, assuming that $q_2^*>0$ would have led us to the same conclusion, we deduce that the (unique) Cournot equilibrium is given by
\begin{equation}\label{ch02:eq:SolutionCournot}
    (q_1^*,q_2^*) = \left( \frac{p_0 - c}{3\alpha},\frac{p_0 - c}{3\alpha}\right).
\end{equation}
The profit for each company is given by
\begin{equation}\label{ch02:eq:ProfitCournot}
    \text{Profit: }\,\left( p_0 - c - \alpha\left(\frac{2(p_0 - c)}{3\alpha}\right)\right)\frac{p_0 - c}{3\alpha} = \frac{(p_0 - c)^2}{9\alpha}
\end{equation}


In 1934, in his thesis dissertation \citep*{Stackelberg1934}, H. von Stackelberg proposed an alternative to Cournot equilibrium. The idea was to distinguish both companies assuming that one of them had knowledge about the other. On the one hand, Stackelberg considered that the company $E_1$, now the leader, new the problem of company $E_2$ and was able to consider the best response function $B_2(q_1)$ into account. On the other one, the company $E_2$, now the follower, would only be able to react to the decisions of $E_1$, solving the same problem as it did before. 

\begin{formulation}{Stackelberg Dupololy}
  Two companies $E_1$ and $E_2$ produce the same (fractional) good, at marginal cost $c>0$. Each company $E_i$ ($i=1,2$) must decide its production $q_i \in [0,+\infty)$. All the production $q=q_1+q_2$ is consumed, but the selling price varies depending on $q$. We consider that the price is given by $p(q) = p_0 -\alpha q$, with $p_0>c$. The companies aim to maximize their revenues. Company $E_1$ knows the best response function of company $E_2$ and takes this information into account. The final formulation is
  
  \begin{equation}\label{ch02:eq:StackelbergFormulation}
    \begin{array}{|c|c|}
        \hline
        &\\
        \begin{array}{cl}
            \displaystyle\max_{q_1}   & (p(q_1+B_2(q_1))-c)q_1\\
            \text{s.t.} & q_1\geq 0.
        \end{array}
        &
        \begin{array}{cl}
            \displaystyle\max_{q_2}   & (p(q_1+q_2)-c)q_2\\
            \text{s.t.} & q_2\geq 0.
        \end{array}\\
        
        &\\
        \hline
        \text{Company $E_1$ (Leader)}&\text{Company $E_2$ (Follower)}\\
        \hline
    \end{array}
\end{equation}
\end{formulation}
Note that in this second formulation, the problem of the leader does not really involve the follower's decision variable at all. Therefore, it can be solved directly. The objective function of the leader is given by
\begin{align*}
\theta(q_1) &= (p(q_1+B_2(q_1))-c)q_1\\
&= (p_0 - c -\alpha(q_1 + B_2(q_2)))q_1 = \left(p_0 - c -\alpha\left(q_1 + \max\left\{0,\frac{p_0-c-\alpha q_1}{2\alpha}\right\}\right)\right)q_1
\end{align*}
Note that, if $p_0-c-\alpha q_1 < 0$, then $\theta(q_1) < 0 = \theta(0)$. Thus, we can assume that $p_0-c-\alpha q_1 >0$. In such a case, we can write
\begin{align*}
   \theta(q_1) &=  \left(p_0 - c -\alpha\left(q_1 + \frac{p_0-c-\alpha q_1}{2\alpha}\right)\right)q_1\\
   &= \frac{p_0 - c}{2} q_1 -\frac{\alpha}{2} q_1^2.
\end{align*}
Thus, the objective function of the leader is concave. By further assuming that the solution $q_1^*>0$,  we can write Fermat's rule
\[
\frac{\partial \theta}{\partial q_1}(q_1^*) = 0 \implies q_1^* = \frac{p_0 - c}{2\alpha}.
\]
Since $q_1^* = \frac{p_0 - c}{2\alpha}$ verifies both assumptions, it is a critical point of the problem. Comparing the objective value at $q_1^*$ with the only other possible critical point $q_1 = 0$, we get
\[
\theta(q_1^*) = \frac{p_0 - c}{2}\left(\frac{p_0 - c}{2\alpha}\right) -\frac{\alpha}{2} \left(\frac{p_0 - c}{2\alpha}\right)^2 = \frac{(p_0-c)^2}{8\alpha} > 0 = \theta(0).
\]
Thus, $q_1^* = \frac{p_0 - c}{2\alpha}$ is the optimal decision of the leader. The follower then solves its problem by choosing
\[
q_2^* = B_2(q_1^*) = \max\left\{0,\frac{p_0-c-\alpha q_1}{2\alpha}\right\} = \frac{p_0 - c}{4\alpha}.
\]
That is, the equilibrium is now given by
\begin{equation}\label{ch02:eq:SolutionStackelberg}
    (q_1^*,q_2^*) = \left(\frac{p_0 - c}{2\alpha}, \frac{p_0 - c}{4\alpha}\right)
\end{equation}
Note that in this model the symmetry is no longer. Moreover, by evaluating the objective functions of each company, we get that
\begin{subequations}\label{ch02:eq:ProfitStackelberg}
 \begin{align}
    \text{Leader: }\left( p_0 - c - \alpha(q_1^* + q_2^*)\right)q_1^* &= \left( p_0 - c - \frac{3(p_0 - c)}{4}\right) \frac{p_0 -c}{2\alpha}= \frac{(p_0 -c)^2}{8\alpha}.\\
    \text{Follower: }\left( p_0 - c - \alpha(q_1^* + q_2^*)\right)q_1^* &= \left( p_0 - c - \frac{3(p_0 - c)}{4}\right) \frac{p_0 -c}{4\alpha}= \frac{(p_0 -c)^2}{16\alpha}.
\end{align}   
\end{subequations}

Thus, the model of Stackelberg gives a clear advantage for the leader. By accessing $B_2(q_1)$ the leader knows how her decisions will affect the follower, and she uses the information to deviate from Cournot equilibrium. In contrast, the follower's acts as if the situation were the same as in Cournot formulation: the decision of the leader is treated as exogenous information to which she can only reacts to.

\begin{table}[ht]  
    \centering
    \setlength{\extrarowheight}{6pt}
    \begin{tabular}{|c|c|c|}
    \hline
         & Company $E_1$ (Leader) & Company $E_2$ (Follower) \\[6pt]
         \hline
      Decision Cournot   & $\displaystyle\frac{p_0 - c}{3\alpha}$ & $\displaystyle\frac{p_0 - c}{3\alpha}$\\[10pt]
      Profit Cournot & $\displaystyle\frac{(p_0 - c)^2}{9\alpha}$ & $\displaystyle\frac{(p_0 - c)^2}{9\alpha}$\\[10pt]
      \hline
      Decision Stackelberg   & $\displaystyle\frac{p_0 - c}{2\alpha}$ & $\displaystyle\frac{p_0 - c}{4\alpha}$\\[10pt]
      Profit Stackelberg & $\displaystyle\frac{(p_0 - c)^2}{8\alpha}$ & $\displaystyle\frac{(p_0 - c)^2}{16\alpha}$\\[10pt]
      \hline
    \end{tabular}
    \caption{Comparison table: Cournot vs Stackelberg}
    \label{ch02:table:ComparisonCournotStackelberg}
\end{table}


\section{The model: Optimistic and Pessimistic formulations}
\label{ch02:sec:Formulation}

Motivated by the work of Stackelberg, people in Game Theory and Optimization started to look into equilibrium problems of two agents enjoying the hierarchical structure of a leader and a follower.

\begin{definition}[Bilevel Programming problem]\label{ch02:def:BilevelProgramming-IllPosed}
    Two agents, the leader and the follower, are considered. The leader decides $x\in X\subset \R^p$. For each leader's decision $x\in X$, the followers solve a parametric optimization problem, deciding $y\in K(x)\subset\R^q$. The leader and the follower aim to minimize the functions $\theta_l,\theta_f:\R^p \times \R^q \to \R$, respectively. The final formulation is given by
    \begin{equation}\label{ch02:eq:BilevelProgramming-IllPosed}
        \begin{array}{rl}
           \displaystyle\min_{x\in\R^p}  & \theta_l(x,y)  \\
            s.t. & \left\{\begin{array}{l}
                 x\in X,  \\
                 y \text{ solves } \left\{\begin{array}{rl}
           \displaystyle\min_{y\in\R^q}  & \theta_f(x,y)  \\
            s.t. & y\in K(x).
        \end{array}\right.
            \end{array}\right.
        \end{array}
    \end{equation}
A leader's decision is said to be \textbf{feasible} if $x\in X$ and the follower's problem at $x$ admits at least one solution.    
\end{definition}

\begin{warning}
    In many books and papers, the feasibility requirement of $S(x)\neq\emptyset$ is usually not written as such and it is assumed implicitly.
\end{warning}

Bilevel Programming problems, or Bilevel Optimization problems, are also known as Stackelberg games. The leader's problem is also known as the upper-level problem, and the follower's problem is also known as the lower-level problem. Hence, the name of ``Bilevel'' Programming. If, for every $x\in X$, we denote by $S(x)$ the set of solutions of the lower-level problem, we can rewrite formulation \eqref{ch02:eq:BilevelProgramming-IllPosed} as follows:

  \begin{equation}\label{ch02:eq:Stackelberg-IllPosed}
    \begin{array}{|c|c|}
        \hline
        &\\
        \begin{array}{cl}
            \displaystyle\min_{x\in \R^p}   & \theta_l(x,y)\\
            \text{s.t.} & x\in X,\, y\in S(x).
        \end{array}
        &
        \begin{array}{cl}
            \displaystyle\min_{y\in \R^q}   & \theta_l(x,y)\\
            \text{s.t.} & y\in K(x).
        \end{array}\\
        
        &\\
        \hline
        \text{Upper-level (Leader)}&\text{Lower-level (Follower)}\\
        \hline
    \end{array}
\end{equation}

The set $S(x)$ is often called the \emph{reaction set}  or the \emph{best response set} of the follower. In this formulation, we see the hierarchical structure only on the upper-level, through the constraint $y\in S(x)$. That is, the leader knows that for any given $x\in X$, the follower will react by choosing $y\in S(x)$. If such constraint would be missing, then we would be facing a classic (non-hierarchical) equilibrium problem.

One can think on three straight-forward interpretations of a Bilevel Programming problem:
\begin{enumerate}
    \item \textbf{Hierarchical game}: The leader and the follower are solving an equilibrium problem, where the leader is able to anticipate the follower's reaction. This is consistent with formulation \eqref{ch02:eq:Stackelberg-IllPosed} and was the original interpretation of Stackelberg.
    \item \textbf{Sequential game}: The leader decides first, and the follower reacts afterwards. By knowing the follower's problem, the leader anticipates its decision when solving the upper-level problem. This interpretation can be linked to Extensive-Forms Nash Games and subgame perfect equilibria (see, \textit{e.g.}, \citep*{Tadelis2013GameTheory} for the definitions). This interpretation can be read from, \textit{e.g.}, \citep*{CarusoLignolaMorgan2020Regularization}.
    \item \textbf{Leader's problem only}: One can imagine that bilevel programming is not really a game, but only the leader's problem, who has a \textit{model} of the follower at hand. That is, the lower-level is the model that allows the leader to estimate the follower's reaction $y\in S(x)$, which is then used to solve the optimization problem, that is, the upper level. Following this interpretation, we might also write problem \eqref{ch02:eq:BilevelProgramming-IllPosed} simply as
    \begin{equation}\label{ch02:eq:BilevelProgramming-Leader-IllPosed}
        \begin{array}{cl}
            \displaystyle\min_{x\in \R^p}   & \theta_l(x,y)\\
            \text{s.t.} & x\in X,\, y\in S(x),
        \end{array}
    \end{equation}
   defining $S(x)$ separately. 
\end{enumerate}

Regardless the interpretation, Problem \eqref{ch02:eq:BilevelProgramming-IllPosed} (and its other formulations \eqref{ch02:eq:Stackelberg-IllPosed} and \eqref{ch02:eq:BilevelProgramming-Leader-IllPosed}) is \textit{ill-posed}. Indeed, unless the solution set $S(x)$ is single-valued, there exist an ambiguity on how the follower's reaction is determined.
\begin{example}\label{ch02:example:MultipleSolutionsFollower}Consider the problem
\begin{equation*}
        \begin{array}{rl}
           \displaystyle\min_{x\in\R}  & |x|+y  \\
            s.t. &
                 y \text{ solves } \left\{\begin{array}{rl}
           \displaystyle\min_{y\in\R}  & -xy  \\
            s.t. & 0\leq y \leq 1.
        \end{array}\right.
            
        \end{array}
    \end{equation*}
One might try to reason as Stackelberg, and define the Cournot best response $y:x\in\R\mapsto y(x)\in\R$. However, in this problem, for $x\in \R$ we have that
\[
S(x) = \left\{\begin{array}{cl}
     \{ 0\} \quad&\text{ if }x<0,  \\
     {[0,1]}\quad &\text{ if } x= 0,\\
     \{1\}\quad &\text{ if }x>0.
\end{array}\right.
\]
Then, the objective function of the leader is given by 
\[
\theta(x) = x^2+y(x) = \left\{\begin{array}{cl}
     -x \quad&\text{ if }x<0,  \\
     {[0,1]}\quad &\text{ if } x= 0,\\
     x + 1\quad &\text{ if }x>0.
\end{array}\right.
\]
\begin{figure}[ht]
    \centering
    \includegraphics[width=0.9\linewidth]{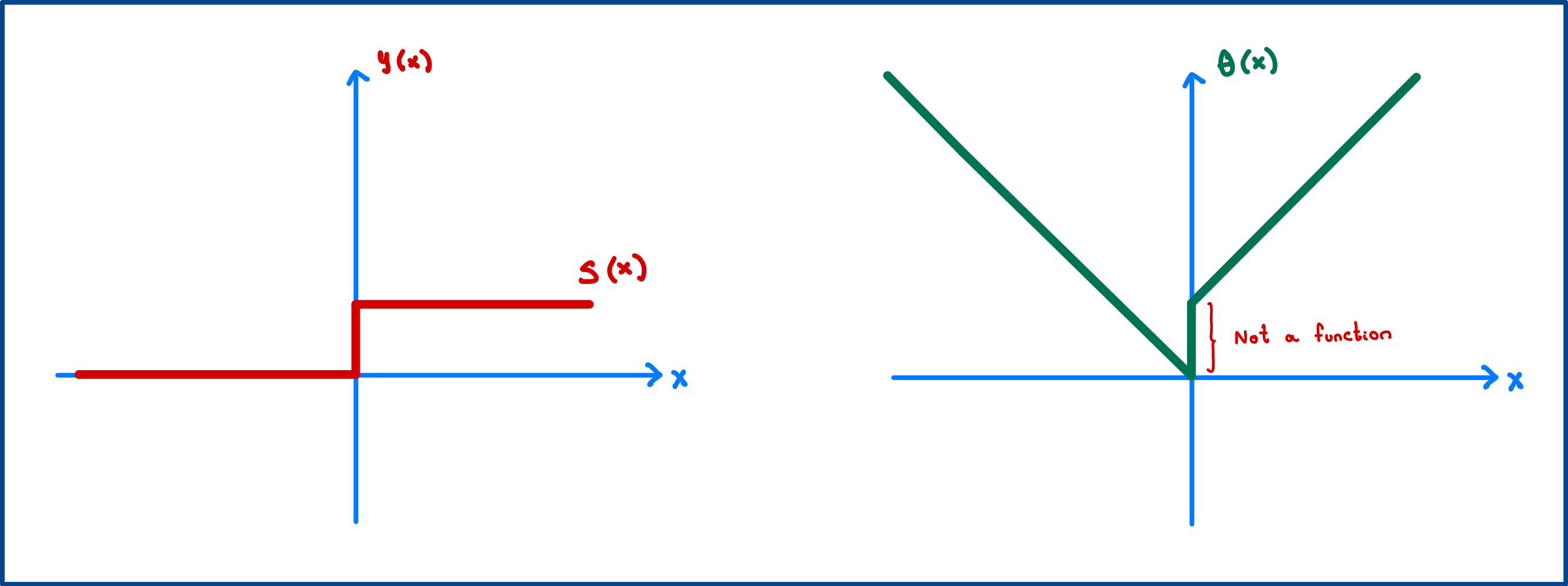}
    \caption{Graphs of $S(x)$ and $\theta(x)$, which are not functions.}
    \label{ch02:fig:example-MultipleSolutionsFollower}
\end{figure}
The situation is illustrated in Figure  \ref{ch02:fig:example-MultipleSolutionsFollower}. As one can see, this objective $\theta(\cdot)$ \textit{is not} a function. Given the information we have been given, it is not possible to determine the value of $\theta(0)$.
\end{example}

To solve this ambiguity, some extra assumptions are needed to decide how $y\in S(x)$ is selected. While many approaches have been proposed in the literature, in this notes we will revise the two most relevant in the literature: the \textit{optimistic} and \textit{pessimistic} approaches.

\paragraph{Optimistic Approach.} The idea of the optimistic approach is to think that the follower is in some sense \textit{collaborative}. Thus, among all her optimal reactions $y\in S(x)$ she will select the one that \textit{favors the leader the most}. 

\begin{definition}[Optimistic Bilevel Programming problem]\label{ch02:def:BilevelProgramming-Optimistic}
    Under the same setting as in Definition \ref{ch02:def:BilevelProgramming-IllPosed} the optimistic formulation is given by
    \begin{equation}\label{ch02:eq:BilevelProgramming-Optimistic}
        \begin{array}{rl}
           \displaystyle\min_{x\in\R^p, y \in \R^q}  & \theta_l(x,y)  \\
            s.t. & \left\{\begin{array}{l}
                 x\in X,  \\
                 y \text{ solves } \left\{\begin{array}{rl}
           \displaystyle\min_{y\in\R^q}  & \theta_f(x,y)  \\
            s.t. & y\in K(x).
        \end{array}\right.
            \end{array}\right.
        \end{array}
    \end{equation}
By defining $S(x) = \argmin_y\{ \theta_f(x,y)\, :\, y\in K(x)  \}$, we can also rewrite the optimistic formulation as
\begin{equation}\label{ch02:eq:BilevelProgramming-Leader-Optimistic}
        \begin{array}{cl}
            \displaystyle\min_{x\in \R^p,y\in\R^q}   & \theta_l(x,y)\\
            \text{s.t.} & x\in X,\, y\in S(x).
        \end{array}
    \end{equation}
\end{definition}

The difference between \eqref{ch02:eq:BilevelProgramming-IllPosed} and \eqref{ch02:eq:BilevelProgramming-Optimistic}  (as between \eqref{ch02:eq:BilevelProgramming-Leader-IllPosed} and \eqref{ch02:eq:BilevelProgramming-Leader-Optimistic}) is subtle: they are almost the same, except that now $y$ is also a decision variable of the leader. This means that the leader is allow to \textit{choose} the follower's reaction but only among those that are optimal, that is, respecting the constraint $y\in S(x)$.

One can also define the function $\varphi^o: X\to \R\cup\{+\infty\}$ given by $\varphi^o(x) = \inf_{y\in S(x)} \theta_l(x,y)$. Then, the optimistic bilevel problem \eqref{ch02:eq:BilevelProgramming-Optimistic} can be further written as
\begin{equation}\label{ch02:eq:BilevelProgramming-Optimistic-phiO}
    \begin{array}{cl}
            \displaystyle\min_{x\in \R^p}   & \varphi^o(x)\\
            \text{s.t.} & x\in X.
        \end{array}
\end{equation}
This last formulation reinforces the interpretation that bilevel programming is in fact the problem of the leader.

\paragraph{Pessimistic Approach.} The idea of the pessimistic approach is exactly the opposite of the optimistic one, that is, to think that the follower is \textit{adversarial}. Thus, among all her optimal reactions $y\in S(x)$ she will select the one that \textit{harms the leader the most}. The formulation is then a min-max type problem. To formulate it, we will proceed as in \eqref{ch02:eq:BilevelProgramming-Optimistic-phiO}.
\begin{note}{Notation: $\dom S$}
The optimistic approaches encompasses implicitly that a feasible leader decision $x\in X$ needs to verify that the follower's problem admits a solution. Indeed, if $S(x) = \emptyset$, then $\varphi^o(x) = +\infty$ due to the convention that $\inf\emptyset = +\infty$. However, this restriction needs to be explicit in the pessimistic formulation. So, we define the set
\begin{equation}\label{ch02:eq:domS-preliminaryDef}
    \dom S = \{ x\in \R^p\, :\, S(x)\neq \emptyset \}.
\end{equation}
As we will see in Chapter \ref{Chapter03:Existence}, this notation is given by the notion of domain of set-valued mappings.
\end{note}
\begin{definition}[Pessimistic Bilevel Programming problem] \label{ch02:def:BilevelProgramming-Pessimistic}
    Under the same setting as in Definition \ref{ch02:def:BilevelProgramming-IllPosed}, and defining $S(x) = \argmin_y\{ \theta_f(x,y)\, :\, y\in K(x)  \}$, the pessimistic formulation is given by
    \begin{equation}\label{ch02:eq:BilevelProgramming-Pessimistic}
        \begin{array}{cl}
           \displaystyle\min_{x\in\R^p, y \in \R^q}  & \sup_y\{\theta_l(x,y)\, :\, y\in S(x)\}  \\
            s.t. & x\in X\cap \dom S.
        \end{array}
    \end{equation}
By defining  the function $\varphi^p:x\in X\mapsto \sup\{ \theta_l(x,y)\ :\ y\in S(x)\}$, we can also rewrite the optimistic formulation as
\begin{equation}\label{ch02:eq:BilevelProgramming-Pessimistic-phip}
        \begin{array}{cl}
            \displaystyle\min_{x\in \R^p}   & \varphi^p(x)\\
            \text{s.t.} & x\in X\cap \dom S.
        \end{array}
    \end{equation}
In both formulations, a decision $x\in \R^p$ is feasible if and only if $x\in X$ and $S(x)\neq\emptyset$.    
\end{definition}

As wee will see through this lecture notes, the pessimistic approach is further more challenging that its optimistic counterpart. A first glance of this arises already in Example \ref{ch02:example:MultipleSolutionsFollower}. Indeed, one can easily see that for the optimistic formulation, the optimal solution is given by $x^{\ast}=0$, while the pessimistic formulation does not admit exact solutions. This is illustrated in Figure \ref{ch02:fig:example-MultipleSolutionsFollower-phiO-and-phiP}.
\begin{figure}[ht]
    \centering
    \includegraphics[width=0.9\linewidth]{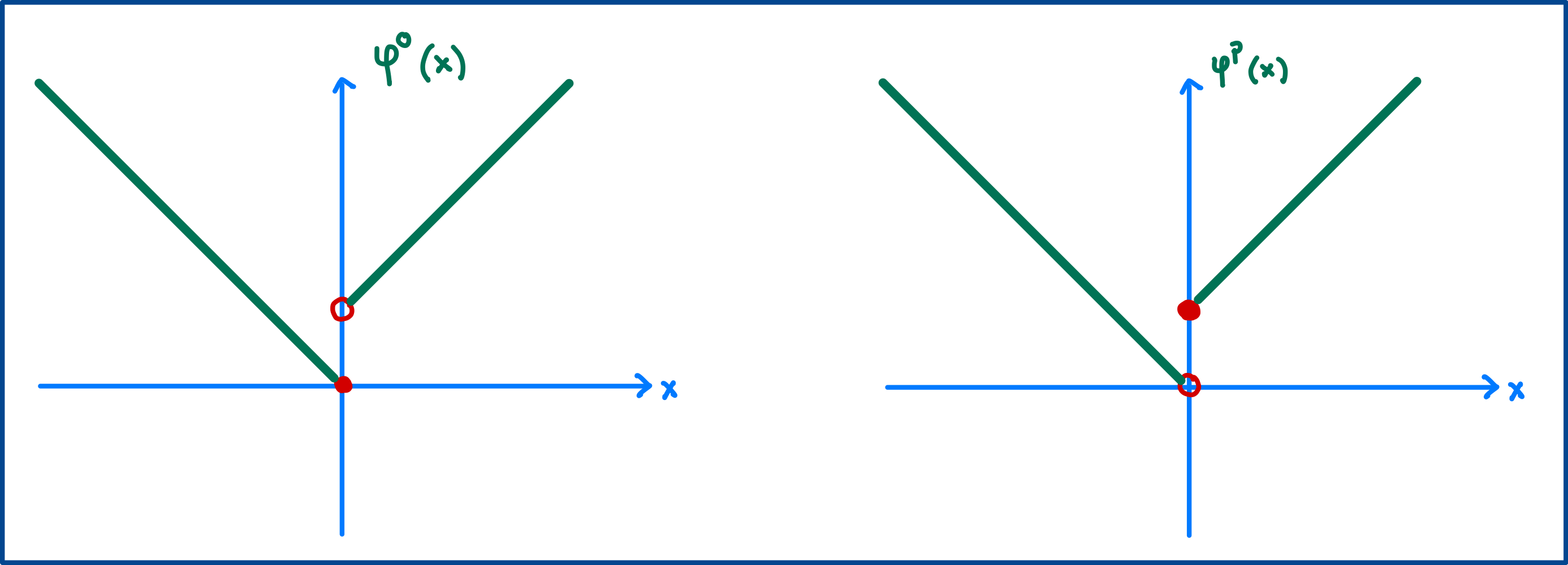}
    \caption{Function $\varphi^o$ (right) and $\varphi^p$ (left) for the problem of Example \ref{ch02:example:MultipleSolutionsFollower}. }
    \label{ch02:fig:example-MultipleSolutionsFollower-phiO-and-phiP}
\end{figure}


\section{Linear bilevel problems}
\label{ch02:sec:Linear}

A fundamental setting for bilevel programming are the \emph{linear bilevel problems}. This family of problems are those for which the data is fully linear, that is:
\begin{enumerate}
    \item $\theta_l(x,y) = c_l^{\top}x + d_l^{\top}y$ with $c_l\in \R^p$ and $d_l\in\R^q$.
    \item $X = \{x\in\R^p\, :\, A_lx \leq b_l \}$, for some $A_l\in\R^{m_l\times p}$ and $b_l\in \R^{m_l}$.
    \item $\theta_f(x,y) = c_f^{\top}y$ with $c_f\in \R^q$.
    \item $K(x) = \{y\in\R^q\, :\, A_fx + B_fy \leq b_f \}$, for some $A_f\in\R^{m_f\times p}$, $A_f\in\R^{m_f\times q}$ and $b_f\in \R^{m_f}$.
\end{enumerate}
With these considerations, one can have both, optimistic and pessimistic formulations fully determined.
\begin{formulation}{Optimistic Linear Bilevel programming}
    For vectors $c_l,d_l, c_f, b_l, b_f$ and matrices $A_l, A_f$ and $B_f$ as above, the optimistic linear bilevel problem is given by
    \begin{equation}\label{ch02:eq:BLP-Optimistic}
        \begin{array}{rl}
           \displaystyle\min_{x\in\R^p, y \in \R^q}  & c_l^{\top}x + d_l^{\top}y  \\
            s.t. & \left\{\begin{array}{l}
                 A_lx\leq b_l,  \\
                 y \text{ solves } \left\{\begin{array}{rl}
           \displaystyle\min_{y\in\R^q}  & c_f^{\top}y \\
            s.t. & A_fx + B_fy \leq b_f.
        \end{array}\right.
            \end{array}\right.
        \end{array}
    \end{equation}    
\end{formulation}

\begin{formulation}{Pessimistic Linear Bilevel programming}
    For vectors $c_l,d_l, c_f, b_l, b_f$ and matrices $A_l, A_f$ and $B_f$ as above, the pessimistic linear bilevel problem is given by
    
    \begin{equation}\label{ch02:eq:BLP-Pessimistic}
        \begin{array}{cl}
           \displaystyle\min_{x\in\R^p, y \in \R^q}  & c_l^{\top}x + \sup_y\{d_l^{\top}y\, :\, y\in S(x)\}  \\
            s.t. & A_lx\leq b_l,\, x \in \dom S,
        \end{array}
    \end{equation}  
    
where $S(x) = \argmin_y\{c_f^{\top}y\, :\, A_fx + B_fy \leq b_f \}$.    
\end{formulation}

Through these notes, we will study these two problems in depth. Linear bilevel programming is not only as an example of the theory, but it is very relevant on its own. We will optain further results and algorithms by leveraging on the linear structure we just presented.

\begin{warning}
    It very important for a linear bilevel problem that the objective function of the follower to be \textbf{independent of the leader's decision}. 
    Namely, a function of the form $y\mapsto \alpha(x)^{\top}y$ might be considered linear in terms of $y$, but it does not induced a valid follower's linear objective function unless $\alpha(\cdot)$ is constant.
\end{warning}

\begin{example}\label{ch02:example:NonlinearThatLooksLinear}Consider the problem of example \ref{ch02:example:MultipleSolutionsFollower}, that is,
\begin{equation*}
        \begin{array}{rl}
           \displaystyle\min_{x\in\R}  & |x|+y  \\
            s.t. &
                 y \text{ solves } \left\{\begin{array}{rl}
           \displaystyle\min_{y\in\R}  & -xy  \\
            s.t. & 0\leq y \leq 1.
        \end{array}\right.
            
        \end{array}
    \end{equation*}
It is possible to reformulate the problem so it looks linear by means of \textit{epigraphical reformulation}:  
\begin{equation*}
        \begin{array}{rl}
           \displaystyle\min_{x,t\in\R}  & t+y  \\
            s.t. & \begin{cases}t-x\geq 0, t+x\geq 0,\\
                 y \text{ solves } \left\{\begin{array}{rl}
           \displaystyle\min_{y\in\R}  & -xy  \\
            s.t. & 0\leq y \leq 1.
        \end{array}\right.
        \end{cases}  
        \end{array}
    \end{equation*}
This second formulation looks linear: the leader's objective function is linear, all constraints are linear, and the objective function of the follower is linear with respect to the follower's variable. However, as we saw in Figure \ref{ch02:fig:example-MultipleSolutionsFollower-phiO-and-phiP}, the pessimistic approach of this problem does not admit solutions. As we will see later on, (fully) linear bilevel problems always admit solutions. 
\end{example}

\section{Some Applications}
\label{ch02:sec:Applications}

To finish this chapter, we will visit three examples of applied bilevel programming problems, classic in the literature. As a disclaimer, the number of applications is enormous and even for the problems presented here, there is a vast literature around them. Thus, this last section is nothing more than an illustration of what we can do with this framework.

\subsection{Toll installation problem}

Here, we present a simplified version of the multicommodity problem studied in the seminal work \citep*{Brotcorne2001Toll}. A traffic network can be modeled as a directed graph $G=(V,\vec{E})$, where 
\begin{itemize}
    \item the nodes $V=\{v_1,\ldots,v_n\}$ are intersections of streets, and
    \item an edge $(v_i,v_j)\in \vec{E}$ represents the street that goes from $v_i$ to $v_j$ with that traffic direction.
\end{itemize}
For this graph, we mark two special nodes $s,t\in V$, which are the source and the target, and we set a demand $d\in\R_+$ of (fractional) drivers that aim to go from $s$ to $t$ using this network. By supposing that each street $e\in \vec{E}$ has a marginal cost $c_e>0$ of being used, the classic  \textbf{minimal-cost flow problem} is given by
\begin{equation}\label{ch02:eq:NonatomicCongestionGame}
    \begin{array}{cl}
      \displaystyle \min_{x}  &\sum_{e\in \vec{E}} c_ex_e  \\
        s.t. & \left\{\begin{array}{l}
        x\geq 0,\\
        \forall v\in V, \, \displaystyle\sum_{e=(u,v)\in\vec{E}}x_e - \sum_{e=(v,w)\in\vec{E}}x_e  = \begin{cases}
            d\quad&\text{ if }v = s,\\
            0\quad&\text{ if }v \in V\setminus\{s,t\}\\
            -d&\text{ if }v = t.
        \end{cases}
        \end{array}\right.
    \end{array}
\end{equation}
This problem models the equilibrium of a ``volume'' $d$ of homogeneous drivers optimizing their routes from their starting point ($s$) to their destination ($t$). 
\begin{formulation}{Toll installation problem}
    For a traffic network $G=(V,\vec{E})$, a regulator aims decide the toll pricing $p = (p_e\, :\, e\in\vec{E})$ in order to solve
    \begin{equation}\label{ch02:eq:Toll-problem}
        \begin{array}{cl}
            \displaystyle\min_{p}   & \theta(p,x)\\
            \text{s.t.} & 0\leq p\leq p_{\max},\, x\in S(c+p),
        \end{array}
    \end{equation}
    where, $S(c+p)$ is the solution set of the nonatomic congestion game \eqref{ch02:eq:NonatomicCongestionGame} with augmented marginal costs $(c_e+p_e\, :\, e\in \vec{E})$ instead of $(c_e\, :\, e\in \vec{E})$. 
\end{formulation}
Note that the Toll installation problem here is given in the ill-posed formulation. Thus, one can consider the optimistic or pessimistic approaches depending on the situation that is being modeled.
\subsection{Security games}
Stackelberg Security games (SSG) were introduced in the mid 2000's as model to deal with the allocation of limited security resources (see, e.g., \cite{SinhaEtAl2018SSG}). 

In the most basic model, a defender must protect $K\in\N$ objects against an attacker. The defender can protect only one object at per day, and similarly the attacker can attack only one object (no simultaneous attacks are considered). Each object $k\in [K]$ has a value $a_k$. 

We denote by $\Delta(K)$ the $K$-simplex set. A vector on $\Delta(K)$ represents a probability distribution of choosing an object over $K$. The defender decides a defending policy $p\in \Delta(K)$, and the attacker decides an attacking policy $q\in \Delta(K)$. 

In the day of the attack, we distinguish two outcomes:
\begin{itemize}
    \item The attacker chooses to attack the object $k\in K$ protected by the defender. Then, the payment for the defender is $a_k$ and for the attacker $-a_k$.
    \item The attacker chooses to attack the object $k\in K$ not protected by the defender. Then, the payment for the defender is $-a_k$ and for the attacker $a_k$.
\end{itemize}

We assume that the attacker can observe for many days the defender's policy before choosing the day of the attack. Thus, the game becomes a Stackelberg game were the leader (defender) commits to a defending policy, and the follower (attacker) decides the attacking policy optimally.

\begin{formulation}{Stackelberg Security game}
For a logical proposition $P$, set $\ind_P = 1$ if $P$ is true, and $\ind_P=0$, if $P$ is false. Considering the description above, the formulation is as follows:

  \begin{equation}\label{ch02:eq:SSG-IllPosed}
    \begin{array}{|c|c|}
        \hline
        &\\
        \begin{array}{cl}
            \displaystyle\max_{p}   & \displaystyle\sum_{i\in K}\sum_{j\in K} a_j(\ind_{\{i=j\}} - \ind_{\{i\neq j\}})p_iq_j\\[15pt]
            \text{s.t.} & \begin{cases}
                p\geq 0,\\
                \sum_{i\in K}p_i = 1, q\in S(p).
            \end{cases}
        \end{array}
        &
        \begin{array}{cl}
            \displaystyle\max_{q}   & \displaystyle\sum_{i\in K}\sum_{j\in K} a_j(\ind_{\{i\neq j\}}-\ind_{\{i=j\}})p_iq_j\\[15pt]
            \text{s.t.} & \begin{cases}
                q\geq 0,\\
                \sum_{j\in K}q_j = 1.
            \end{cases}
        \end{array}\\
        
        &\\
        \hline
        \text{Upper-level (Defender)}&\text{Lower-level (Attacker)}\\
        \hline
    \end{array}
\end{equation}
    Since this is a zero-sum game, only the pessimistic formulation is meaningful, which is given by
    \begin{equation}\label{ch02:eq:SSG-Pessimistic}
        \max_{p\in \Delta(K)}\min_{q\in S(p)} \displaystyle\sum_{i\in K}\sum_{j\in K} a_j(\ind_{\{i\neq j\}}-\ind_{\{i=j\}})p_iq_j.
    \end{equation} 
\end{formulation}

Note that Stackelberg Security Games are in fact games. Therefore, the first formulation \eqref{ch02:eq:SSG-IllPosed} is given in the ``game-oriented'' form. 

\subsection{Dispatch problem in electricity markets}

Bilevel programming was firstly successfully applied to electricity Market models in the seminal paper \citep*{Hobbs2000Strategic}. While game theoretical aspects of electricity markets had already been considered, the application of a bilevel model allowed to capture the effects of a very particular agent: the Independent System Operator (ISO). We present here a simplified version of the model of \citep*{HuRalph2007EPEC}.

\begin{note}{\textbf{Notation:} Number sets}
    We consider the following notation:
    \begin{itemize}
        \item We denote by $\R_+$ the set of nonnegative real numbers, that is, $\R_+=[0,+\infty)$.
        \item We consider the $\N$ to be the positive integer numbers, that is, $\N = \{1,2,3,\ldots\}$.
        \item For an integer $N\in \N$, we write $[N] := \{1,\ldots,N\}$.
    \end{itemize}
    
\end{note}

We consider $N$ power producers that, at a given time, must satisfy a known demand $D$. Each producer $i\in [N]$ has an increasing cost function of power production $c_i:\R_+\to\R_+$ with $c_i(0) = 0$.

A regulator, called the independent system operator (ISO), needs to solve how much production to assign to each producer, in order to minimize the total cost of production and satisfy the demand.

Each producer $i\in [N]$ declares a cost function (that might differ from the real one), by giving parameters $a_i\in [a_i^l,a_i^u]$ and $b_i \in [b_i^l,b_i^u]$ to the ISO. The ISO then solves her problem, by considering the cost of the producer as
\begin{align*}
\tilde{c}_i:\R_+&\to\R\\
q_i &\mapsto b_iq_i^2 + a_iq_i.
\end{align*}
Finally, once the ISO assign a dispatch $q = (q_i\, :\, i\in [N])$, each producer is paid according its declared cost function, that is, $\tilde{c}_i(q_i)$.
\begin{note}{\textbf{Notation:} Game theory indexing}
    For a vector $x\in \R^N$ and an index $i\in [N]$, it is usual to write 
    \[
    x= (x_i,x_{-i}),
    \]
    to emphasize the $i$th coordinate. In this notation:
    \begin{itemize}
        \item $x_{-i}$ stands for the coordinates in $q$ but the $i$th one. That is, $x_{-i}:= (x_j\, :\, j\in [N]\setminus\{i\})$.
        \item the pair $(x_i,x_{-i})$ is an abuse of notation, since it could be interpreted as a reordering of coordinates. However, no order is assumed in the assigment $x=(x_i,x{-i})$.
    \end{itemize}
    In game theory, this notation is useful to separate on each agent's problem the effects due to their own decisions, $x_i$, and the effects due the decisions of the other ones, $x_{-i}$. 
\end{note}

Once every producer has submit her cost function's parameters $(a_i,b_i)$, the ISO solves the problem
\begin{equation}\label{ch02:eq:ISO-problem}
   \mathrm{ISO}(a,b)= \left\{\begin{array}{cl}
         \displaystyle\min_{q}    & \displaystyle\sum_{j=1}^N  b_jq_j^2+ a_{j}q_j\\[15pt]
             s.t. &  \begin{cases}
             \sum_{j=1}^N q_j\geq D,\\
             q\geq 0.
             \end{cases}
        \end{array}\right.
\end{equation}
The dispatch $q\in \R_+^N$ is the production assigment for each producer.
\begin{formulation}{Uninodal model for electricity market}
    With all the considerations above, the problem of each producer $i\in [N]$ is given by
    \begin{equation}\label{ch02:eq:UninodalSPOT}
    P_i(a_{-i},b_{-i}) = \left\{\begin{array}{cl}
      \displaystyle \min_{a_i,b_i}  & b_iq_i^2+ a_{i}q_i - c_i(q_i) \\
        s.t. & \left\{\begin{array}{l}
        a_i \in [a_i^l,a_i^u],\,\, b_i \in [b_i^l,b_i^u],\\[10pt]
        q =(q_i,q_{-i})\text{ solves }\left\{\begin{array}{cl}
         \displaystyle\min_{q}    & \displaystyle\sum_{j=1}^N  b_jq_j^2+ a_{j}q_j\\[15pt]
             s.t. &  \begin{cases}
             \sum_{j=1}^N q_j\geq D,\\
             q\geq 0.
             \end{cases}
        \end{array}\right.
        \end{array}\right.
    \end{array}\right.
\end{equation}
\end{formulation}

For a given decision vector $(a_{-i},b_{-i})$, the $i$th producer is playing a Stackelberg game with the ISO. Since the actions of producers affect each other (the ISO takes all the cost functions into account to decide the dispatch), producers are playing Nash game. These kind of models are known as Multi-Leader-Single-Follower games, and are an extension of Stackelberg games (one leader and one follower). We will revisit them in Chapter \ref{Chapter05:Extensions}.

\newpage
\section{Problems}\label{ch02:sec:problems}

\begin{problem}[bid b-only electricity market] Consider the uninodal model for electricity market described in \eqref{ch02:eq:UninodalSPOT}, with the following parameters:
\begin{itemize}
    \item $N= 2$, that is, there is only two producers.
    \item For each producer $i=1,2$, $c_i(q_i) = Bq_i^2$ for some shared constant $B>0$.
    \item Both players only bid the $b_i$ parameter, that is, they always set $a_i = 0$.
    \item Suppose that the second producer is honest, that is, she bids $b_2 = B$. 
\end{itemize}
With this considerations, we want to study the induced problem for the first producer.
\begin{enumerate}
    \item Write the bilevel problem of the first producer (considering the ISO as the follower). That is, write problem $P_1(b_2)$ as in \eqref{ch02:eq:UninodalSPOT}, with $b_2=B$.
    \item For each $b_1\in [b_1^l,b_1^u]$, write the (unique) solution of the ISO problem (assume that $b_1^l>0$).
    \item Replacing $q_1$ by $q_1(b_1)$ given by (first component) of the solution of the ISO, identify all critical points for the problem of the first producer.
    \item Determine the optimal solution $b_1^*$ for the case of $B=1$, $b_1^l = 0.4$ and $b_1^u = 1.5$.
\end{enumerate}
    
\end{problem}
\chapter{Existence of Solutions}
\label{Chapter03:Existence}

In this chapter, our main goal is to provide sufficient conditions for existence of solutions of bilevel programming problems. We will do so by studying the continuity properties of the best response set-valued map $x\mapsto S(x)$.

\section{Preliminaries on Set-Valued Analysis}\label{ch03:sec:Set-Valued}
\begin{definition}[Set-Valued map]\label{ch03:def:Set-ValuedMap}
Let $X$ and $Y$ be two nonsempty sets. A \textbf{set-valued map} (also known as \textbf{multifunction} or \textbf{correspondence}) $M$ from $X$ to $Y$, which we denote by $M:X\tto Y$, is an assignment that to each element $x\in X$, it assigns a set $M(x)\subset Y$. 
\end{definition}

\begin{example}\label{ch03:ex:SetValuedMaps}Some examples of set-valued maps are:
\begin{enumerate}
    \item $M:\R\tto \R$, given by $M(x) = \{y \,:\, y^2 = x\}$. This set-valued map gives the solutions the polynomial equation $y^2 = x$. It takes empty values for every $x<0$. See Figure~\ref{ch03:fig:Examples-setvalued-maps}.
    \item For a function $f:X\to \R$, we can consider the set-valued map $F:X\tto \R$ given by $F(x) = \{r~\,:~\,~r\geq f(x)\}$. See Figure~\ref{ch03:fig:Examples-setvalued-maps}.
    \item For a function $g:X\times Y\to \R^m$, we can consider the set-valued map of sublevel sets $L:X\tto Y$ given by $L(x) = \{y\,:\, g(x,y)\leq 0\}$, where the inequality $f(x,y)\leq 0$ is meant to be componentwise.
\end{enumerate}

\begin{figure}[ht]
    \centering
    \includegraphics[width=0.9\linewidth]{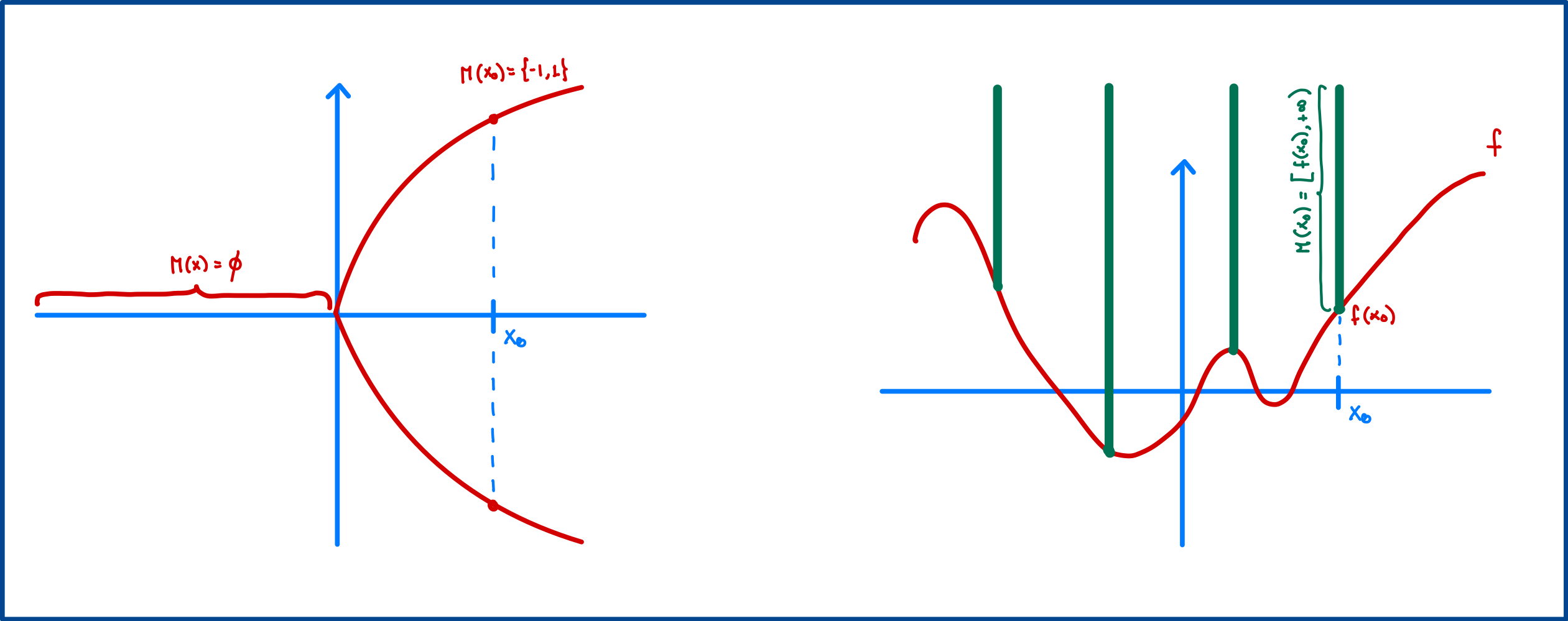}
    \caption{Illustration of Examples \ref{ch03:ex:SetValuedMaps}.1 (left) and \ref{ch03:ex:SetValuedMaps}.2 (right).}
    \label{ch03:fig:Examples-setvalued-maps}
\end{figure}

The most pertinent example in our setting is the solution set of a parametric optimization problem. That is, considering the problem
\[
\min_{y\in Y} \{ f(x,y)\,:\, g(x,y)\leq 0 \},
\]
with $f:X\times Y\to \R$, and $g:X\times Y\to \R^m$. Then, we can define the set-valued map $S:X\tto Y$ where $S(x)$ is the set of minimizers of the optimization problem parametrized by $x$. That is, $S(x) := \argmin_y\{ f(x,y)\,:\, g(x,y)\leq 0\}$.
\end{example}

\begin{warning}
    One might be tempted to say that a set-valued map $M:X\tto Y$ is just a function between $X$ and the power set $2^Y$. While this might be correct in terms of the correspondence $x\mapsto M(x)$, the properties that we will study from $M$ will be linked to the space $Y$, rather than the space $2^Y$. Thus, it is better to do the distinction. 
\end{warning}

For a set-valued map $M:X\tto Y$ we define:
\begin{enumerate}
    \item The \textbf{domain} of $M$ as 
    \begin{equation}\label{ch03:eq:domainSetValuedMap}
        \dom M = \{ x\in X\,:\, M(x)\neq\emptyset\}.
    \end{equation}
    Note that this definition coincides with the domain of $S$ defined in \eqref{ch02:eq:domS-preliminaryDef}.
    \item The \textbf{graph} of $M$ as
    \begin{equation}\label{ch03:eq:gphSetValuedMap}
        \gph M = \{ (x,y)\in X\times Y\,:\, y\in M(x)\}.
    \end{equation}
    \item For $A\subset X$, the \textbf{image} of $A$ through $M$ is given by
    \(    
        M(A) = \bigcup\{ M(x)\,:\, x\in A\}.
    \)
    The \textbf{range} of $M$ is simply given by $\mathrm{Range}(M) := M(X)$.
    \item For $B\subset Y$, the \textbf{preimage} of $B$ through $M$ is given by
 \(
 M^{-1}(B) = \{x\in X\ :\ M(x)\cap B\neq\emptyset\}.
\)
    With this definition, $\dom M = M^{-1}(Y)$.
    \item The \textbf{inverse} set-valued map $M^{-1}:Y\tto X$, which is given by
    \begin{equation}
        M^{-1}(y) = \{ x\in X\, :\, y\in M(x) \}.
    \end{equation}
    Note that with this definition, for a set $B\subset Y$, $M^{-1}(B)$ can denote either the preimage of $B$ through $M$, or the image of $B$ through $M^{-1}$. Moreover, $\mathrm{Range}(M^{-1}) = \dom M$. 
\end{enumerate}

\begin{note}{\textbf{Remark:} Set-valued maps coincide with sets in the product space}
    Note that any set-valued map can be identified with its graph by considering the identification
    \begin{equation}
        M(x) = \{ y\in Y\, :\, (x,y) \in \gph M \}.
    \end{equation}
    Basicaly, the images of $M$ can be identified as \textbf{slices} on $\gph M$. Moreover, any set $K\subset X\times Y$ can be seen as the graph of a set-valued map, just by replacing $\gph M$ by $K$ in the formula above. Thus, set-valued maps from $X$ to $Y$ and subsets of $X\times Y$ are in a one-to-one correspondence. This fact is illustrated in Figure~\ref{ch03:fig:SetValued-Graph}
\end{note}
\begin{figure}[ht]
    \centering
    \includegraphics[width=0.7\linewidth]{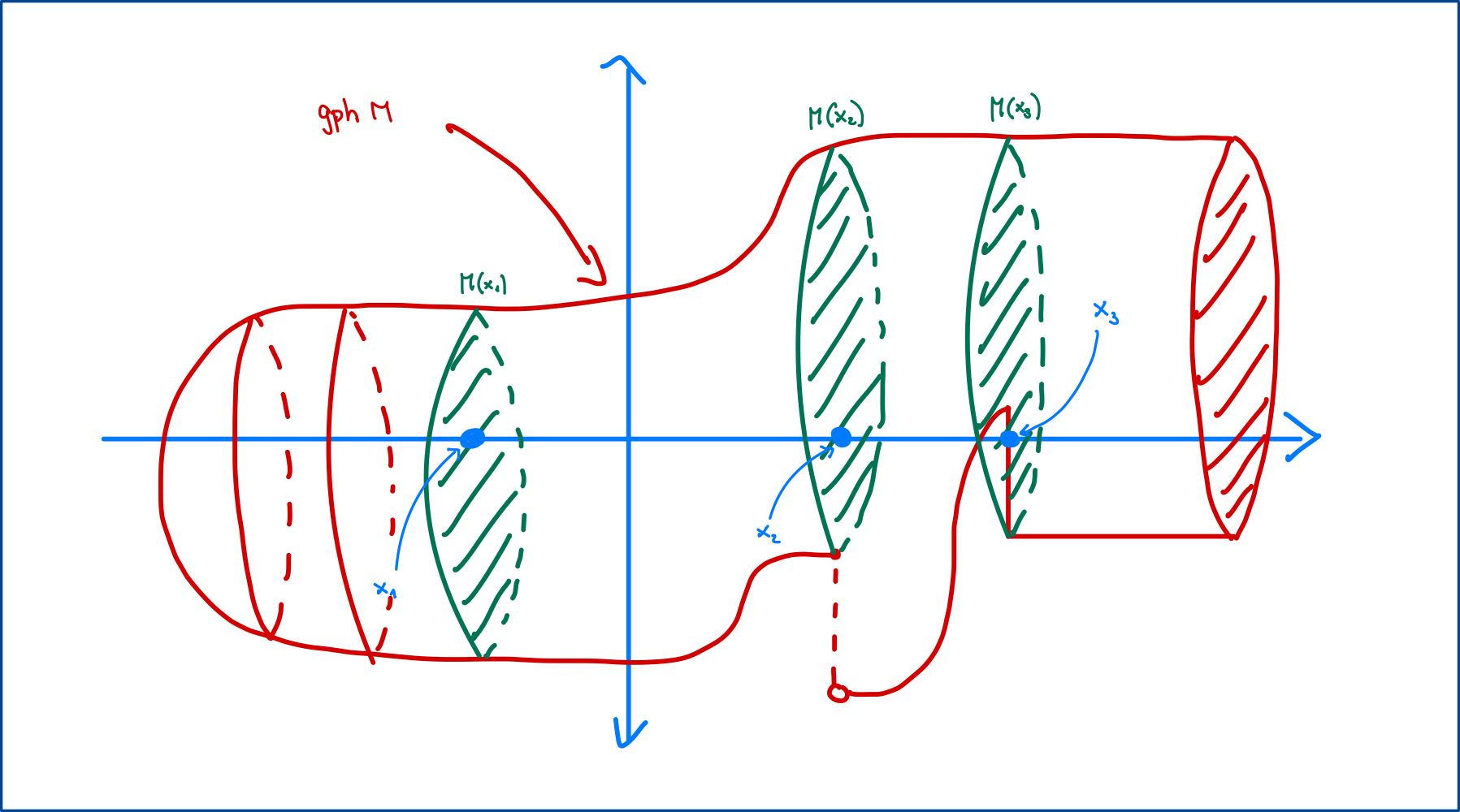}
    \caption{A set-valued map identified as the slices of its graph.}
    \label{ch03:fig:SetValued-Graph}
\end{figure}
\subsection{Upper and Lower semicontinuity}

Let us now suppose that $X$ and $Y$ are metric spaces, with metrics $d_X$ and $d_Y$ respectively. Our aim is to study the notions of continuity of a set-valued map $M:X\tto Y$. The approach here is to consider extensions to set-valued maps for the \textit{neighborhood-based} definition of continuity.

\begin{note}{Notation: Neighborhoods}
    For a metric space $(X,d_X)$ and a point $x_0\in X$. We denote the open ball of radius $\delta>0$ and centered at $x_0$ by $B_{d_X}(x_0,\delta)$. We denote the family of all neighborhoods of $x_0$ as $\mathcal{N}_{d_X}(x_0)$, that is,
    \[
    \mathcal{N}_{d_X}(x_0) = \{ V\subset X\ :\ \exists \delta>0, B_{d_X}(x_0,\delta)\subset V\}.
    \]
    If there is no ambiguity about the metric, we might simply write $B_X(x_0,\delta)$ and $\mathcal{N}_{X}(x_0)$, or even $B(x_0,\delta)$ and $\mathcal{N}(x_0)$.
\end{note}

Recall first that a function $f:X\to Y$ is continuous at $x_0\in X$ if
\begin{equation}\label{ch03:eq:def-continuity}
\forall V\in\mathcal{N}_{Y}(f(x_0)),\,\exists U\in \mathcal{N}_X(x_0)\text{ such that }\forall x\in U,\, f(x)\in V.
\end{equation}
A first idea to generalize continuity to a set-valued map $M$ is to ask the same behavior that we have for the images of $f$ to the images of $M$. The inclusions point-to-set need to be replaced by inclusions of sets. That is, to replace $f(x_0)\in V$ and $f(x)\in V$ by $M(x_0)\subset V$ and $M(x)\subset V$, respectively, and consider $V$ to be directly an open set. This leads to the notion of upper semicontinuity.
\begin{definition}[Upper semicontinuity]\label{ch03:def:UpperSemicontinuity} Let $X$ and $Y$ be metric spaces, and let $x_0\in X$. We say that $M$ is \textbf{upper semicontinuous} at $x_0$ if 
\[
\forall V\subset Y\text{ open with }M(x_0)\subset V,\,\exists U\in\mathcal{N}_{X}(x_0),\text{ such that }\forall x\in U,\, M(x)\subset V.
\] 
We simply say that $M$ is upper semicontinuous (respectively, on a set $K\subset X$) if it is so at every point $x\in X$ (respectively, 
at every point $x\in K$).
\end{definition}

A second idea to generalize continuity to a set-valued map $M$ is to study the continuity at $x_0$ by looking at the continuity of each point $y_0\in M(x_0)$. One can think on selecting points $y\in M(x)$ such that the selection, as a function, is continuous at $y_0$. In other words, whenever we choose a neighborhood $V\in\mathcal{N}_Y(y_0)$, there should be a neighborhood $U\in\mathcal{N}_X(x_0)$ such that for every $x\in U$ ($x$ close enough to $x_0$) we can select $y\in M(x)$ with $y\in V$ (select $y$ close enough to $y_0$). This leads to the notion of lower semicontinuity.

\begin{definition}[Lower semicontinuity]\label{ch03:def:LowerSemicontinuity} Let $X$ and $Y$ be metric spaces, and let $x_0\in X$. We say that $M$ is \textbf{lower semicontinuous} at $x_0$ if 
\[
\forall V\subset Y\text{ open with }M(x_0)\cap V\neq\emptyset,\,\exists U\in\mathcal{N}_{X}(x_0),\text{ such that }\forall x\in U,\, M(x)\cap V\neq\emptyset.
\] 
We simply say that $M$ is lower semicontinuous (respectively, on a set $K\subset X$) if it is so at every point $x\in X$ (respectively, 
at every point $x\in K$).
\end{definition}

Note that both definitions coincide with the classic continuity for single-valued maps. That is, for any function $f:X\to Y$ we can consider the set-valued map $F:X\tto Y$ given by $F(x) = \{f(x)\}$. In such a case, upper semicontinuity of $F$ is equivalent to lower semicontinuity of $F$, and both are equivalent to the continuity of $f$. However, in general, both notions might differ.

\begin{example}[Gap between Upper and Lower semicontinuity]\label{ch03:ex:UpperVsLower} let $X=Y=[0,1]$ and consider the set-valued maps $F_1,F_2:[0,1]\tto[0,1]$ given by
\[
F_1(t) = \begin{cases}
    \{0\}&t<1,\\
    [0,1]&t=1,
\end{cases}\quad\text{ and }\quad F_2(t) = \begin{cases}
    [0,1]&t<1,\\
    \{0\}&t=1.
\end{cases}
\]
\begin{figure}[ht]
    \centering
    \includegraphics[width=0.7\linewidth]{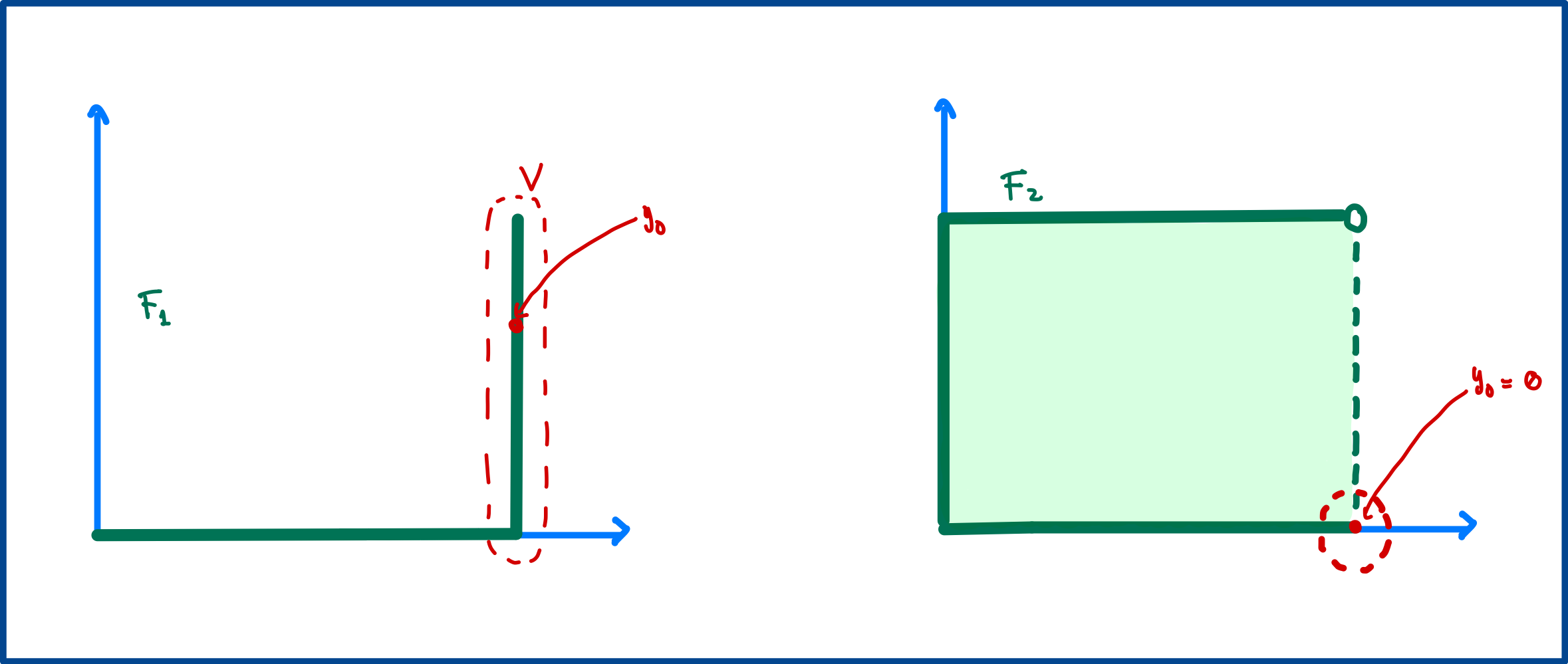}
    \caption{Illustration of gap between Upper and Lower semicontinuity.}
    \label{ch03:fig:UpperVsLower}
\end{figure}

Then, $F_1$ is upper semicontinuous at but not lower semicontinuous at $t=1$, and $F_2$ is lower semicontinuous but not upper semicontinuous at $t=1$.  See Figure \ref{ch03:fig:UpperVsLower}.  
\end{example}

\begin{definition}[Continuity (in Berge's sense)] Let $X$ and $Y$ be topological spaces, and let $x_0\in X$. We say that $M$ is \textbf{continuous} at $x_0$ if it is both, upper and lower semicontinuous at $x_0$. We simply say that $M$ is continuous (respectively, on a set $K\subset X$) if it is so at every point $x\in X$ (respectively, 
at every point $x\in K$).   
\end{definition}

\subsection{Characterizations of upper and lower semicontinuity}

Let us now study some characterizations of upper and lower semicontinuity of set-valued maps.

\begin{proposition} Let $X,Y$ be two topological spaces and $M:X\tto Y$ be a set-valued map. Then,
\begin{enumerate}
    \item $M$ is upper semicontinuous $\iff$ the premiage through $M$ of closed sets of $Y$ are closed on $X$.
    \item $M$ is lower semicontinuous $\iff$ the premiage through $M$ of open sets of $Y$ are open on $X$.
\end{enumerate}
\end{proposition}
\begin{proof} Let us show the first equivalence:

\paragraph{($\implies$):} Let $K\subset Y$ be a closed set. let us show that $A = X\setminus M^{-1}(K)$ is open. Indeed, choose $x\in A$. By construction, $M(x)\subset V:=Y\setminus K$, which is an open set of $Y$. Then, there exists $U\in \mathcal{N}_X(x)$ such that $M(U)\subset V$. In particular, $M(U)\cap K = \emptyset$ and so $U\subset X\setminus M^{-1}(K) = A$. Since $x\in A$ is arbitrary, we deduce that $A$ is open.
\paragraph{($\impliedby$):} Let $x_0\in X$ and an open set $V\subset Y$ with $M(x_0)\subset V$. Consider the closed set $K=Y\setminus V$. Then, $M^{-1}(K)$ is closed and so $U = X\setminus M^{-1}(K)$ is open. By construction, $x_0\in U$ and $M(U)\subset V$. From the arbitrariness of $x_0$ and $V$, we deduce that $M$ is upper semicontinuous.\\

\noindent Let us show the second equivalence:

\paragraph{($\implies$):}Let $V$ be an open set of $Y$, and choose $x_0\in M^{-1}(V)$. Then, $M(x_0)\cap V$ is nonempty and so, since $M$ is lower semicontinuous at $x_0$, there exists $U\in \mathcal{N}_X(x_0)$ such that $M(x)\cap V\neq \emptyset$ for every $x\in U$. This yields that $U\subset M^{-1}(V)$. Since this holds for every $x_0\in M^{-1}(V)$, we deduce that $M^{-1}(V)$ is open.
\paragraph{($\impliedby$):} Let $x_0\in X$ and let $V\subset Y$ be an open set such that $V\cap M(x_0)\neq \emptyset$. Then, $U=M^{-1}(V)$ is open, and so $U\in\mathcal{N}_X(x_0)$. By construction, $M(x)\cap V\neq \emptyset$ for every $x\in U$. From the arbitrariness of $x_0$ and $V$, we deduce that $M$ is lower semicontinuous.

\end{proof}

While the last proposition and the definitions tell us that it is possible to deepen the study of continuity in general topological spaces, our interest is restricted to metric spaces. Since $X$ and $Y$ are metric spaces, it is possible to derive sequential characterizations of both, upper and lower semicontinuity.

\begin{note}{Note: Semilimits of functions} Let $f:X\to \R$ be a real-valued function over a metric space $X$, and let $x_0\in X$. The inferior and superior semilimits of $f$ at $x_0$ are given by
\begin{align}
    \liminf_{x\to x_0} f(x) &:= \sup_{U\in \mathcal{N}(x_0)} \inf_{x\in U} f(x) = \sup_{n\in\N} \inf_{x\in B(x_0,1/n)} f(x).\label{ch03:eq:Def-liminf}\\
    \limsup_{x\to x_0} f(x) &:= \inf_{U\in \mathcal{N}(x_0)} \sup_{x\in U} f(x) = \inf_{n\in\N} \sup_{x\in B(x_0,1/n)} f(x).\label{ch03:eq:Def-limsup}
\end{align}
Semilimits are linked to semicontinuity of functions. Namely,
\begin{enumerate}
    \item $f$ is upper semicontinuous at $x_0$ if $f(x_0) \geq \displaystyle\limsup_{x\to x_0}f(x)$.
    \item $f$ is lower semicontinuous at $x_0$ if $f(x_0) \leq \displaystyle\liminf_{x\to x_0}f(x)$.
\end{enumerate}
Of course, a function $f$ is continuous at $x_0$ if and only if it is both upper and lower semicontinuous at $x_0$.

Recall also that for a sequence $(r_n)\subset \R$, one can define the semilimits of $(r_n)$ as
\begin{align}
    \liminf_{n\to\infty} r_n &:= \sup_{n\in\N} \inf_{k\geq n} r_k.\label{ch03:eq:Def-liminf-seq}\\
    \limsup_{n\to\infty} r_n &:= \inf_{n\in\N} \sup_{k\geq n} r_k.\label{ch03:eq:Def-limsup-seq}
\end{align}
Then, semicontinuity of real-valued functions over metric spaces can be characterized sequentially as follows:
\begin{enumerate}
    \item $f$ is upper semicontinuous at $x_0$ if and only if for every sequence $x_n\to x_0$, one has that $f(x_0)~\geq~\displaystyle\limsup_{n\to\infty}f(x_n)$.
    \item $f$ is lower semicontinuous at $x_0$ if and only if for every sequence $x_n\to x_0$, one has that $f(x_0)~\leq~\displaystyle\liminf_{n\to\infty}f(x_n)$.
\end{enumerate}
\end{note}

\begin{theorem}[Sequential characterization of lower semicontinuity]\label{ch03:thm:seq-LSC} Let $X$ and $Y$ be two metric spaces, and let $x_0\in X$. For a set-valued map $M:X\tto Y$, the following assertions are equivalent:
\begin{enumerate}
    \item[(i)] $M$ is lower semicontinuous  at $x_0$.
    \item[(ii)] For every sequence $(x_n)\subset X$ converging to $x_0$ and every $y_0\in M(x_0)$, there exists a sequence $(y_n)\subset Y$ and $n_0\in \N$ such that $y_n\to y_0$ and $y_n\in M(x_n)$ for every $n\geq n_0$.
    \item[(iii)] $M(x_0) \subset \{ y\in Y\, :\, \limsup_{x\to x_0} d_Y(y,M(x)) = 0  \}$.
\end{enumerate}
  
\end{theorem}
\begin{proof} Without lose of any generality, let us assume that $M(x_0)\neq\emptyset$ (Indeed, if $M(x_0) = \emptyset$ all statements in the theorem are trivially verified). 

\paragraph{$(i)\implies (ii)$:} Suppose then that $M$ is lower semicontinuous at $x_0$ and take a sequence $(x_n)\subset X$ converging to $x_0$ and $y_0\in M(x_0)$. Lower semicontinuity entails that there is $n_0\in\N$ such that $M(x_n)\neq\emptyset$ for every $n\geq n_0$. Now, for each $n\geq n_0$, choose $y_n\in M(x_n)$  verifying that
\[
 d(y_n,y_0) \leq d(y_0,M(x_n)) + \frac{1}{n}.
\]
 Fix $\varepsilon>0$ and set $V = B(y_0,\varepsilon/2)$. By lower semicontinuity, there exists $n_\varepsilon\geq \max\{n_0, 2\varepsilon^{-1}\}$ large enough such that for all $n\geq n_{\varepsilon}$, $V\cap M(x_n)\neq\emptyset$. Then, for every $n\geq n_{\varepsilon}$ we get that
 \[
 d(y_n,y_0) \leq d(y_0,M(x_n)) + \frac{1}{n} \leq \frac{\varepsilon}{2} + \frac{1}{n_{\varepsilon}}\leq \frac{\varepsilon}{2} + \frac{\varepsilon}{2} = \varepsilon.
 \]
 Thus, $y_n\to y_0$, which proves the necessity.

\paragraph{$(ii)\implies (iii)$:} Let us denote $S= \{ y\in Y\, :\, \limsup_{x\to x_0} d_Y(y,M(x)) = 0  \}$. Let $y\in M(x_0)$ and suppose that $y\notin S$. Then, $\limsup_{x\to x_0} d_Y(y,M(x_0))>0$ and so there exist $\varepsilon>0$ and a sequence $x_n\to x_0$ such that $d_Y(y,M(x_n))>\varepsilon$. However, this is a contradiction since, by hypothesis, there must exist a sequence $y_n\to y$ with $y_n\in M(x_n)$ for every $n\in\N$ large enough. We conclude that $M(x_0)\subseteq S$.

\paragraph{$(iii)\implies (i)$:} Take an open set $V\subset Y$ such that $M(x_0)\cap V\neq\emptyset$, and choose $y\in V\cap M(x_0)$. Without lose of generality, we can assume $V = B(y,\varepsilon)$ for some $\varepsilon>0$. Now, since $y\in M(x_0)$, we have that $\limsup_{x\to x_0} d_Y(y,M(x)) = 0$. Thus, there exists $U\in\mathcal{N}_X(x_0)$ such that
\[
\sup_{x\in U} d(y,M(x)) < \frac{\varepsilon}{2}.
\]
In particular, for each $x\in U$, $d(y,M(x))<\varepsilon/2$ and so there must exists $z\in M(x)$ with $z\in B(y,\varepsilon)$. Thus, $M(x)\cap V \neq\emptyset$ for every $x\in U$. We conclude that $M$ is lower semicontinuous at $x_0$, finishing the proof.
\end{proof}

Upper semicontinuity needs complementary hypotheses to be characterized sequentially. In fact, what we will do is to characterizes it in terms of closedness of the graph.

\begin{theorem}[Sequential characterization of Upper Semicontinuity] \label{ch03:thm:Seq-USC} Let $X$ and $Y$ be two metric spaces, and let $x_0\in X$. For a set-valued map $M:X\tto Y$ consider the following assertions:
\begin{enumerate}
\item[(i)] $M$ is upper semicontinuous at $x_0$.
\item[(ii)] For any sequence $(x_n,y_n)\in \gph M$ converging to $(x_0,y_0)$, one has that $y_0\in M(x_0)$.
\item[(iii)] $M(x_0) \supset \{ y\in Y\,:\, \liminf_{x\to x_0} d_Y(y,M(x)) = 0 \}$.
\end{enumerate}
Then $(ii)\iff (iii)$ and they are independent of $(i)$. However:
\begin{enumerate}
    \item If $M(x_0)$ is closed, then $(i)\implies (ii)$.
    \item If there exists $U\in\mathcal{N}_X(x_0)$ with $\overline{M(U)}$ compact, then $(ii)\implies (i)$.
\end{enumerate}
\end{theorem}
\begin{proof}Let us show first the equivalence $(ii)\iff (iii)$.
 \paragraph{($\implies$):} Let $S = \{y \in Y\, :\,  \liminf_{x\to x_0} d(y,M(x)) = 0\}$, and pick $y\in S$. Then, there exists a sequence $x_n\to x$ such that $d(y,M(x_n))\to 0$. By considering a subsequence $(x_n)_{n\geq n_0}$ with $n_0\in\N$ large enough, we can assume that $M(x_n)\neq\emptyset$ for every $n\in\N$. Thus, we can select a sequence $(y_n)$ such that
 \[
 y_n\in M(x_n)\quad\text{and}\quad d(y,y_n) \leq d(y,M(x_n))+\frac{1}{n}, \qquad\forall n\in\N.
 \]
 Then, $y_n\to y$, and by hypothesis, $y\in M(x_0)$. We deduce that $S\subset M(x_0)$.
 \paragraph{($\impliedby$):} Let $(x_n,y_n)\subset \gph M$ converging to $(x_0,y_0)\in X\times Y$. By taking a subsequence, we can suppose that $d(x_n,x_0)\leq 1/n$. Thus, we can write that
 \begin{align*}
     \liminf_{x\to x_0} d(y_0,M(x)) &= \sup_{n\in\N}\inf_{x\in B(x_0,1/n)} d(y_0,M(x))\\
     &\leq \sup_{n\in\N}\inf_{k\geq n} d(y_0, M(x_k))\\
     &\leq \sup_{n\in\N}\inf_{k\geq n} d(y_0, y_k) = \liminf_{n\to\infty} d(y_0,y_n) = 0.
 \end{align*}
 Thus, $y_0\in S$ and by hypothesis, $y_0\in M(x_0)$. This completes the proof of $(ii)\iff (iii)$.\\

\noindent Now, for the independence with $(i)$, see Example~\ref{ch03:ex:Closed-Vs-USC}. Finally, let us prove the last two statements.
 \begin{enumerate}
     \item Assume that $M(x_0)$ is closed and that $M$ is upper semicontinuous. Let $(x_n,y_n)\subset \gph M$ converging to $(x_0,y_0)\in X\times Y$. Let $V_{\varepsilon} = M(x_0) + B(0,\varepsilon)$. Clearly, $V_{\varepsilon}$ is open and contains $M(x_0)$. Thus, there exists $U\in\mathcal{N}(x_0)$ such that $M(U)\subset V_{\varepsilon}$. Since $x_n\to x_0$, this yields that for $n\geq n_0$ with $n_0$ large enough, we have that $(y_n)_{n\geq n_0}\subset V_{\varepsilon}$ which yields, by continuity of $d(\cdot,M(x_0))$
     \[
     d(y_0,M(x_0)) = \lim_n d(y_n,M(x_0)) \leq \varepsilon.
     \]
     Since $\varepsilon>0$ is arbitrary, we deduce that $d(y_0,M(x_0)) = 0$, and since $M(x_0)$ is closed, we conclude that $y_0\in M(x_0)$.
     \item Reasoning by absurd, let us suppose that $M$ is not upper semicontinuous at $x_0$. Then, there exists an open set $V\subset Y$ with $M(x_0)\subset V$ such that, for every $O\in \mathcal{N}(x_0)$, the set $M(O)\setminus V$ is nonempty. in particular, for each $n\in \N$ we can select $x_n\in O_n:=U\cap B(x_0,1/n)$ and $y_n\in M(x_n)$ such that $y_n\in M(O_n)\setminus V$. Note that $(y_n)\subset M(U)$. Thus, by compactness of $\overline{M(U)}$, we can assume (up to a subsequence of $(x_n,y_n)$) that $(y_n)$ converges. Then, we have that $y_0 = \lim_n y_n \in M(x_0)\subset V$, which is a contradiction since $\lim_n y_n \in Y\setminus V$. The proof is now completed.
 \end{enumerate}
\end{proof}

The condition $(ii)$ in Theorem~\ref{ch03:thm:Seq-USC} is often referred as the \textit{closedness of $M$ at $x_0$}. That is, a set-valued map $M$ (or its graph $\gph M$) is said to be \textbf{closed at $x_0$}, if the sequential condition $(ii)$ holds. Indeed, it is not hard to verify that the condition $(ii)$ holds at every $x\in X$ if and only if $\gph M$ is a closed set. Thus, condition $(ii)$ is like a pointwise closedness of $\gph M$.

\begin{example}[Closedness versus Upper semicontinuity] \label{ch03:ex:Closed-Vs-USC}
Consider the set-valued maps $M_1,M_2:[0,+\infty)\tto \R$ given by
\[
M_1(t) = \begin{cases}
   \displaystyle \left[\frac{1}{t+1}, \frac{1}{t}\right]&t>0,\\
    \{0\}&t=0,
\end{cases}\quad\text{ and }\quad M_2(t) = \begin{cases}
    \displaystyle \left[\frac{1}{t+1}, \frac{1}{t}\right]&t>0,\\
    (0,+\infty)&t=0.
\end{cases}
\]
\begin{figure}[ht]
    \centering
    \includegraphics[width=0.7\linewidth]{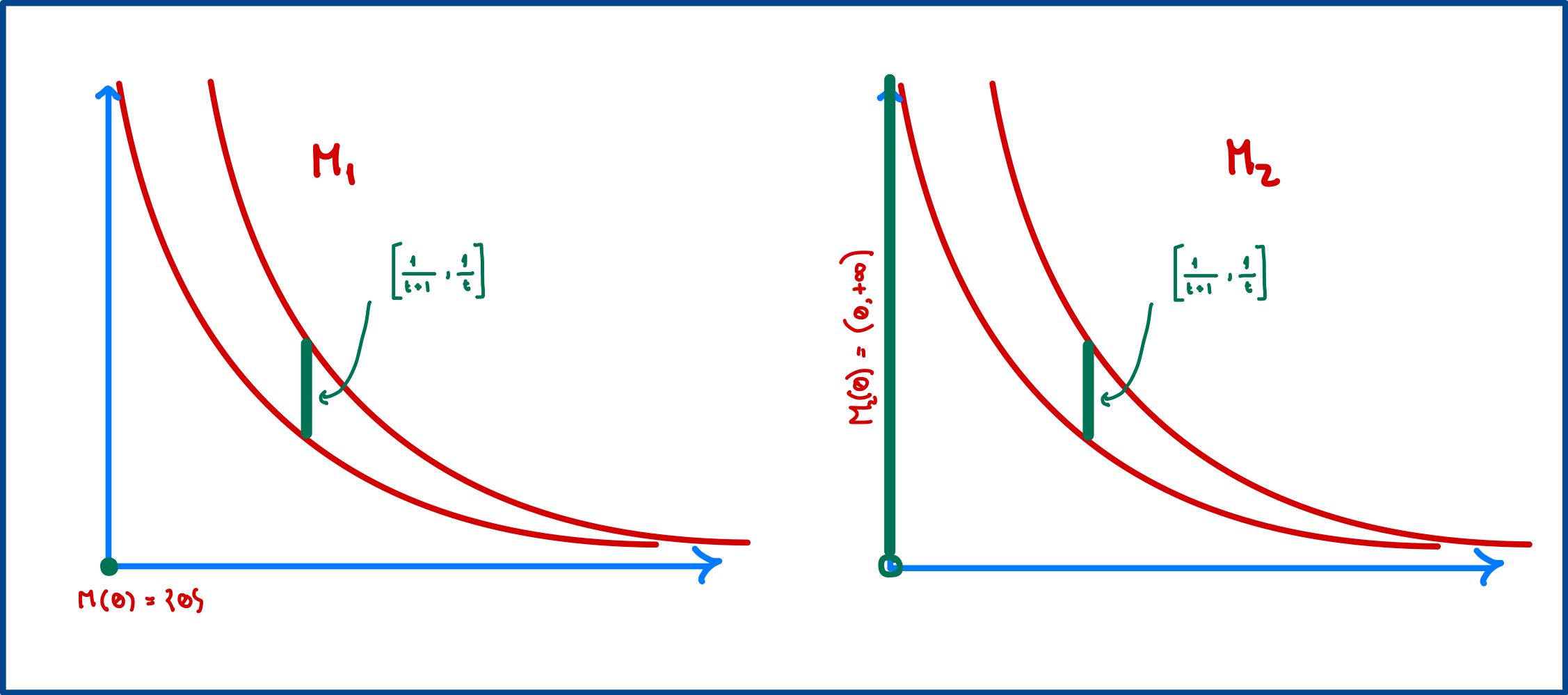}
    \caption{Illustration of gab between closedness and upper semicontinuity.}
    \label{ch03:fig:ClosedVsUpper}
\end{figure}

Then, $M_1$ is closed at $0$ but not upper semicontinuous at $0$, while $M_2$ is upper semicontinuous at $0$ but not closed at $0$. The situation is illustrated in Figure~\ref{ch03:fig:ClosedVsUpper}.
\end{example}

By considering global closedness and global upper semicontinuity, Theorem \ref{ch03:thm:Seq-USC} yields the following direct (and useful) corollary.

\begin{corollary}\label{ch03:cor:USC-ClosedGraph} Let $X$ and $Y$ be two metric spaces and $M:X\tto Y$, One has that
 \begin{enumerate}
    \item If $M$ is upper semicontinuous and $M$ has closed values, then $M$ is closed.
    \item If $M$ is closed and $M$ is locally precompact (i.e. for every $x\in X$, there exists $U\in\mathcal{N}(x)$ such that $\overline{M(U)}$ is compact), then $M$ is upper semicontinuous.
\end{enumerate}   
\end{corollary}

\section{Berge Maximum Theorem}

Upper and lower semicontinuity of set-valued maps allows us to study the continuity of the solution set and the value function of parametric optimization problems. In this section, we present the Berge Maximum theorem in two forms. Then, we present some important examples where upper/lower semicontinuity holds, allowing us to apply the Maximum theorem.
\subsection{The two forms of the Maximum theorem}

\begin{theorem}[Berge Maximum theorem, first form] \label{ch03:thm:Berge-firstForm} Let $X$ and $Y$ be two metric spaces, $f:X\times Y\to\R$ a function, and $F:X\tto Y$ a set-valued map with nonempty values. Let us consider the \textbf{marginal function}
\begin{align*}
    g:X&\to\R\cup\{+\infty\}\\
    x&\mapsto \sup_y\{ f(x,y)\ :\ y\in F(x) \}.
\end{align*}
 Then, 
 \begin{enumerate}
     \item If $f$ is lower semicontinuous and $F$ is lower semicontinuous, then $g$ is lower semicontinuous.
     \item If $f$ is upper semicontinuous, $F$ is upper semicontinuous, and $F$ takes compact values, then $g$ is upper semicontinuous.
 \end{enumerate}
\end{theorem}
\begin{proof}Let us prove each of the statements separately.
\paragraph{1.} Let $x\in X$ and $(x_n)\subset X$ with $x_n\to x$. It is enough to prove that $g(x) \leq \liminf_n g(x_n)$. 

Suppose first that $g(x)\in\R$. Let $\varepsilon>0$ and choose $y\in F(x)$ such that $g(x) \leq f(x,y) + \varepsilon$. Since $F$ is lower semicontinuous, there exists a sequence $(y_n)\subset Y$ converging to $y$ and such that $y_n\in F(x_n)$ for every $n\in\N$ large enough. Then, lower semicontinuity of $f$ allows us to write
\[
g(x) \leq \varepsilon+ f(x,y) \leq \varepsilon +\liminf_n f(x_n,y_n) \leq \varepsilon +\liminf_n g(x_n).
\]
Since $\varepsilon>0$ is arbitrary, the result follows.

Now, suppose that $g(x) = +\infty$. Let $M>0$ and choose $y\in F(x)$ so $f(x,y)>M$. Again taking a sequence $(y_n)$ converging to $y$ and such that $y_n\in F(x_n)$ for every $n\in\N$ large enough, we can write
\[
\liminf_n g(x_n) \geq \liminf_n f(x_n,y_n) \geq f(x_n,y_n) > M.
\]
Since $M$ is arbitrary, we conclude that $\liminf_n g(x_n) = +\infty$ and the result follows.
\paragraph{2.}Let $x\in X$ and $(x_n)\subset X$ with $x_n\to x$. It is enough to prove that $g(x) \geq \limsup_n g(x_n)$. Let $\alpha := \limsup_n g(x_n)\in\overline{\R}$. Up to a subsequence, we might suppose that $g(x_n)$ converges to $\alpha$.

Now, since $F(x_n)$ is compact and $f(x_n,\cdot)$ is upper semicontinuous, there exists $y_n\in F(x_n)$ such that $g(x_n) = f(x_n,y_n)$. Also, since $F(x)$ is compact, we can always find
\[
z_n \in \mathrm{Proj}(y_n, F(x)) := \argmin_{w\in F(x)} d(y_n,w).
\]
Up to a subsequence,  we can suppose that $z_n\to z\in F(x)$. 

Let $\varepsilon>0$ and let us consider $V = \{y\in Y\, :\, d(y,F(x))<\tfrac{\varepsilon}{2} \}$. Since $F$ is upper semicontinuous, there exists $n_0\in \N$ large enough such that $F(x_n)\subset V$ for all $n\geq n_0$. By enlarging $n_0$ if necessary, we might assume that $d(z_n,z)\leq \tfrac{\varepsilon}{2}$ for each $n\geq n_0$ as well. Then,
\[
\forall n\geq n_0, d(y_n,z) \leq d(y_n,z_n) + d(z_n,z) = d(y_n,F(x_n)) + d(z_n,z) \leq \varepsilon.
\]
Since $\varepsilon>0$ is arbitrary, we deduce that $y_n\to z$. Then, by upper semicontinuity of $f$ we have that
\[
\alpha = \lim_n f(x_n,y_n) \leq f(x,z) \leq \sup_{y\in F(x)} f(x,y) = g(x).
\]
The proof is then completed.
\end{proof}

\begin{theorem}[Berge Maximum theorem, second form]\label{ch03:thm:Berge-secondForm} Let $X$ and $Y$ be two metric spaces, $f:X\times Y\to\R$ be a continuous function, and $F:X\tto Y$ be a continuous (i.e. upper and lower semicontinuous) set-valued map with nonempty compact values. Then,
\begin{enumerate}
    \item The value function $\varphi(x) := \min_y\{ f(x,y)\, :\, y\in F(x)\}$ is continuous.
    \item The solution map $S:X\tto Y$ given by $S(x) = \argmin_{y}\{ f(x,y)\, :\, y\in F(x)\}$ is upper semicontinuous.
\end{enumerate}
\end{theorem}
\begin{proof} The continuity of the value function is exactly the conclusion of Theorem~\ref{ch03:thm:Berge-firstForm}. Thus, we only need to prove the second statement. Reasoning by absurd, let us suppose that $S$ is not upper semicontinuous. That is, there exists $x\in X$ and $V\subset Y$ open with $S(x)\subset V$ such that for every $U\in\mathcal{N}(x)$, $S(U)\setminus V$ is nonempty. In particular, we can construct a sequence $(x_n,y_n)\subset \gph S$ such that $x_n\to x$ and $y_n\in V\setminus S(x_n)$ for every $n\in \N$.

    We can replicate the argument of the second part of the proof of Theorem~\ref{ch03:thm:Berge-firstForm} to deduce that, up to a subsequence, $(y_n)$ is convergent and  $y = \lim_n y_n \in F(x)$. Using the continuity of $f$ and the continuity of the value function $\varphi$, we get that
    \[
    \varphi(x) = \lim_n \varphi(x_n) = \lim_n f(x_n,y_n) = f(x,y).
    \]
    Then, $y\in S(x) \subset V$ which is a contradiction. The proof is then completed.
\end{proof}
\begin{note}{Remark: First form vs Second form}
    The first form of the Berge Maximum theorem allow us to decouple the sufficient conditions to have upper and lower semicontinuity of the value function, separately. The second form instead mixes up both sufficient conditions to deduce also the upper semicontinuity of the solution map.
\end{note}

\subsection{Continuity of constrained sets}\label{ch02:subsec:ContinuityCOnstrainedSets}

How hard is to verify upper and lower semicontinuity? In order to use Berge Maximum theorem, this is a very relevant issue we need to address. In particular, we will be interested in set-valued maps $F:X\tto Y$ of the form
\[
F(x) = \{ y\in Y\,:\, g_i(x,y)\leq 0,\, i=1,\ldots,m \}.
\]
In this case, closedness is very easy to verify.
\begin{proposition}[Closedness of constrained maps]\label{ch03:prop:closednessConstrained} Let $X$ and $Y$ be two metric spaces and let $F:X\tto Y$ given by
    \[
    F(x) = \{ y\in Y\,:\, g_i(x,y)\leq 0,\, i=1,\ldots,m \},
    \]
  where $g_i:X\times Y \to \R$ are lower semicontinuous functions for $i=1,\ldots,m$.  Then, $F$ is closed.
\end{proposition}
\begin{proof}
    Chose $(x_n,y_n)\subset \gph F$ with $(x_n,y_n)\to (x_0,y_0)\in X\times Y$. Then, for each $i\in\{1,\ldots,m\}$
    \[
    g_i(x_0,y_0) \leq \liminf_{n\to \infty} g_i(x_n,y_n) \leq 0.
    \]
    Thus, $y_0 \in F(x_0)$, finishing the proof.
\end{proof}

Usually, upper semicontinuity is deduced using Proposition~\ref{ch03:prop:closednessConstrained} and some compactness assumption, either on the set $Y$, or on the set $\bigcup_{x\in X}\{ y\,:\, g_i(x,y)\leq 0\}$, for some $i\in\{1,\ldots,m\}$. One might hope that upper semicontinuity (or maybe continuity) of the functions $(g_i\,:\, i=1,\ldots,m)$ might entail lower semicontinuity of $F$ as set-valued map. However, this is not the case in general.

\begin{example}[Constrained map failing lower semicontinuity] Consider the set $X = [0,\pi/2]$ and $Y=\R^2$, and define $F:X\tto Y$ given by
\[
F(\theta) = \left\{ y=(y_1,y_2) \,:\, \begin{array}{l}
     y_1\geq -1\\
     y_2\leq 0\\
     \sin(\theta)y_1 \leq \cos(\theta)y_2
\end{array}  \right\}.
\]
Note that, for $0<\theta<\pi/2$, $F(\theta)$ is given by the compact triangle between the lines $y_1 = -1$, $y_2=0$ and $y_2 = \tan(\theta) y_1$. For $\theta = \pi/2$, $F(\pi/2)$ is given by the (unbounded) box $[-1,0]\times (-\infty,0]$. Finally, at $\theta = 0$, $F(0)$ coincides with the (unbounded) ray $[-1,+\infty)\times \{0\}$. See Figure~\ref{ch03:fig:Constrained-Not-LSC}.

\begin{figure}[ht]
    \centering
    \includegraphics[width=0.9\linewidth]{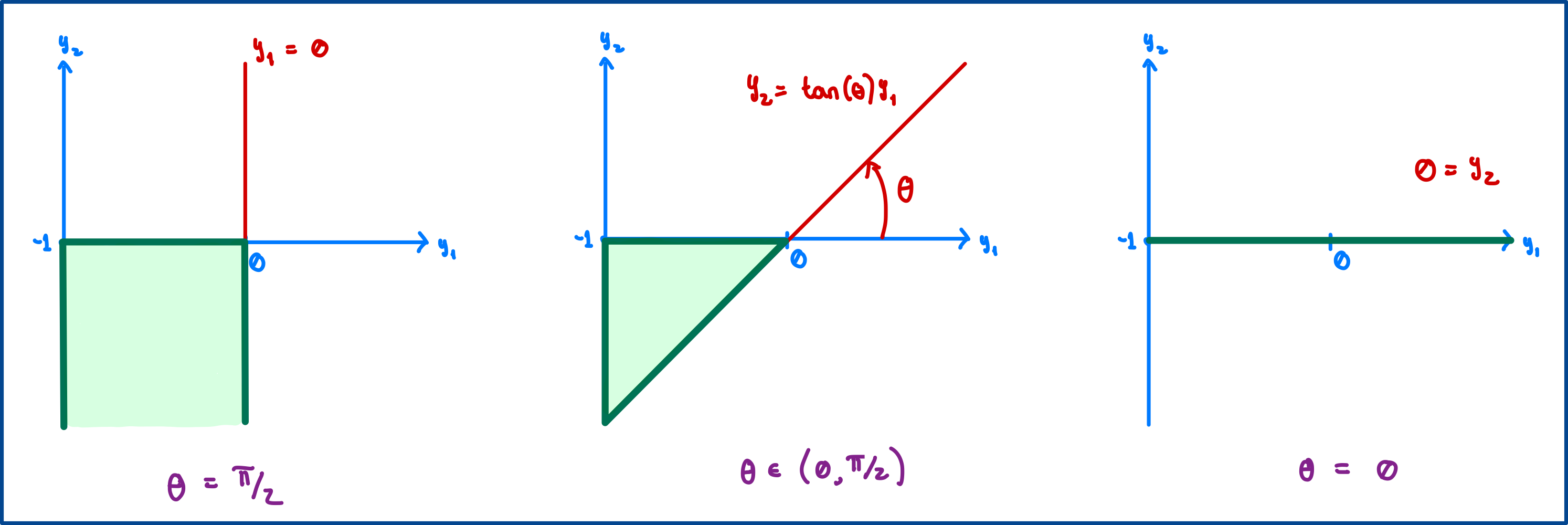}
    \caption{Illustration of set-valued map $F:[0,\pi/2]\tto \R^2$.}
    \label{ch03:fig:Constrained-Not-LSC}
\end{figure}

While $F$ us upper semicontinuous and even continuous on $(0,\pi/2]$, is is clear that it fails to be lower semicontinuous at $0$, due to the ``sudden explosion'' of the $y_1$-axis.
\end{example}

The next theorem shows that when all constrains are linear, the constrained map verifies lower semicontinuity. The proof relies on structural properties of linear programming, which we quickly recall in the following notes. The interested reader is referred to \citep*{Bertsimas1997Introduction}.

\begin{note}{Note: Existence of solutions for linear programming}
    Consider the problem
    \[
    \begin{cases}
        \displaystyle\min_{x\in \R^p}& \langle c,x\rangle\\
        \text{s.t} & Ax \leq b.
    \end{cases}
    \]
    Then, one and only one of the following statements hold:
    \begin{enumerate}
        \item The problem is infeasible (and the value of the problem is set as $+\infty$).
        \item The problem is unbounded (and the value of the problem is set as $-\infty$).
        \item The problem admits a solution $x^*$ (and the value of the problem is given by $\langle c,x^*\rangle$). 
    \end{enumerate}
    Moreover, if the problem admits a solution and the set $P=\{x\,:\, Ax\leq b\}$ has at least one extreme point, then $\langle c,\cdot\rangle$ attains its minimum at $\ext(P)$.
\end{note}

\begin{note}{Note: Representability of Polyhedra}
    Consider the polyhedron $P = \{x \in \R^p\, :\, Ax\leq b \}$. The \textbf{recession cone} of $P$ is given by
    \[
    \mathrm{rec}(P) = \{d\in\R^p\,:\, Ad\leq 0\}.
    \]
    Equivalently, the recession cone is given by all directions $d\in\R^p$ such that $P+d\subset P$. If $\ext(P)$ is nonempty, then we have that
    \[
    P = \co(\ext(P)) + \mathrm{rec}(P),
    \]
    Note that, since $\ext(P)$ is a finite set, then $\co(\ext(P))$ is compact.
\end{note}

\begin{note}{Note: Duality}
    Consider the (primal) problem
    \[
    (P) = \begin{cases}
        \displaystyle\min_{x\in \R^p}& \langle c,x\rangle\\
        \text{s.t} & Ax \geq b.
    \end{cases}
    \]
    Then, its associated dual is given by
    \[
    (D) = \begin{cases}
        \displaystyle\max_{y\in \R^m}& \langle b,y\rangle\\
        \text{s.t} & A^{\top}y = c,\, y\geq 0.
    \end{cases}
    \]
    Then, for any pair $(x,y)$ such that $x$ and $y$ are \textbf{feasible} for the primal and the dual respectively, one has 
    \[
    \langle b,y\rangle \leq \langle c,x\rangle.\tag{weak duality}
    \]
    Moreover, if $(x^*,y^*)$ are \textbf{optimal} solutions of the primal and the dual, respectively, then one has
    \[
    \langle b,y^*\rangle = \langle c,x^*\rangle.\tag{strong duality}
    \]
\end{note}

\begin{theorem}[Lower semicontinuity of linear constrained map]\label{ch03:thm:LSC-of-Linear}
    Let $F:X\subset \R^p\tto \R^q$ given by
    \[
    F(x) = \{ y\in\R^p\ :\ Ax + By \leq b \},
    \]
    where $A\in\R^{m\times p}$, $B \in \R^{m\times q}$ and $b\in\R^m$. Suppose that $X\subset \{x\in\R^p\ :\ \exists y\in \R^q,\, Ax+By\leq b\}$ (i.e., $F$ has nonempty values). Then, $F$ is lower semicontinuous.
\end{theorem}
\begin{proof}
   Reasoning by absurd, let us suppose that there exists $\bar{x}\in X$, $\bar{y}\in F(\bar{x})$ and a sequence $x_n\to \bar{x}$ such that
   \[
   \limsup_{n} d(\bar{y},F(x_n))>0.
   \]
   Up to a subsequence, we might assume that there exists $\delta>0$ small enough such that $d(\bar{y},F(x_n))>\delta$ for all $n\in\N$. By the separation theorem, for each $n\in\N$ there exists $u_n\in \Sph_q$ (where $\Sph_q$ denotes the unit sphere of $\R^q$) such that
   \[
   \sup_{z\in B(\bar{y},\delta)}\langle u_n,z\rangle = \langle u_n,\bar{y}\rangle + \delta < \underbrace{\inf\{ \langle u_n,z\rangle\,:\, z\in F(x_n)\}}_{=\nu_n \in \R}.
   \]
Let $\nu = \liminf_n \nu_n \in \overline{\R}$. Again, by taking a subsequence, we might assume that $(\nu_n)$ converges to $\nu$ in $\overline{\R}$. By compactness of $\Sph_q$, we can assume that $u_n\to \bar{u}\in \Sph_q$. Note that
\[
\langle \bar{u},\bar{y}\rangle + \delta = \lim_n \langle u_n,\bar{y}\rangle + \delta \leq \lim_n \nu_n.
\]
Now, let us consider the problems
\[
(P_n) = \begin{cases}
    \min_z &\langle u_n,z\rangle\\
    s.t. & Bz \leq b - Ax_n,
\end{cases}\quad\text{and}\quad (\bar{P}) = \begin{cases}
    \min_z &\langle \bar{u},z\rangle\\
    s.t. & Bz \leq b - A\bar{x},
\end{cases}
\]
Then, the optimal value of problem $(P_n)$ is $\nu_n$, and the optimal value of $(\bar{P})$, that we denote by $\bar{\nu}$, verifies that
\[
\bar{\nu} +\delta \leq \langle \bar{u},\bar{y}\rangle \leq \lim \nu_n.
\]
Now, let us look at the dual problems of $(P_n)$ and $(\bar{P})$, which are respectively
\[
(D_n) = \begin{cases}
    \max_w &\langle Ax_n-b,w\rangle\\
    s.t. & -B^{\top}w = u_n, w\geq 0,
\end{cases}\quad\text{and}\quad (\bar{D}) = \begin{cases}
    \max_w &\langle A\bar{x}-b,w\rangle\\
    s.t. & -B^{\top}w = \bar{u}, w\geq 0.
\end{cases}
\]
Strong duality implies that every $(D_n)$ admits at least one solution $w_n$, and it verifies that
$\langle b- Ax_n,w_n\rangle = \nu_n$. We would like to show that $(w_n)$ converges (up to a subsequence) to a point $\bar{w}$ that is a solution of $(\bar{D})$. The obstruction is that the sequence $(w_n)$ might be unbounded.

Let us consider the polyhedron 
\begin{align*}
  Q &= \{ (u,w)\ :\ B^{\top}u = w,\, w\geq 0,\, \|u\|_{\infty}\leq 1\}\\
  &= \{ (u,w)\ :\ B^{\top}u = w,\, w\geq 0,\, -1\leq u\leq 1\}.
\end{align*}

Note that the recession cone of $Q$ must verify the inclusion $\mathrm{rec}(Q)\subset \{0\}\times \R^m$. Indeed, if $d = (d_1,d_2)\in\mathrm{rec}(Q)$ the constraints $\|u\|_{\infty}\leq 1$ entail that $d_1 \leq 0$ and that $-d_1\leq 0$. Since both variables $u$ and $w$ are bounded from bellow ($w\geq 0$ and $u\geq -1$), then $\ext(Q)\neq\emptyset$. So,
\[
Q = \hat{Q} + \mathrm{rec}(Q),
\]
where $\hat{Q} = \cco(\ext(Q))$. Now, let us consider the set
\[
\mathcal{U} = \{ u \in \R^q\ :\ \exists w\in\R^m, (u,w)\in Q\}.
\]
Then, the sequence $(u_n)$ and the limit $u$ belong to $\mathcal{U}$. Moreover, each $(u_n,w_n)\in Q$ and therefore,
\[
w_n \in \hat{w}_n + d_n,
\]
where $(u_n,\hat{w}_n)\in \hat{Q}$ and $(0,d_n)\in \mathrm{rec}(Q)$. Now, since $\mathrm{rec}(Q)$ is a cone, necessarily we need that $\langle b-Ax_n,d_n\rangle\leq 0$. Otherwise, the problem $(D_n)$ would be unbounded, which would be a contradiction. Thus, $\hat{w}_n$ must be a solution of $(D_n)$ as well. Then, by replacing $(w_n)$ by $(\hat{w}_n)$ if necessary, and since $\hat{Q}$ is bounded (it is the convex hull of the finite set $\ext(Q)$), we get that $(w_n)$ is bounded.

Now, passing again by a subsequence if necessary, we can assume that $(u_n,w_n) \to (\bar{u},\bar{w})\in \hat{Q}$. Then, using strong duality
\[
\bar{\nu}\geq \langle A\bar{x}-b,\bar{w}\rangle = \lim_n \langle Ax_n-b,w_n\rangle = \lim\nu_n \geq \bar{\nu} + \delta.
\]
This is a contradiction, and so $F$ must be lower semicontinuous.
\end{proof}

\section{Existence of solutions in Bilevel Programming}

The existence of solutions of a bilevel programming problem (optimistic or pessimistic) can be reduced to study the continuity properties of functions $\varphi^o$ and $\varphi^p$, defined in formulations \eqref{ch02:eq:BilevelProgramming-Optimistic-phiO} and \eqref{ch02:eq:BilevelProgramming-Pessimistic-phip}, respectively. Indeed, recall that we consider
\begin{enumerate}
    \item The leader's objective function $\theta_l:\R^p\times\R^q\to \R$.
    \item The leader's feasible region $X\subset\R^p$.
    \item The follower's objective function $\theta_f:\R^p\times\R^q\to \R$.
    \item The follower's constraint set-valued map $K: X\tto \R^q$.
\end{enumerate}
We will consider now an \textbf{ambient space} $Y\subset\R^q$ for the follower's decision, verifying that $Y$ is closed and $K(X)\subset Y$. Thus, we will write $K:X\tto Y$ to explicit it. Then, the solution map $S:X\tto Y$ is given by $S(x) := \argmin_y\{ \theta_f(x,y)\,:\, y\in K(x)  \}$, and 
\begin{align*}
    \varphi^o(x) &= \inf_{y\in S(x)} \theta_l(x,y),\\
    \varphi^p(x) &= \sup_{y\in S(x)} \theta_l(x,y).
\end{align*}
Thus, the key observation is that we can apply Weierstrass theorem as follows:
\begin{enumerate}
    \item \textbf{Optimistic problem:} If $X$ is compact and $\varphi^o$ is lower semicontinuous, then the optimistic problem $\displaystyle\min_{x\in X}\varphi^o(x)$ (see Definition \ref{ch02:def:BilevelProgramming-Optimistic}) admits a solution.
    \item \textbf{Pessimistic problem:} If $X\cap \dom S$ is compact and $\varphi^p$ is lower semicontinuous, then the pessimistic problem $\displaystyle\min_{x\in X\cap\dom S}\varphi^o(x)$ (see Definition \ref{ch02:def:BilevelProgramming-Pessimistic}) admits a solution.
\end{enumerate}
For the optimistic problem, this observation lead us to the following theorem.

\begin{theorem}[Existence of solutions for Optimisitc Bilevel programming]\label{ch03:thm:Existence-Optimistic} Consider the optimistic bilevel programming problem
\[
\left\{\begin{array}{rl}
    \min_{x,y} & \theta_l(x,y)  \\
     s.t. & x\in X, y\in S(x),   
\end{array}\right.
\]
 where  $S(x) := \argmin_y\{ \theta_f(x,y)\,:\, y\in K(x)  \}$. Assume that 
 \begin{itemize}
     \item[(i)] $\theta_l$ is lower semicontinuous.
     \item[(ii)] $\theta_f$ is continuous semicontinuous.
     \item[(iii)] $X$ is nonempty and compact.
     \item[(iv)] $K:X\tto Y$ is both upper and lower semicontinuous, and has nonempty compact closed values.
 \end{itemize}
   Then, the problem admits a solution.
\end{theorem}
\begin{proof}
    Note that hypotheses $(ii)$ and $(iv)$ entail that the solution map $S:X\tto Y$ has nonempty compact values. Using the second form of Berge Maximum theorem (see Theorem \ref{ch03:thm:Berge-secondForm}, we get that $S:X\tto Y$ is upper semicontinuous, which entails that it is closed. Then,
    \[
    \gph S = \{ (x,y)\ :\ x\in X,\, y\in S(x)\}
    \]
    is closed. Now, let us show that $\gph S$ is in fact, compact. Let $(x_n,y_n)$ be a sequence in $\gph S$. Since $X$ is compact by hypothesis $(iii)$, there exists a convergent subsequence $(x_{n_k})_k$ of $(x_n)$, converging to some point $\bar{x}\in X$. Let $C = S(\bar{x}) + \Ball$. Since $S(\bar{x})$ is compact, we get that $C$ is compact as well. Moreover, due to the upper semicontinuity of $S$, we get that $y_{n_k}\in C$ for $k\in\N$ large enough. Thus, up to a subsequence, we can suppose that $y_{n_k}$ converges to a point $\bar{y} \in Y$. closedness of $\gph S$ yields that $(\bar{x},\bar{y})\in \gph S$. This proves that $\gph S$ is compact. Since $\theta_l$ is lower semicontinuous, the result follows.
\end{proof}

An alternative proof of Theorem  \ref{ch03:thm:Existence-Optimistic} is to note that 
\[
-\varphi^o(x) = \sup_{y\in S(x)} -\theta_l(x,y).
\]
Then, hypotheses $(ii)$ and $(iv)$ yields that $S:X\tto Y$ is upper semicontinuous with nonempty compact values, and hypothesis $(i)$ yields that $-\theta_l$ is upper semicontinuous. Then, we can apply the  first form of Berge maximum theorem (see Theorem \ref{ch03:thm:Berge-firstForm}) to deduce that $-\varphi^o$ is upper semicontinuous. Thus, $\varphi^o$ is lower semicontinuous and the compactness of $X$ leads to the conclusion.

Both proofs rely on the same intrinsic aspect of the proof: hypotheses $(ii)-(iii)-(iv)$ yield that the graph of $S$ is compact, and optimistic bilevel optimization can be reduced to optimize both $x$ and $y$ over the graph of $S$. 

\begin{warning}
    The most demanding hypothesis of Theorem \ref{ch03:thm:Existence-Optimistic} is the lower semicontinuity of the constraint map $K:X\tto Y$. However, as we already have discussed above, this holds in two canonical settings:
    \begin{enumerate}
        \item When $K$ is given as the slices of a compact polytope. That is, $X = \{x\, :\, A_l x\leq b_l\}$ and for each $x\in X$,  $K(x) = \{y\,:\, A_fx+B_fy\leq b_f\}$ where \[ 
        D = \{ (x,y) \,:\, A_l x\leq b_l, A_fx+B_fy\leq b_f\}
        \]
        is compact.
        \item When $K$ is constant. That is, $K(x) = Y$ for every $x\in X$.
    \end{enumerate}
    Other formulations would require to analyze the lower semicontinuity of $K$ carefully, or to find a way around to deduce the compactness of $\gph S$. 
\end{warning}

The above discussion does not longer holds in the case in pessimistic bilevel programming. Indeed, Example \ref{ch02:example:MultipleSolutionsFollower} already illustrates the difference: $\gph S$ might be compact and $\theta_l$ might be even continuous, but the pessimistic bilevel problem can fail to have any solutions. The key obstruction is hidden in the first form of Berge Maximum theorem (see Theorem \ref{ch03:thm:Berge-firstForm}). We need the function
\[
\varphi^p(x) = \sup_{y\in S(x)} \theta_l(x,y)
\]
to be lower semicontinuous which requires, at least in the theorem, \textbf{lower semicontinuity of the solution map $S$}. In what follows, we will present two settings where the desired semicontinuity holds: the linear case, and the regularized case. The following result seems to be traced back to \citep*{Lucchetti1987Existence}.

\begin{theorem}[Existence results for linear pessimistic bilevel programming]\label{ch03:thm:Existence-Pessimistic-Linear} Consider the pessimistic bilevel programming problem
    \[
    \min_{x\in X}  \left(\sup_{y\in S(x)}\theta_l(x,y)\right), 
\]
 where  $S(x) := \argmin_y\{ \theta_f(x,y)\,:\, y\in K(x)  \}$. Assume that 
 \begin{itemize}
     \item[(i)] $\theta_l$ is lower semicontinuous.
     \item[(ii)] $\theta_f$ is linear. That is, $\theta_f(x,y) = \langle c,y\rangle$.
     \item[(iii)] $X$ is nonempty and $\dom K = X$.
     \item[(iv)] $K:X\tto Y$ is given by $K(x) = \{y\,:\, Ax + By \leq b\}$, and $\gph K$ is compact.
 \end{itemize}
   Then, the problem admits a solution.
\end{theorem}
\begin{proof}
    Note that hypothesis $(iii)$ and $(iv)$ entail that $X$ is compact. Thus, it is enough to show that $\varphi^p(x) = \sup_{y\in S(x)}\theta_l(x,y)$ is lower semicontinuous. Thanks to the first form of Berge maximum theorem (see Theorem \ref{ch03:thm:Berge-firstForm}) and hypothesis $(i)$, it is enough to show that $S$ is lower semicontinuous. Suppose the contrary. By Theorem~\ref{ch03:thm:seq-LSC}, there exists $x_0\in X$ and $y_0\in S(x_0)$ such that
    \[
    \limsup_{x\to x_0} d(y_0,S(x)) >0.
    \]
    In particular, there exist $\delta>0$ and a sequence $(x_n)\subset X$ converging to $x_0$ such that $d(y_0, S(x_n))\geq \delta$ for all $n\in \N$.
    
    First, note that for every $x\in X$, we can describe $S(x)$ as
    \[
    S(x) = \{ y\,:\, \langle c_i,y\rangle \leq \varphi_i(x), i=0,\ldots,m \},
    \]
    with $\varphi_i:X\to \R$ continuous for every $i = 0,\ldots,m$. Indeed, for $i=1,\ldots,m$, it is enough to take $c_i = B_{i\bullet}^{\top}$ and $\varphi_i(x) = (b- Ax)_i$. For $i=0$, we set $c_i = c$ and $\varphi_0(x) = \min_{y}\{\langle c,y\rangle\,:\, y\in K(x)\}$. The continuity of $\varphi_0$ follows from the second form of Berge maximum theorem (see Theorem \ref{ch03:thm:Berge-secondForm}) and the fact that $K$ is continuous by hypothesis $(iv)$ (readily applying Theorem \ref{ch03:thm:LSC-of-Linear}, Proposition \ref{ch03:prop:closednessConstrained} and Corollary \ref{ch03:cor:USC-ClosedGraph}).

    Observe that, since the set $\{ y\,:\, \limsup_{x\to x_0} d(y_0,S(x)) =0\}$ is closed and $S(x_0)$ is convex, we can assume without losing generality that $y_0 \in \rint(S(x_0))$.  Let us now define the active index set 
    \[
    \mathcal{A} = \mathcal{A}(x_0,y_0) := \{i\,:\, \langle c_i,y_0\rangle = \varphi_i(x_0)\}.
    \]
    Note that for all $y\in \rint(S(x_0))$, we have that $\mathcal{A}(x_0,y) = \mathcal{A}(x_0,y_0)$.
    
    Since $S$ takes nonempty values, we can choose $y_n\in S(x_n)$ for each $n\in \N$. Up to a subsequence, we can assume that $y_n\to \bar{y}$. Clearly, $d(y,y_0) = \lim_n d(y_n,\bar{y}) \geq \delta$. Noting that $S$ has closed graph (see Proposition \ref{ch03:prop:closednessConstrained}), we know that $\bar{y}\in S(x_0)$. Now, take $\hat{y}_n = y_n - \bar{y} + y_0$. Clearly, $\hat{y}_n \to y_0$. Now, take $i\in\{0,\ldots,m\}$:
    \begin{itemize}
        \item If $i\in \mathcal{A}$, then $\langle c_i, y_0-\bar{y}\rangle = 0$. Indeed, the segment $[y_0, \bar{y})$ is contained in $\rint(S(x_0))$, and so the the function $y\mapsto \langle c_i,y\rangle$ is constant in $[y_0, \bar{y})$. Continuity yields that $\langle c_i, y_0\rangle = \langle c_i, \bar{y}\rangle$. In particular, 
        \[
        \langle c_i, \hat{y}_n\rangle = \langle c_i, y_n\rangle + \underbrace{\langle c_i, y_0-\bar{y}\rangle}_{=0} \leq \varphi_i(x_n).
        \]
        \item If $i\notin \mathcal{A}$, then the function $(x,y)\mapsto \langle c_i,y\rangle - \varphi_i(x)$ is strictly negative at $(x_0,y_0)$. Since $(x_n,\hat{y}_n)\to (x_0,y_0)$, continuity yields that there exists $n_i\in \N$ large enough such that
        \[
        \langle c_i,\hat{y}_n\rangle < \varphi_i(x_n),\quad \forall n\geq n_i.
        \]
    \end{itemize}
    By taking $N = \max\{ n_i\, :\, i\notin \mathcal{A}\}$ (with $N= 1$ if $\mathcal{A} = \{0,\ldots,m\}$), we deduce that
    \[
    \forall n\geq N,\forall i\in\{0,\ldots,m\},\quad \langle c_i,\hat{y}_n\rangle \leq \varphi_i(x_n).
    \]
    That is, for every $n\geq N$, $\hat{y}_n\in S(x_n)$. This yields that 
    \[
    0<\delta \leq \lim d(y_0, S(x_n)) \leq \lim_n d(y_0,\hat{y}_n) = 0,
    \]
    which is a contradiction. The proof is then completed.
\end{proof}

When the exact solution map $S:X\tto Y$ is not lower semicontinuous, one can study a relaxed version of the problem, by admitting $\varepsilon$-optimal solutions. Namely, for a fixed tolerance $\varepsilon>0$, one considers the set-valued map $S_{\varepsilon}:X\tto Y$ given by
\begin{align*}
    S_{\varepsilon}(x) &= \varepsilon\text{-}\argmin_y\{ \theta_f(x,y)\,:\, y\in K(x)\}\\
    &:=\{y\in K(x) \,:\, \theta_f(x,y)\leq \theta_f(x,z) + \varepsilon,\,\forall z\in K(x)\}.
\end{align*}
Usually, $S_{\varepsilon}$ enjoys better continuity properties than $S$.
\begin{example}[Regularizing effect of approximate solution map]\label{ch03:ex:Regular-eps-Argmin} Consider again the solution set of Example \ref{ch02:example:MultipleSolutionsFollower}, that is,
\[
S(x) = \argmin_y\{ -xy\, :\, y\in [0,1]  \} = \begin{cases}
    \{0\}&\text{ if }x<0,\\
    [0,1]&\text{ if } x= 0,\\
    \{1\}&\text{ if } x> 0.
\end{cases}
\]
Fix $\varepsilon>0$. Then, if $x<0$ one has that for $y\in [0,1]$
\[
y\in S_{\varepsilon}(x) \iff -xy \leq \varepsilon \iff y\leq -\frac{\varepsilon}{x}.
\]
Similarly, if $x>0$ one has that for $y\in [0,1]$
\[
y\in S_{\varepsilon}(x) \iff -xy \leq -x+\varepsilon \iff y\geq 1-\frac{\varepsilon}{x}.
\]
Thus,
\[
S_{\varepsilon}(x) = \begin{cases}
    \left[0, \min\left\{1,-\frac{\varepsilon}{x}\right\}\right]&\text{ if }x<0,\\
    [0,1]&\text{ if } x= 0,\\
    \left[\max\left\{0,1-\frac{\varepsilon}{x}\right\},1\right]&\text{ if } x> 0.
\end{cases}
\]
\begin{figure}[ht]
    \centering
    \includegraphics[width=0.9\linewidth]{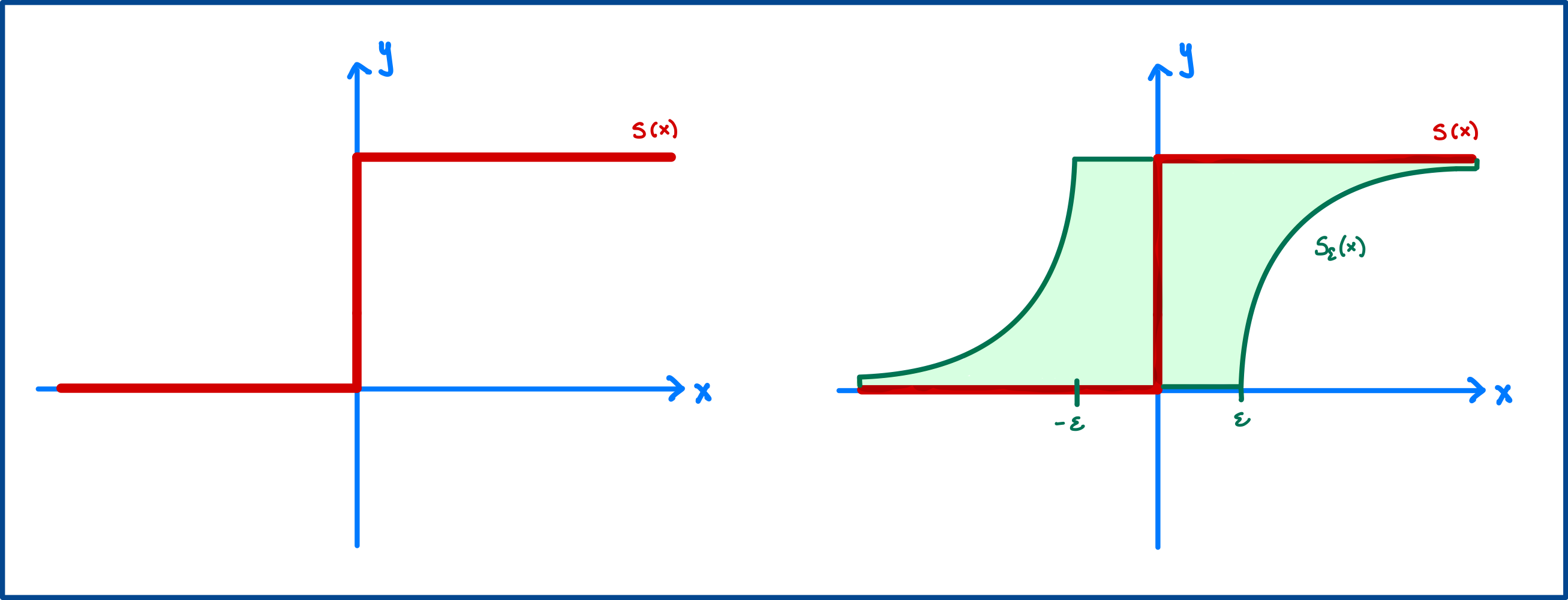}
    \caption{Illustration of $S(x)$ versus $S_{\varepsilon}(x)$.}
    \label{ch03:fig:eps-argmin}
\end{figure}
The graph of $S$ and $S_{\varepsilon}$ are illustrated in Figure \ref{ch03:fig:eps-argmin}. We now can define the functions
\[
u(x) = \begin{cases}
    -\frac{\varepsilon}{x} &\text{ if }x\leq -\varepsilon,\\
    1&\text{ if } x>-\varepsilon,
\end{cases}\quad\text{ and }\quad\ell(x) = \begin{cases}
    0 &\text{ if }x\leq \varepsilon,\\
    1- \frac{\varepsilon}{x}&\text{ if } x>\varepsilon.
\end{cases}
\]
Then, $u$ and $\ell$ are continuous, $\ell\leq u$, and $S_{\varepsilon}(x) = [\ell(x),u(x)]$. We deduce that $S_{\varepsilon}$ is both, upper and lower semicontinuous (see Problem \ref{ch03:problem:Semicontinuity-BetweenTwoFunctions}).
\end{example}

As hinted by Example \ref{ch03:ex:Regular-eps-Argmin}, using $S_{\varepsilon}:X\tto Y$ instead of $S:X\tto Y$ allows us to deduce a new existence result for a broader class of pessimistic problems.

\begin{theorem}[Existence result for regularized pessimistic bilevel programming] 
    Fix $\varepsilon>0$ and consider the pessimistic bilevel programming problem
    \[
    \min_{x\in X}  \left(\sup_{y\in S_{\varepsilon}(x)}\theta_l(x,y)\right), 
\]
 where  $S_{\varepsilon}(x) := \varepsilon\text{-}\argmin_y\{ \theta_f(x,y)\,:\, y\in K(x)  \}$. Assume that 
 \begin{itemize}
     \item[(i)] $\theta_l$ is lower semicontinuous.
     \item[(ii)] $K:X\tto Y$ is both upper and lower semicontinuous, and it has nonempty convex values.
     \item[(iii)] $\theta_l$ is continuous and for every $x\in X$, $\theta_l(x,\cdot)$ is strictly convex.  
     \item[(iv)] $\gph K$ is compact and $X=\dom K$.
 \end{itemize}
   Then, $S_{\varepsilon}$ is lower semicontinuous, and the problem admits a solution.
\end{theorem}
\begin{proof}
    Again, if $S_{\varepsilon}$ is lower semicontinuous, the existence of solutions follows directly from compactness of $X$ (due to $(iv)$) and lower semicontinuity of $\varphi^p(x) = \sup_{y\in S_{\varepsilon}(x)}\theta_l(x,y)$ (due to $(i)$ and the first form of Berge Maximum theorem, cf. Theorem \ref{ch03:thm:Berge-firstForm}).\\

    Now, let us fix $x_0\in X$ and let $C = \{y\in Y\,:\, \limsup_{x\to x_0} d(y,S_{\varepsilon}(x)) = 0 \}$. In view of Theorem \ref{ch03:thm:seq-LSC}, we need to show that
    \(
    S_{\varepsilon}(x_0) \subset C.
    \)
    Note that since $S_{\varepsilon}(x_0)$ is convex, one has that $S_{\varepsilon}(x) = \overline{\rint S_{\varepsilon}(x)}$. Then, since $C$ is closed, it is enough to show that $\rint(S_{\varepsilon}(x_0)) \subset C$.\\
    
    Choose $y_0\in \rint(S_{\varepsilon}(x_0))$, and let $\phi(x) = \min_{y\in K(x)}\theta_{f}(x,y)$. We distinguish two situations:
    \begin{itemize}
        \item $K(x_0)$ is a singleton with $K(x_0) = \{y_0\}$. In such a case, $\rint(S_{\varepsilon}(x_0)) = \{y_0\}$ and $\theta_f(x_0,y_0) = \phi(x_0) < \phi(x_0) + \varepsilon$.
        \item $K(x_0)$ is not a singleton. Then, $S_{\varepsilon}(x_0)$ is not a singleton and therefore $\rint(S_{\varepsilon}(x))\cap \ext(S_{\varepsilon}(x)) = \emptyset$. Then there exist $y_0^+,y_0^\in S_{\varepsilon}(x_0)$ such that $y_0 = \tfrac{1}{2}y_0^+ + \tfrac{1}{2}y_0^-$. Using the strict convexity of $\theta_f(x_0,\cdot)$, we get that
        \[
        \theta_f(x_0,y_0) < \tfrac{1}{2}\theta_f(x_0,y_0^+) + \tfrac{1}{2}\theta_f(x_0,y_0^-) \leq \tfrac{1}{2}(\phi(x_0)+\varepsilon)+\tfrac{1}{2}(\phi(x_0)+\varepsilon) = \phi(x_0) + \varepsilon.
        \]
    \end{itemize}
    In either case, $\theta_f(x_0,y_0) < \phi(x_0) + \varepsilon$. Using the second form of Berge Maximum theorem (Theorem \ref{ch03:thm:Berge-secondForm}), we get that $\phi$ is continuous. Therefore, $(x,y)\mapsto\theta_f(x,y) - \phi(x)$ is also continuous and so, the inequality $\theta_f(x_0,y_0) - \phi(x_0) < \varepsilon$ hold in a nieghborhood. That is, there exists a neighborhood $U\times V$ of $(x_0,y_0)$ such that
    \[
    \theta_f(x,y) < \phi(x) + \varepsilon, \quad \forall (x,y) \in U\times V.
    \]
    Now, let $(x_n)$ be a sequence converging to $x_0$ such that 
    \[
    \limsup_{x\to x_0}d(y_0,S_{\varepsilon}(x)) = \lim_n d(y_0,S_{\varepsilon}(x_n)).
    \]
    Since $K$ is itself lower semicontinuous, we can find a sequence $y_n\to y_0$ with $y_n\in K(x_n)$ for every $n\geq n_0$ large enough. Then, by enlarging $n_0$ if necessary, we can assume that $(x_n,y_n)\in U\times V$ for every $n\geq n_0$. Then,
    \[
    \forall n\geq n_0, \theta_f(x_n,y_n) < \phi(x_n) + \varepsilon,
    \]
    and so $y_n\in S_{\varepsilon}(x_n)$ for each $n\geq n_0$. This yields that $d(y_0,S_{\varepsilon}(x_n)) \leq d(y_0,y_n)\to 0$, proving that $y_0\in C$. The proof is then finished.
\end{proof}

\newpage
\section{Problems}\label{ch03:sec:Problems}

\begin{problem}\label{ch03:problem:Semicontinuity-BetweenTwoFunctions} Let $f,g: X\subset \R\to\R$ two functions with $f\leq g$, and consider the set-valued map $F:X\tto \R$ given by
\[
F(x) = \{ y\in \R\,:\, f(x)\leq y\leq g(x) \}.
\]
Prove that, for $x_0\in X$ one has that
\begin{enumerate}
    \item $F$ is upper semicontinuous at $x_0$ $\iff$ $f$ is lower semicontinuous at $x_0$ and $g$ is upper semicontinuous at $x_0$.
    \item $F$ is lower semicontinuous at $x_0$ $\iff$ $f$ is upper semicontinuous at $x_0$ and $g$ is lower semicontinuous at $x_0$.
\end{enumerate}
\end{problem}

\begin{problem}\label{ch03:problem:Closedness-semilimits-distances} Let $X$ and $Y$ be two metric spaces and $M:X\tto Y$ be any set-valued map. Fix $x_0\in X$. Prove that
\begin{enumerate}
    \item $A = \{y\,:\, \limsup_{x\to x_0} d(y,M(x)) = 0 \}$ is closed.
    \item $B = \{y\,:\, \liminf_{x\to x_0} d(y,M(x)) = 0 \}$ is closed.
\end{enumerate}
\end{problem}
\chapter{Reformulations and algorithms}
\label{Chapter04:Algorithms}

In this chapter we focus our attention on optimistic bilevel programming problems, and how to produce algorithm to solve them. For fairly general problems, we will present the principal tool we have: one-level reformulations. The idea is to replace the bilevel problem for an equivalent one-level problem, in the sense that they have ``the same'' solutions. Of course, such an approach requires some regularity, mainly on the lower-level.

\section{Karush-Kuhn-Tucker (KKT) conditions}
\label{ch04:sec:KKT}
Before we start, let us quickly recall the main aspects of Karush-Kuhn-Tucker (KKT) conditions in smooth optimization. To ease the exposition and maintaining the spirit of a quick reminder, we will omit the proofs of this section. 

Let us consider a (single-level) optimization problem of the form
\begin{equation}\label{ch04:eq:GeneralSmooth-OptProblem}
\min_{y\in Y} f(y)\quad\text{ where }\quad Y = \left\{ y\in\R^q\,:\,\begin{array}{ll}
    h_i(y) = 0 & \forall i\in I \\
    g_j(y)\leq 0 & \forall j\in J 
\end{array} \right\}.
\end{equation}
Here, $I$ and $J$ are finite (possibly empty) index sets, and $h_i:\R^q\to\R$ and $g_j:\R^q\to\R$ are assumed to be of class $\mathcal{C}^1$, for each $i\in I$ and each $j\in J$. One can define the \textbf{tangent cone} and the \textbf{normal cone} to $Y$ at a point $y\in Y$ as
\begin{align}
    T_Y(y) &:= \left\{ \nu\in \R^q\,:\, \begin{array}{l}\exists \nu_k\to \nu, t_k\to 0^+,\text{ such that }\label{ch04:eq:TangentCone}\\
    y+t_k\nu_k \in Y,\text{ for each } k\in\N.
    \end{array}\right\}\\[2em]
    N_Y(y) &:= [T_Y(y)]^{\circ} = \big\{ \eta\in\R^q\, :\, \langle \eta,\nu\rangle \leq 0, \forall \nu\in T_Y(y) \big\}.\label{ch04:eq:NormalCone}
\end{align}
There are many other notions of tangent cones and normal cones (see, e.g. \citep*{Thibault2023Unilateral}). Here, we are presenting $T_Y(y)$ to be the \textit{Bouligand tangent cone}, and $N_Y(y)$ to be its polar cone, which we know it coincides (in finite dimensions) with the \textit{Fr\'echet normal cone} (see, e.g. \citep*{Kruger2009Nonsmooth}). Both notions are illustrated in Figure~\ref{ch04:fig:TangentAndNormal}.

\begin{figure}[ht]
    \centering
    \includegraphics[width=0.7\linewidth]{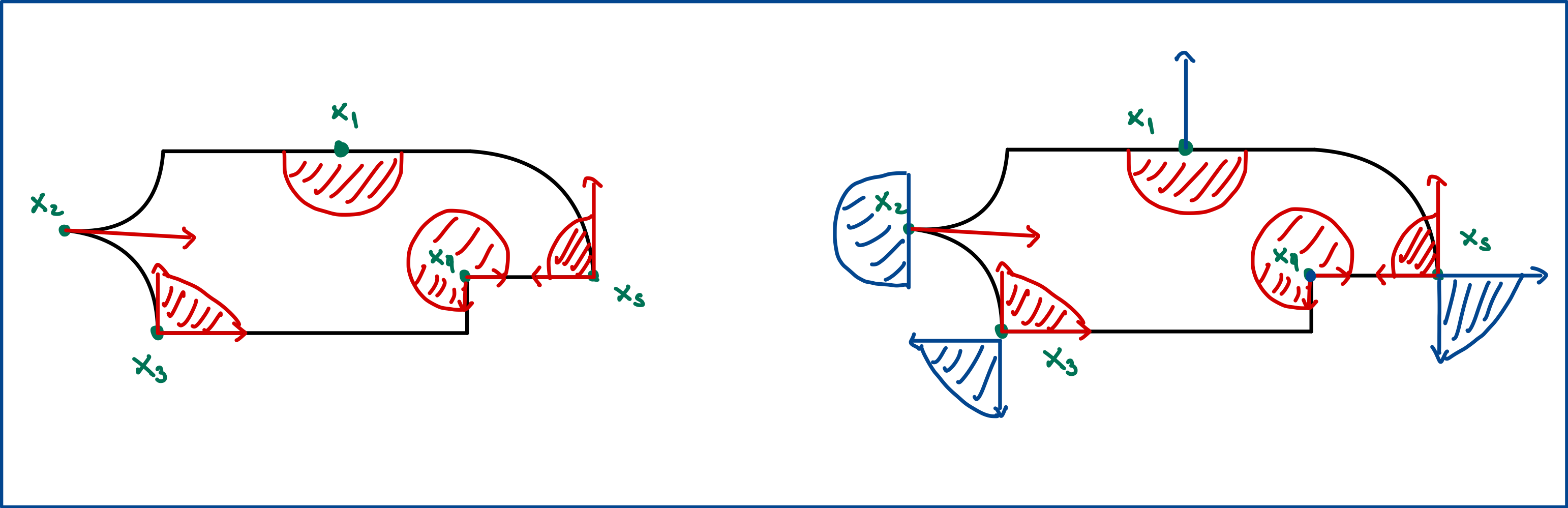}
    \caption{Tangent and Normal cones at different points. Left: Only tangent cones. Right: Tangents and Normals.}
    \label{ch04:fig:TangentAndNormal}
\end{figure}

With these notions, we can establish a fairly general proposition about first-order necessary conditions in smooth constrained optimization. The idea is simple: the normal cone $N_Y(y)$ encodes the directions that are fully ``leaving'' the set $Y$ at the given point $y\in Y$. If $y\in Y$ is a local minima for $f$, then all descent directions of $f$ at $y$ should be ``leaving'' the set $Y$. This is captured by the inclusion $-\nabla f(y)\in N_Y(y)$.

\begin{proposition}\label{ch04:prop:First-OrderNC} Consider the problem \eqref{ch04:eq:GeneralSmooth-OptProblem} and let $\bar{y}\in Y$. Assume that the objective function $f$ is differentiable at $\bar{y}$. Then,
\[
\bar{y}\text{ is a local optima of \eqref{ch04:eq:GeneralSmooth-OptProblem}}\implies -\nabla f(\bar{y}) \in N_Y(\bar{y}).
\]
Moreover, if \eqref{ch04:eq:GeneralSmooth-OptProblem} is a convex problem (i.e., $f$ convex, $g_j$ convex for all $j\in J$, $h_i$ affine for all $i\in I$), then 
\[
-\nabla f(\bar{y}) \in N_Y(\bar{y}) \iff \bar{y}\in \argmin_{y\in Y} f(y).
\]
\end{proposition}

Now, computing the tangent cone $T_Y(y)$ and verifying the inclusion $-\nabla f(y) \in N_Y(y)$ can be very difficult using the definitions \eqref{ch04:eq:TangentCone} and \eqref{ch04:eq:NormalCone}. However, if we have access to the derivatives of the constraint functions, we can use them to try to provide an approximation of the objects. 
\begin{definition}\label{ch04:def:LinearizedTangent}Consider problem \eqref{ch04:eq:GeneralSmooth-OptProblem} and let $y\in Y$. Let $\mathcal{A}(y)\subset J$ be the set of active constraints at $y$, that is,
\[
\mathcal{A}(y) := \left\{ j\in J\ :\ g_j(y) = 0 \right\}.
\]
We define the \textit{linearized tangent cone} of $Y$ at the point $y$ as
\begin{equation}\label{ch04:eq:LinearizedTangent-Equation}
    L_{Y}(y) = \left\{ \nu\in\R^q\, :\, \begin{array}{ll}
         \langle \nabla h_i(y),\nu\rangle = 0,&\forall i\in I  \\
         \langle \nabla g_j(y),\nu\rangle \leq 0,&\forall j\in \mathcal{A}(y) 
    \end{array} \right\}
\end{equation} 
\end{definition}

\begin{note}{Note: Tangents and normal are local}
    As Figure~\ref{ch04:fig:TangentAndNormal} illustrates, tangent and normal cones are local concept: they only depend on what happens near the point. This is why in Definition~\ref{ch04:def:LinearizedTangent} it is important to consider only active constraints.\\
    
    If at a given point $y\in Y$ the constraint $[g_j\leq 0]$ is not active (i.e., $g_j(y)<0$), then locally it is not possible to distinguish between the set $[g_j < 0]$ and the whole space. Thus, $T_Y(y)$ is the same whether we consider the constraint $g_j$ or not. Consistently, $L_Y(y)$ must dismiss $g_j$ to correctly approximate $T_Y(y)$. 
\end{note}

The main property of the Linearized cone is that we can explicitly compute its polar. It is a mild consequence of Farkas' Lemma (see, e.g., \cite[Theorem 4.6]{Bertsimas1997Introduction}). Using this computation, we can write the following lemma.

\begin{lemma}\label{ch04:lemma:Polar-Linearized} Consider the problem \eqref{ch04:eq:GeneralSmooth-OptProblem} and let $y\in Y$. If $T_y(y) = L_Y(y)$, then
    \[
    N_Y(y) = \left\{\sum_{i\in I}\lambda_i\nabla h_i(y) +\sum_{j\in \mathcal{A}(y)}\mu_i\nabla g_j(y)\,:\, \lambda\in\R^{I}, \mu\in\R_+^{J} \right\}.
    \]
\end{lemma}

\begin{note}{Intuition of the lemma}
   The set $Y$ is nonlinear. However, at a given point $y\in Y$ we could consider a linear approximation of $Y$ by linearizing the active constraints. This gives us the set
   \[
   \tilde{Y} = \left\{ z\in\R^q\,:\,\begin{array}{ll}
    h_i(y) + \langle\nabla h_i(y),z-y\rangle = 0, & \forall i\in I \\
    g_j(y) + \langle\nabla g_j(y),z-y\rangle \leq 0, & \forall j\in \mathcal{A}(y) 
\end{array} \right\}. 
   \]
   The tangent cone of $\tilde{Y}$ is exactly $L_Y(y)$. The normal cone of $L_Y(y)$ is exactly the linear combination of the gradients of the linear constraints defining $\tilde{Y}$, with the detail that coefficients of inequality constraints must be positive (to consider only outgoing directions for the constraints $[g_j\leq 0]$). This intuition is illustrated in Figure~\ref{ch04:fig:LinearizationOfSet}.
\end{note}

\begin{figure}[ht]
    \centering
    \includegraphics[width=0.7\linewidth]{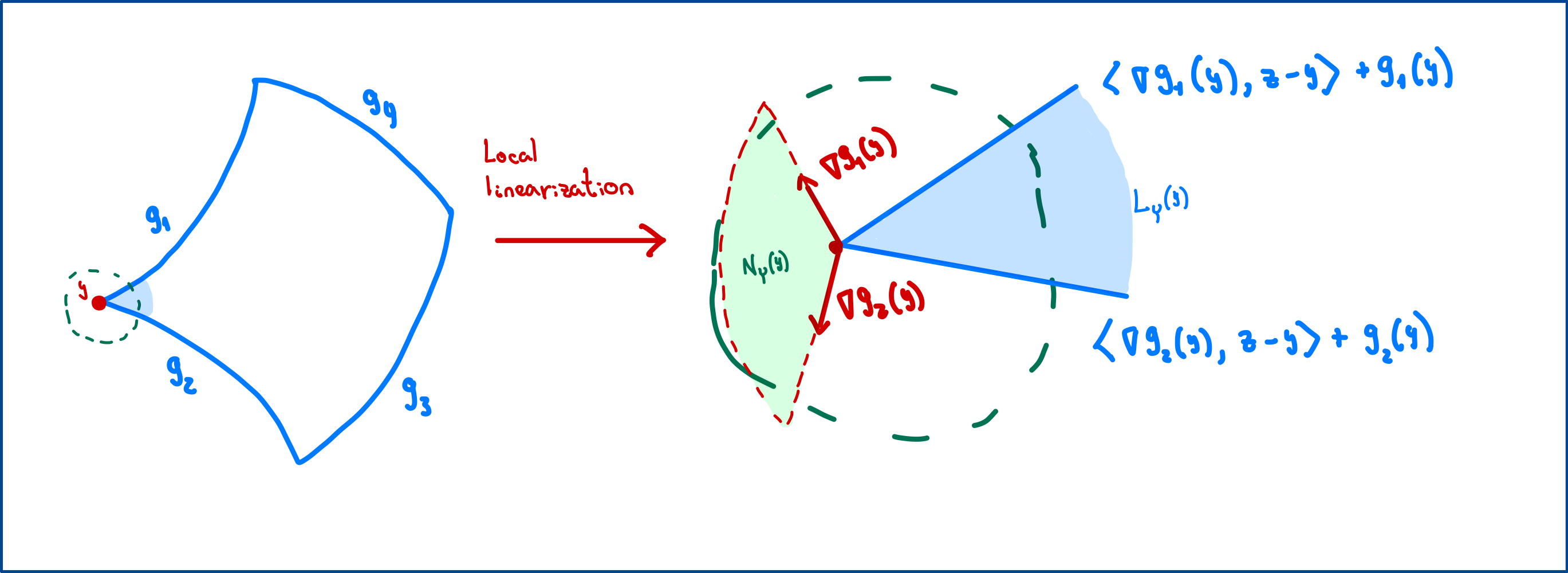}
    \caption{The approximation of the set $Y$ to a linear set $\tilde{Y}$ near a point $y$. $y + L_Y(\bar{y})$ coincides with $\tilde{Y}$, near $\bar{y}$. If $T_Y(y) = L_Y(y)$, the normal cone $N_Y(y)$ is obtained as the polar of $L_Y(y)$.}
    \label{ch04:fig:LinearizationOfSet}
\end{figure}

\begin{theorem}[KKT equations]\label{ch04:thm:KKTEquations} Consider problem \eqref{ch04:eq:GeneralSmooth-OptProblem} and let $y^*\in Y$ be a local minima of $f$ over $Y$. If $T_Y(y^*) = L_Y(y^*)$ then there exist multipliers $(\lambda,\mu)\in\R^I\times\R^J$ veryfing that
\begin{align*}
    \nabla f(y^*) + \sum_{i\in I} \nabla h_i(y^*) + \sum_{j\in J}\mu_j\nabla g_j(y^*) &= 0,\\
    \forall j\in J, \mu_j&\geq 0,\\[1em]
    \forall j\in J, \mu_j g_j(y^*) &= 0.
\end{align*}
\end{theorem}
\begin{proof} From Proposition \ref{ch04:prop:First-OrderNC}, we have that $-\nabla f(y^*) \in \N_{Y}(y^*)$. Since $T_Y(y^*) = L_Y(y^*)$, there exists $(\lambda_i\, :\, i\in I)$ and $(\hat{\mu}_j\, :\, j\in \mathcal{A}(y^*))$ such that
\[
    -\nabla f(y^*) = \sum_{i\in I} \nabla h_i(y^*) + \sum_{j\in \mathcal{A}(y^*)}\mu_j\nabla g_j(y^*),\quad\text{ and }\quad
    \forall j\in \mathcal{A}(y^*), \hat{\mu}_j\geq 0.
\]
We can extend the vector $\hat{\mu}\in\R_+^{\mathcal{A}(y^*)}$ to $\mu \in\R_+^J$ simply by setting
\[
\mu_j = \begin{cases}
    \hat{\mu}_j\quad&\text{ if }j\in \mathcal{A}(y^*),\\
    0\quad& \text{ otherwise}.
\end{cases}
\]
Then, it is not hard to verify that $\mu_jg_j(y^*) = 0$ for all $j\in J$, and so, the result follows.
\end{proof}

\begin{formulation}{How to use KKT equations} 
    Theorem \ref{ch04:thm:KKTEquations} is commonly use as follows: We first show that $T_Y(y) = L_Y(y)$ at every point $y\in Y$ so the theorem is valid for every potential local minima. Then, we solve the KKT equations: Find all tuples $(y,\lambda,\mu)\in\R^q\times\R^I\times \R^J$ verifying that
    \begin{align*}
    \nabla f(y) + \sum_{i\in I} \nabla h_i(y) + \sum_{j\in J}\mu_j\nabla g_j(y) &= 0,\\
    \forall j\in J, \mu_j&\geq 0,\\[1em]
    \forall j\in J, \mu_j g_j(y) &= 0,\\[1em]
    \forall i\in I, h_j(y) &= 0,\\[1em]
    \forall j\in J, g_j(y) &\leq 0.
\end{align*}
The last two equations are included so the tuples $(y,\lambda,\mu)$ satisfying the KKT system verify that $y\in Y$. Since the hypotheses of Theorem \ref{ch04:thm:KKTEquations} hold at every point, we can ensure that all local minima (including global minima) can be found among the first coordinate of tuples $(y,\lambda,\mu)$.
\end{formulation}

A set $Y$ together with its representation with equalities/inequalities as in \eqref{ch04:eq:GeneralSmooth-OptProblem}, that verifies $T_Y(y) =L_Y(y)$ is said to be \emph{constraint qualified}. Similarly, an optimization problem is \emph{constraint qualified} if its feasible set together with its representation, is constraint qualified.  How easy or hard is to verify that a set is constraint qualified? In general, it can get tricky to verify the condition $T_Y(y) = L_Y(y)$ at every point. However, one can apply sufficient conditions to obtain constraint qualifications (see, e.g., \citep*{solodov2010constraint}). The following proposition surveys the most classic constraint qualifications.

\begin{proposition}\label{ch04:prop:ConstraintQualifications}
   Let $Y\subseteq \R^q$ given by
   \[
   Y = \left\{ y\in\R^q\,:\,\begin{array}{ll}
    h_i(y) = 0 & \forall i\in I \\
    g_j(y)\leq 0 & \forall j\in J 
\end{array} \right\},
   \]
   and let $y\in Y$. Consider the following conditions:
   \begin{itemize}\setlength{\itemsep}{0.3cm}
       \item\textbf{(LCQ):} The functions $h_i$ and $g_j$ are affine for every $i\in I$, $j\in J$.
       \item\textbf{(LICQ):} The set $\{\nabla h_i(y)\,:\, i\in I\}\cup \{\nabla g_j(y)\,:\, j\in \mathcal{A}(y)\}$ is linearly independent.
       \item\textbf{(MFCQ)}: The set $\{\nabla h_i(y)\,:\, i\in I\}$ is linearly independent and there exists $\nu\in\R^q$ such that
       \begin{align*}
           \langle \nabla h_i(y),\nu\rangle = 0,& \forall i \in I,\\
           \langle \nabla g_j(y),\nu\rangle < 0,& \forall j \in \mathcal{A}(y).
       \end{align*}
       \item \textbf{(Slater)}: The functions $h_i$ are affine for every $i\in I$, the functions $g_j$ are convex for every $j\in J$ and there exists $y^*\in Y$ (in principle, different from $y$) such that
       \(
        g_j(y^*) < 0, \text{ for all }j\in J.
       \)
   \end{itemize}
   If any of the conditions above are verified, then $T_Y(y) = L_Y(y)$.
\end{proposition}

\begin{note}{Abadi CQ}
    In fact, a set $Y$ is said to be constraint qualified at a point $y\in Y$, if
    \[
    N_Y(y) = \left\{\sum_{i\in I}\lambda_i\nabla h_i(y) +\sum_{j\in \mathcal{A}(y)}\mu_i\nabla g_j(y)\,:\, \lambda\in\R^{I}, \mu\in\R_+^{J} \right\}.
    \]
    The reader can appreciate that in fact is this the minimal condition to deduce KKT theorem. The condition $T_Y(y) = L_Y(y)$ is then a \emph{sufficient condition} to have constraint qualification, known as Abadi CQ. For the sake for the exposition, we consider Abadi CQ as the starting point, since it has a clear geometric interpretation. It is worth to mention that the examples where the feasible set is constraint qualified but Abadi CQ doesn't hold are quite pathological.
\end{note}
\section{Single-level reformulation for Optimistic problems}
\label{ch04:sec:Reformulations}

Let us come back to optimistic bilevel programming, but now, assuming that for every leader's decision, the lower-level problem is given as in \eqref{ch04:eq:GeneralSmooth-OptProblem}. That is, we study the problem
\begin{equation}\label{ch04:eq:Optimistic-LowerLevelFunctionalConstraints}
    \begin{array}{rl}
        \displaystyle\min_{x,y} & \theta_l(x,y)  \\
         s.t. &\begin{cases}
             x\in X,\\
             y\text{ solves }\left\{\begin{array}{rl}
                 \displaystyle\min_{y} & \theta_f(x,y)  \\
                  s.t. & \begin{cases}
                      h_i(x,y) = 0,&\quad\forall i\in I,\\
                      g_j(x,y) \leq 0,&\quad \forall j\in J.
                  \end{cases}
             \end{array}\right.
         \end{cases} 
    \end{array}
\end{equation}
Here, the constraint map $K:X\tto Y$ is given by
\begin{equation}\label{ch04:eq:FunctionalConstraintsMap}
K(x) = \left\{y\in \R^q\, :\, \begin{array}{rl}
        h_i(x,y) = 0,&\,\forall i\in I,\\
        g_j(x,y) \leq 0,&\, \forall j\in J
\end{array}\right\}.
\end{equation}
To simplify the analysis, we will assume that $K$ has nonempty compact values and $\theta_f(x.\cdot)$ is at least lower semicontinuous. Thus, for every leader's decision $x\in X$, the follower's problem admits a solution. The goal of this section is two revise the two most common single-level reformulations of \eqref{ch04:eq:Optimistic-LowerLevelFunctionalConstraints}: the value-function reformulation and the Mathematical programming with complementarity constraints (MPCC) reformulation.
\subsection{Value-function reformulation}
Suppose that it is possible to compute the following value-function:
\begin{equation}\label{ch04:eq:ValueFunction}
    \begin{aligned}
        V:X&\to \R\\
        x&\mapsto \min_y\{ \theta_f(x,y)\, :\, y\in K(x)\}. 
    \end{aligned}
\end{equation}
Note that the value function has finite values due to the assumptions on $\theta_f$ and $K$. For every $x\in X$, the solution set $S(x)$ can be written in terms of $V$ as
\begin{equation}\label{ch04:eq:Equations-of-S(x)-valueFunction}
    S(x) = \left\{y\in \R^q\, :\, \begin{array}{rl}
        \theta_f(x,y) - V(x) \leq 0,\\
        h_i(x,y) = 0,\,\forall i\in I,\\
        g_j(x,y) \leq 0,\, \forall j\in J
\end{array}\right\}.
\end{equation}
Using this description, we can rewrite problem \eqref{ch04:eq:Optimistic-LowerLevelFunctionalConstraints} using the value function as follows.
\begin{formulation}{Value-function reformulation}
    Let $V:X\to\R$ given as in \eqref{ch04:eq:ValueFunction}. Then, problem \eqref{ch04:eq:Optimistic-LowerLevelFunctionalConstraints} is equivalent to
    \begin{equation}\label{ch04:eq:Optimistic-ValueFunctionRef}
    \begin{array}{rl}
        \displaystyle\min_{x,y} & \theta_l(x,y)  \\
         s.t. &\begin{cases}
             x\in X,\\
             \theta_f(x,y) - V(x) \leq 0,\\
        h_i(x,y) = 0,\,\forall i\in I,\\
        g_j(x,y) \leq 0,\, \forall j\in J.
        \end{cases}
    \end{array}
\end{equation}
\end{formulation}

While formulation \eqref{ch04:eq:Optimistic-ValueFunctionRef} is single-level, there is a price to pay: all the nonconvexity and nonsmoothness of the problem is now encoded in the function $V(\cdot)$, which, in general, doesn't have a closed form. This is a huge obstruction for the value-function reformulation, and its applicability is still at research level.

\subsection{MPCC reformulation}
Suppose that for every leader's decision $x\in X$, the set $K(x)$ is constraint qualified at every point $y\in K(x)$ and the follower's objective function $\theta_f(x,\cdot)$ is of class $\mathcal{C}^1$. Then, we could replace the optimality condition $y\in S(x)$ by the KKT conditions given by Theorem~\ref{ch04:thm:KKTEquations}. Indeed, we can define the Lagrangian function given by
\begin{equation}\label{ch04:eq:LagrantianFunction}
\mathcal{L}(x,y,\lambda,\mu) = \theta_f(x,y) + \sum_{i\in I} \lambda_ih_i(x,y) + \sum_{j\in J} \mu_j g_j(x,y),
\end{equation}
where the gradient with respect to the $y$-variable is given by
\begin{equation}\label{ch04:eq:LagrantianFunction-Gradient}
\nabla_y\mathcal{L}(x,y,\lambda,\mu) = \nabla_y\theta_f(x,y) + \sum_{i\in I} \lambda_i\nabla_yh_i(x,y) + \sum_{j\in J} \mu_j \nabla_y g_j(x,y).
\end{equation}
Then, instead of looking at the solution set $S(x)$, we consider the set of all tuples $(y,\lambda,\mu)$ that verify the KKT equations, that is, the set
\begin{equation}\label{ch04:eq:KKT-map}
\mathrm{KKT}(x) = \left\{(y,\lambda,\mu)\, :\, \begin{array}{rl}
\nabla_y\mathcal{L}(x,y,\lambda,\mu) &= 0,\\
    \forall j\in J,\,\, \mu_j&\geq 0,\\
    \forall j\in J,\,\, \mu_j g_j(x,y) &= 0,\\
    \forall i\in I,\,\, h_j(x,y) &= 0,\\
    \forall j\in J,\,\, g_j(x,y) &\leq 0.
\end{array}\right\}.  
\end{equation}
Note that all the constraints in $\mathrm{KKT}(x)$ are usual constraints, except for the complementarity constraints $\mu_jg_j(x,y) = 0$. The inclusion of these constraints is what leads to the formulation of a problem of Mathematical Programming with Complementarity Constraints (MPCC).
\begin{formulation}{MPCC formulation}
    Let us consider the Lagrangian function $\mathcal{L}$ as in \eqref{ch04:eq:LagrantianFunction}. The MPCC reformulation of problem \eqref{ch04:eq:Optimistic-LowerLevelFunctionalConstraints} is given by
    \begin{equation}\label{ch04:eq:MPCC-reformulation}
        \begin{array}{rl}
             \displaystyle\min_{x,y,\lambda,\mu} &\theta_l(x,y) \\
             s.t. &\begin{cases}
             x\in X,\\
                 \nabla_y\mathcal{L}(x,y,\lambda,\mu) = 0,\\
    \forall j\in J,\,\, \mu_j\geq 0,\\
    \forall j\in J,\,\, \mu_j g_j(x,y) = 0,\\
    \forall i\in I,\,\, h_j(x,y) = 0,\\
    \forall j\in J,\,\, g_j(x,y) \leq 0.
             \end{cases} 
        \end{array}
    \end{equation}
  Note that the decision variable of the MPCC reformulation is not longer $(x,y)\in\R^p\times\R^q$ but rather  $(x,y,\lambda,\mu)\in\R^p\times\R^q\times\R^I\times\R^J$, and that the feasible set is no longer compact, due to the unboundedness of $(\lambda,\mu)$. 
\end{formulation}

The main difference with the value-function reformulation, is that the nonconvexity and nonsmoothness of the problem are encapsulated in the constraints $\nabla_y\mathcal{L} = 0$, and the complementarity constraints $\mu_jg_j(x,y) = 0$. However, both constraints are now very regular: both are, in some sense, ``quadratic'' constraints induced by the product between multipliers and the data functions (or their derivatives). The price to pay, though, is that this new formulation is placed at the lifted space $\R^p\times\R^q\times\R^I\times\R^J$, and the equivalence between it and the original problem \eqref{ch04:eq:Optimistic-LowerLevelFunctionalConstraints} is less clear and possibly false.

\begin{example}\label{ch04:ex:Example1-BPvsMPCC} We construct an example where the Bilevel problem admits a unique solution, and the MPCC formulation doesn't have any. Consider the problem
\[
(BP) = \left\{\begin{array}{rl}
    \displaystyle\min_{x,y} & x  \\
     s.t. & \begin{cases}
         x\geq 0,\\
         y=(y_1,y_2)\text{ solves } \left\{\begin{array}{cl}
    \displaystyle\min_{y} & y_1  \\
     s.t. &\begin{cases}
         y_1^2 - y_2 - x\leq 0,\\
         y_1^2 + y_2\leq 0.
     \end{cases} 
\end{array}\right.
     \end{cases} 
\end{array}\right.
\]
The solution of the lower level is unique and it is given by $y(x) = (-\sqrt{x/2}, -x/2)$. The unique solution of $(BP)$ is then $x= 0$ and $y(x) = (0,0)$. See Figure \ref{ch04:fig:Example1-BPvsMPCC}. 
\begin{figure}[ht]
    \centering
    \includegraphics[width=0.7\linewidth]{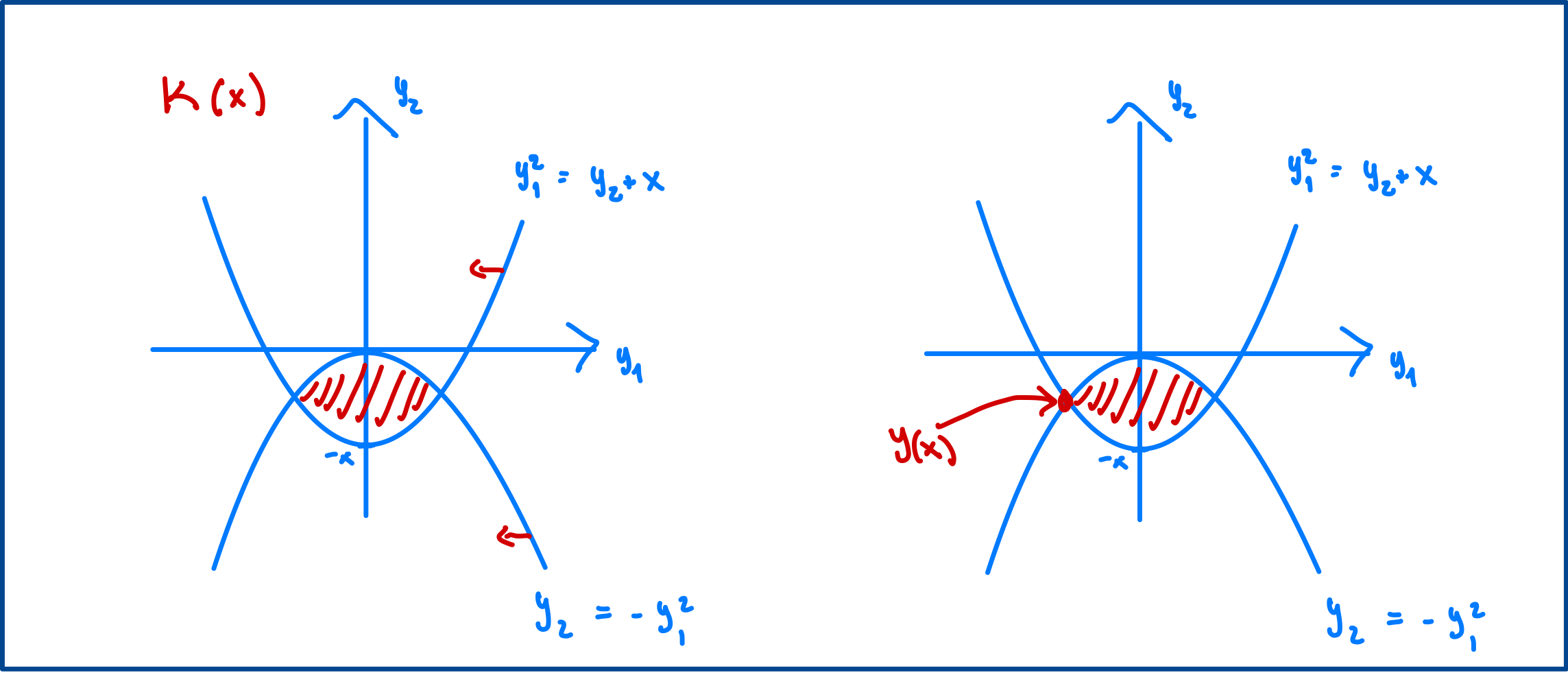}
    \caption{Lower level feasible set $K(x)$ and solution $y(x)$ of problem (BP) in Example \ref{ch04:ex:Example1-BPvsMPCC}.}
    \label{ch04:fig:Example1-BPvsMPCC}
\end{figure}
If we do the MPCC reformulation of $(BP)$, we obtain the problem
\[
(MPCC) = \left\{\begin{array}{cl}
    \displaystyle\min_{x,y,\mu} & x  \\
     s.t. &\begin{cases}
         x\geq 0,\\
         1 + 2\mu_1y_1 + 2\mu_2y_2 = 0,\\
         \mu_2 - \mu_1 = 0,\\
         \mu_1(y_1^2 - y_2 - x) = 0,\\
         \mu_2(y_1^2 + y_2) = 0,\\
         y_1^2 - y_2 - x\leq 0,\\
         y_1^2 + y_2\leq 0,\\
         \mu_1,\mu_2\geq 0.
     \end{cases} 
\end{array}\right.
\]
This problem fails to have a solution. Indeed, it is possible to verify that $x>0$ is always feasible. However, for $x=0$ one has that 
\[
y_1^2\leq y_2\text{ and }y_1^2\leq -y_2 \implies y_1 = y_2 = 0,
\]
and so the second constraint becomes $1=0$, leading to infeasibility. The problem here is that for $x = 0$ one has that $K(0) = \{(0,0)\}$ and therefore:
\begin{itemize}
    \item $T_{K(0)}(0,0) = \{(0,0)\}$.
    \item $L_{K(0)}(0,0) = \left\{ z\in\R^2\,:\, -y_2\leq 0,y_2\leq 0 \right\} = \R\begin{pmatrix}
        1\\ 0
    \end{pmatrix}$.
\end{itemize}
The lower level problem is not constraint qualified for $x=0$, and Theorem \ref{ch04:thm:KKTEquations} doesn't apply.
\end{example}

\begin{example}\label{ch04:ex:Example2-BPvsMPCC} We now provide an example of a Bilevel problem for which the MPCC admits solutions, but none of them is a solution of the original problem.
\[
(BP) = \left\{\begin{array}{rl}
    \displaystyle\min_{x,y} & (x-1)^2 + y^2  \\
     s.t. & y\text{ solves } \left\{\begin{array}{cl}
    \displaystyle\min_{y} & x^2y  \\
     s.t. & y^2\leq 0.
\end{array}\right.
\end{array}\right.
\]
Note that here $K(x) = \{0\}$ for every $x\in \R$ and therefore the unique solution of the lower-level problem is given by $y(x) =0$. This yields that the optimal solution is given by $(x,y) = (1,0)$. 

Concerning constraint qualification, we have that $T_{K(x)}(0) = \{0\}$. However,
\begin{align*}
L_{K(x)}(0) &= \left\{ z\in\R\,:\, g'(0)z \leq 0 \right\}\\
&= \left\{ z\in\R\,:\, 2yz\big|_{y=0} \leq 0 \right\} = \R.
\end{align*}
Thus, for any $x\in\R$, the lower-level problem fails to be constraint qualified. Finally, the MPCC formulation of $(BP)$ is given by
\[
(MPCC) = \left\{\begin{array}{cl}
    \displaystyle\min_{x,y,\mu} & (x-1)^2 + y^2  \\
     s.t. &\begin{cases}
         x^2 + 2\mu y = 0,\\
         \mu y^2 = 0,\\
         y^2 \leq 0,\\
         \mu\geq 0.   
     \end{cases} 
\end{array}\right.
\]
The feasible set of $(MPCC)$ is given by $\{ (0,0,\mu)\, :\, \mu\geq 0 \}$, all feasible points being optimal solutions. However, $(0,0)$ is not optimal solution of $(BP)$. 
\end{example}

Both Examples~\ref{ch04:ex:Example1-BPvsMPCC} and \ref{ch04:ex:Example2-BPvsMPCC} are based on the pathological behavior of representations of feasible sets via constraints that are not constraint qualified. However, if we impose constraint qualifications on the lower-level problem, together with convexity, the MPCC reformulation yields an equivalent problem, in the sense of the following theorem.

\begin{theorem}[Equivalent MPCC reformulation]\label{ch04:thm:Equivalence-BP-MPCC} Consider problem \eqref{ch04:eq:Optimistic-LowerLevelFunctionalConstraints} and its MPCC reformulation \eqref{ch04:eq:MPCC-reformulation}. Suppose that the lower-level problem is convex. That is,
\begin{enumerate}
    \item[(i)] The function $\theta_f(x,\cdot)$ is convex for every $x\in X$.
    \item[(ii)] The functions $g_j(x,\cdot)$ are convex for every $x\in X$ and every $j\in J$.
    \item[(iii)] The functions $h_i(x,\cdot)$ are affine for every $x\in X$ and every $i\in I$.
\end{enumerate}
    Suppose moreover that for every $x\in X$ the constraint set $K(x)$ as described in \eqref{ch04:eq:FunctionalConstraintsMap} is constraint qualified. Then
    \begin{itemize}
        \item[(a)] If $(\bar{x},\bar{y})$ is an optimal solution of the bilevel problem \eqref{ch04:eq:Optimistic-LowerLevelFunctionalConstraints}, then there exists $(\bar{\lambda},\bar{\mu})\in \R^I\times\R^J$ such that $(\bar{x},\bar{y},\bar{\lambda},\bar{\mu})$ is an optimal solution of the MPCC reformulation \eqref{ch04:eq:MPCC-reformulation}.
        \item[(b)] If $(\bar{x},\bar{y},\bar{\lambda},\bar{\mu})$ is an optimal solution of the MPCC reformulation \eqref{ch04:eq:MPCC-reformulation}, then $(\bar{x},\bar{y})$ is an optimal solution of the bilevel problem \eqref{ch04:eq:Optimistic-LowerLevelFunctionalConstraints}. 
    \end{itemize}
\end{theorem}
\begin{proof} Let us show both statements separately. To simplify the exposition, let us denote by $(BP)$ the bilevel problem \eqref{ch04:eq:Optimistic-LowerLevelFunctionalConstraints} and by $(MPCC)$ the MPCC reformulation \eqref{ch04:eq:MPCC-reformulation}. As usual, we denote by $S(x)$ the solution set of the lower-level problem induced by $x\in X$.

\paragraph{(a):} Suppose that $(\bar{x},\bar{y})$ is an optimal solution of $(BP)$. Then, $\bar{y}\in S(\bar{x})$, which yields, by Proposition \ref{ch04:prop:First-OrderNC}, that $-\nabla_y\theta_f(\bar{x},\bar{y})\in N_{K(\bar{x})}(\bar{y})$.

Since $K(x)$ is constrained qualified, Theorem~\ref{ch04:thm:KKTEquations} applies, and so there exists $(\bar{\lambda},\bar{\mu})\in \R^I\times\R^J$ such that $(\bar{y},\bar{\lambda},\bar{\mu})\in KKT(\bar{x})$ as in \eqref{ch04:eq:KKT-map}. Equivalently, $(\bar{x},\bar{y},\bar{\lambda},\bar{\mu})$ is feasible for $(MPCC)$. 

Now, let $(x,y,\lambda,\mu)$ be any other feasible point of $(MPCC)$. Constraint qualification of the lower-level problem yields that $-\nabla_y\theta_f(x,y) \in N_{K(x)}(y)$. Since the lower-level problem induced by $x$ is convex, we deduce that in fact $y\in S(x)$. Thus, $(x,y)$ is feasible for $(BP)$. Optimality of $(\bar{x},\bar{y})$ yields that
\[
\theta_l(\bar{x},\bar{y})\leq \theta_l(x,y).
\]
Since $(x,y,\lambda,\mu)$ is arbitrary, we deduce that $(\bar{x},\bar{y},\bar{\lambda},\bar{\mu})$ is optimal for $(MPCC)$.

\paragraph{(b):} Suppose that $(\bar{x},\bar{y},\bar{\lambda},\bar{\mu})$ is optimal for $(MPCC)$. As before, convexity yields that $(\bar{x},\bar{y})$ is feasible for $(BP)$. Let $(x,y)$ be any other feasible point of $(BP)$. As before, Proposition \ref{ch04:prop:First-OrderNC} and the constraint qualification of $K(x)$ entail that there exists $(\lambda,\mu)\in\R^I\times\R^J$ such that $(x,y,\lambda,\mu)$ is feasible for $(MPCC)$. Then, optimality of $(\bar{x},\bar{y},\bar{\lambda},\bar{\mu})$
\[
\theta_l(\bar{x},\bar{y})\leq \theta_l(x,y).
\]
Since $(x,y)$ is arbitrary, we deduce that $(\bar{x},\bar{y})$ is optimal for $(BP)$.
\end{proof}

\begin{warning}
    The MPCC reformulation usually requires convexity of the lower-level problem to be a valid reformulation. Indeed, it is convexity that provides the equivalence between optimal solutions (that is, elements of $S(x)$) and critical points (that is, under constraint qualifications, elements of $KKT(x)$).\\ 
    
    Without the convexity assumption, the MPCC reformulation might create artifacts: new feasible points $(x,y,\lambda,\mu)$ for which $(x,y)$ are not feasible points of the original problem. These points appear, for example, when $y$ is a local (and not global) optima of the lower-level problem.
\end{warning}

It is also reasonable to ask whether the local optima of the bilevel problem \eqref{ch04:eq:Optimistic-LowerLevelFunctionalConstraints} are related or not to the local optima of \eqref{ch04:eq:MPCC-reformulation}, in the spirit of Theorem~\ref{ch04:thm:Equivalence-BP-MPCC}. In general, even under convexity and constraint qualifications, this doesn't hold: local optima of the bilevel problem induce local optima of the MPCC reformulation, but the converse doesn't necessarily hold. However, it is possible to provide sufficient conditions to have equivalence. We refer the reader to \citep*{DempeDutta2012IsBilevel}.


\section{B\&B and hardness in linear bilevel programming}
\label{ch04:sec:BandB}

Let us go back to the case of (optimistic) linear bilevel programming, given in \eqref{ch02:eq:BLP-Optimistic}:

\begin{equation}\label{ch04:eq:LBP}
        \left\{\begin{array}{rl}
           \displaystyle\min_{x,y}  & c_l^{\top}x + d_l^{\top}y  \\
            s.t. & \left\{\begin{array}{l}
                 A_lx\leq b_l,  \\
                 y \text{ solves } \left\{\begin{array}{rl}
           \displaystyle\min_{y\in\R^q}  & c_f^{\top}y \\
            s.t. & A_fx + B_fy \leq b_f.
        \end{array}\right.
            \end{array}\right.
        \end{array}\right.
        \tag{LBP}
    \end{equation} 

Since linear problems are always constraint qualified, Theorem \ref{ch04:thm:Equivalence-BP-MPCC} applies and we can consider the equivalent linear MPCC reformulation given by
\begin{equation}\label{ch04:eq:LMPCC}
        \left\{\begin{array}{rl}
           \displaystyle\min_{(x,y,\mu)}  & \alpha_l^{\top}x + \beta_l^{\top}y  \\
            s.t. & \left\{\begin{array}{l}
                 A_lx\leq b_l,  \\
                 A_fx + B_fy \leq b_f,\\
                 -B_f^{\top}\mu = c_f,\\
                 \mu\geq 0,\\
                 \mu\odot(A_fx + B_fy - b_f) = 0.
            \end{array}\right.
        \end{array}\right.
        \tag{LMPCC}
\end{equation}
Here, the $\odot$ operation stands for the pointwise product of vectors. That is, for $a,b\in \R^n$, $a\odot b = (a_i\cdot b_i\,:\, i\in [n])$. In particular, we have that
\[
\mu\odot(A_fx + B_fy - b_f) = ( \mu_i\cdot( A_fx + B_fy - b_f)_i\,:\, i\in [m] ).
\]
In this problem, two main difficulties arise:
\begin{enumerate}
    \item The constraints $\mu\odot(A_fx + B_fy - b_f) = 0$, called \emph{complementary slackness}, are nonlinear.
    \item The (dual) variables $\mu\in\R_+^m$ are unbounded.
\end{enumerate}

\subsection{Big-M reformulation}
A first idea to deal with the difficulties of \eqref{ch04:eq:LMPCC} is to provide a constant $M>0$ large enough such that we can include the constraint $\mu\leq M$ without cutting any optimal solution of \eqref{ch04:eq:LBP}. This idea is based on the fact that multipliers of the MPCC reformulation are in fact dual solutions.

\begin{lemma}\label{ch04:lemma:MPCC-DualityPairs} A tuple $(x,y,\mu)$ is feasible for \eqref{ch04:eq:LMPCC} if and only if $A_lx\leq b_l$ and $(y,\mu)$ is a primal-dual solution pair for the lower-level problem of \eqref{ch04:eq:LBP} induced by $x$ (i.e. $y$ is a solution of the lower-level problem induced by $x$, and $\mu$ is a solution of the associated dual).
\end{lemma}
\begin{proof}
    Let $X = \{ x\in\R^p\,:\, \exists y\in\R^q,\, (x,y)\in D\}$. For a given $x\in X$, let us write the lower-level problem induced by $x$ as
    \[
    P(x) = \left\{\begin{array}{rl}
           \displaystyle\min_{\mu}  & c_f^{\top}y  \\
            s.t. & A_fx+B_fy \leq b_f.
        \end{array}\right.
    \]
    Note that the dual problem of $P(x)$ is given by
   \[
   D(x) = \left\{\begin{array}{rl}
           \displaystyle\max_{\mu}  & (A_fx-b_f)^{\top}\mu  \\
            s.t. & \left\{\begin{array}{l}
                 -B_f^{\top}\mu = c_f,  \\
                 \mu\geq 0.
            \end{array}\right.
        \end{array}\right.
   \]
   Then, if $(x,y,\mu)$ is a feasible point of \eqref{ch04:eq:LMPCC}, it yields directly that $y$ is feasible for $P(x)$ and $\mu$ is feasible for $D(x)$. Moreover, using the complementary slackness constraint, we can write
   \begin{align*}
       (A_fx-b_f)^{\top}\mu &= (A_fx + B_fy-b_f)^{\top}\mu - \mu^{\top}B_fy = \mu^{\top}B_fy = c_f^{\top}y.
   \end{align*}
   Strong duality (see notes of Section \ref{ch02:subsec:ContinuityCOnstrainedSets}) entails that $(y,\mu)$ are a primal-dual solution pair associated to the lower-level problem induced by $x$.

   The converse also holds: indeed, suppose that $(y,\mu)$ is a primal-dual solution pair associated to the lower-level problem induced by $x$. Then, by primal-dual feasibility, $(x,y,\mu)$ verifies directly almost all the constraints of \eqref{ch04:eq:LMPCC}, but for the complementary slackness. However, we can easily write from strong duality that
   \begin{align*}
       0 = -c_f^{\top}y - (b_f-A_fx)^{\top}\mu = (B_f^{\top}\mu)^{\top}y + (A_fx-b_f)^{\top}\mu = (A_fx + B_fy-b_f)^{\top}\mu.
   \end{align*}
   The proof is then finished.
\end{proof}

\begin{proposition} \label{ch04:prop:existence-BigM} Consider the linear bilevel problem \eqref{ch04:eq:LBP}. There exists $M$ large enough such that, for every feasible point $(x,y)$ of  \eqref{ch04:eq:LBP}, there exists $\mu\leq M$ such that $(x,y,\mu)$ is a feasible point of \eqref{ch04:eq:LMPCC}.      
\end{proposition}

\begin{proof}
    Consider the region 
    \[
    \Lambda = \{ \mu\in\R^m\,:\, -B_f^{\top}\mu = c_f,\, \mu\geq 0 \}.
    \]
    since $\Lambda$ doesn't contain lines, then $\ext(\Lambda)$ is nonempty. Recall that for every vector $\gamma\in\R^m$, one has that
    \[
    \max_{\mu\in\Lambda} \gamma^{\top}\mu <+\infty \implies \max_{\mu\in\Lambda} \gamma^{\top}\mu = \max_{\mu\in\ext(\Lambda)}\gamma^{\top}\mu.
    \]
    Define
    \[
    M = \max\{ \|\mu\|_{\infty}\,:\, \mu\in\ext(\Lambda) \}.
    \]
    Now, let $(x,y)$ be a feasible point of \eqref{ch04:eq:LBP}. Then, $y$ is a solution of the lower-level problem of \eqref{ch04:eq:LBP} induced by $x$, and so, the associated dual problem
    \[
    \max_{\mu\in\Lambda} (A_fx - b_f)^{\top}\mu
    \]
    has finite value. In particular, there exists $\mu\in \ext(\Lambda)$ such that $(y,\mu)$ is a primal-dual solution pair of the lower-level problem of \eqref{ch04:eq:LBP} induced by $x$. Using Lemma \ref{ch04:lemma:MPCC-DualityPairs}, we conclude that $(x,y,\mu)$ is a feasible point of \eqref{ch04:eq:LMPCC}. Since $\|\mu\|_{\infty}\leq M$ by construction, the result follows.
\end{proof}

With these ingredients, we are ready to present what is known as the Big-M reformulation of \eqref{ch04:eq:Big-M}. The formulation decouples the complementary slackness constraint in \eqref{ch04:eq:LMPCC} by introducing a binary variable (peer cordinate) that decides which part of the product will be zero. This formulation can be traced back to \citep{Fortuny-Amat1981Representation}.

\begin{formulation}{Big-M formulation}
For $M>0$, the Big-M reformulation of \eqref{ch04:eq:Big-M} is given by the Mixed-Integer programmong problem
        \begin{equation}\label{ch04:eq:Big-M}
     \left\{\begin{array}{rl}
           \displaystyle\min_{(x,y,z,\mu)}  & \alpha_l^{\top}x + \beta_l^{\top}y  \\
            s.t. & \left\{\begin{array}{l}
                 A_lx\leq b_l,  \\
                 A_fx + B_fy -b_f \leq 0,\\
                 -B_f^{\top}\mu = c_f,\\
                 \mu\geq 0,\\
                 \mu \leq M(1-z),\\
                 b_f - B_fy - A_fx \leq Mz\\
                 z\in \{0,1\}^m,
            \end{array}\right.
        \end{array}\right.
        \tag{Big-M}
   \end{equation}
\end{formulation}

\begin{theorem}\label{ch04:thm:Big-M-LBP-Equivalence} Consider the linear bilevel problem \eqref{ch04:eq:LBP} and suppose that it is bounded, that is, the set
\[
D = \left\{ (x,y)\,:\, \begin{array}{l}
     A_lx\leq b_l  \\
     A_f x + B_f y \leq b_f 
\end{array} \right\}
\]
is compact. Then, there exists $M>0$ large enough, such that \eqref{ch04:eq:LBP} is equivalent to \eqref{ch04:eq:Big-M} in the sense that:
 \begin{enumerate}
     \item For any solution (resp. feasible set) $(x,y)$ of \eqref{ch04:eq:LBP}, there exists  a pair $(\mu^*,z^*)$ such that $(x^*,y^*,\mu^*,z^*)$ is a solution (resp. feasible set) of \eqref{ch04:eq:Big-M}.
     \item For any solution (resp. feasible set) $(x^*,y^*,\mu^*,z^*)$ of \eqref{ch04:eq:Big-M}, the pair $(x^*,y^*)$ is a solution (resp. feasible set) of \eqref{ch04:eq:LBP}.
 \end{enumerate}
\end{theorem}
\begin{proof}
    It is enough to prove the statements for feasible points. Let $M_1>0$ be the constant given by Proposition \ref{ch04:prop:existence-BigM}, and let
    \[
    M_2 = \max_{(x,y)} \big\{ \|b_f - A_fx - B_fy\|_{\infty}\,:\, (x,y)\in D \big\}.
    \]
    We set $M= \max\{M_1,M_2\}$.
    
    Now, let $(x,y)$ be a feasible set of \eqref{ch04:eq:LBP}. Then, by Proposition~\ref{ch04:prop:existence-BigM}, there exists $\mu\leq M$ such that $(x,y,\mu)$ is feasible for \eqref{ch04:eq:LMPCC}. Now, define $z\in \{0,1\}^m$ given by
    \[
    \forall i\in [m],\,\,z_i = \begin{cases}
        1\quad&\text{ if }\mu_i = 0,\\
        0\quad&\text{ otherwise.}
    \end{cases}
    \]
    Then, complementary slackness (i.e. the constraint $\mu\odot (b_f - B_fy - A_fx) = 0$) and the fact that the constraint $b_f - B_fy - A_fx\leq M$ is trivial by construction, yield that $(x,y,\mu,z)$ is feasible for \eqref{ch04:eq:Big-M}. 
    
    Conversely, take a feasible point $(x,y,\mu,z)$ of \eqref{ch04:eq:Big-M}. Note that the constraints involving $z$, together with the non-negativity of $\mu$ and $b_f - B_fy -A_fx$, entails that $(x,y,\mu)$ verify the complementary slackness. Thus,  $(x,y,\mu)$ is feasible for \eqref{ch04:eq:LMPCC} and so, $(x,y)$ is feasible for \eqref{ch04:eq:LBP}. The proof is then finished. 
\end{proof}

\subsection{SOS1 branch-and-bound}

The main advantage of \eqref{ch04:eq:Big-M} is that it is a Mixed-Integer programming problem, and so it can be passed directly to well-developed solvers, such as \texttt{Gurobi}. However, the drawback is that the constant $M$ large enough needs to be computed.

It is possible to work directly with \eqref{ch04:eq:LMPCC} using branching over the complementary slackness. This branching method can be also directly passed to commercial solvers by encoding the complementary slackness as SOS1-type constraints (see, e.g., \citep{Aussel2024SOS1}).
\begin{formulation}{SOS1 Branch-and-Bound}
We present the SOS1 branch-and-bound algorithm for \eqref{ch04:eq:LMPCC}. The algorithm receives the formulation of \eqref{ch04:eq:LMPCC}, and returns a solution $(x^*,y^*,\mu^*)$ with optimal value $v^*\in\overline{\R}$.

\begin{enumerate}[leftmargin = 1.5cm,label=\texttt{- Step \arabic*}, ref=\texttt{(Step~\arabic*)}]\setlength{\itemsep}{0.3cm}\setcounter{enumi}{-1}
    \item \texttt{(Initialization)} Initialize the index set $I_0 = [m]$ and consider the relaxed problem
    \[
    P_0 =\left\{\begin{array}{rl}
           \displaystyle\min_{(x,y,\mu)}  & \alpha_l^{\top}x + \beta_l^{\top}y  \\
            s.t. & \left\{\begin{array}{l}
                 A_lx\leq b_l,  \\
                 A_fx + B_fy \leq b_f,\\
                 -B_f^{\top}\mu = c_f,\\
                 \mu\geq 0.\\
            \end{array}\right.
        \end{array}\right.  
    \]
    Initialize the active nodes as $\mathcal{N} = \{ (P_0,I) \}$ and the current solution as $(x^*,y^*,\mu^*) = \texttt{NaN}$, $v = +\infty$.
    \item\label{ch04:alg:BandB-Step1} (Choosing a node) While $\mathcal{N}$ is nonempty, choose a pair $(P,I)\in \mathcal{N}$. Remove $(P,I)$ from $ \mathcal{N}$ and solve the linear problem $P$, obtaining a solution $(x,y,\mu,v)$ (if $v\in\{-\infty,+\infty\}$, set $(x,y,\mu) = \texttt{NaN}$).
    \item \label{ch04:alg:BandB-Step2} \texttt{(Pruning)} Explore the following cases:
    \begin{enumerate}\setlength{\itemsep}{0.3cm}
        \item \texttt{Prune by infeasibility:} If $v=+\infty$, continue to \ref{ch04:alg:BandB-Step4}.
        \item \texttt{Prune by bound:} If $v \in \R$ and $v\geq v^*$, continue to \ref{ch04:alg:BandB-Step4}.
         \item \texttt{Prune by SOS1-feasibility:} If $v\in\R$ and $(x,y,\mu)$ is feasible for \eqref{ch04:eq:LMPCC} then 1) If $v<v^*$, then update $(x^*,y^*,\mu^*,v^*)\leftarrow (x,y,\mu,v)$; 2) continue to \ref{ch04:alg:BandB-Step4}.
         \item \texttt{Leaf:} If $I = \emptyset$ and $v^*<v$, then update $(x^*,y^*,\mu^*,v^*)\leftarrow (x,y,\mu,v)$ and continue to \ref{ch04:alg:BandB-Step4}.
    \end{enumerate}
    \item \label{ch04:alg:BandB-Step3} \texttt{(Branching)} If none of the cases of \ref{ch04:alg:BandB-Step2} holds, then necessarily $I\neq\emptyset$. Choose $i \in I$ and consider the following problems:
         \begin{align*}
             P^- &:= P\text{ with the extra constraint }\mu_i = 0.\\
             P^+ &:= P\text{ with the extra constraint }(A_fx + B_fy - b_f)_i = 0.
         \end{align*}
         Update $\mathcal{N}\leftarrow \mathcal{N}\cup \{ (P^-,I\setminus\{i\}), (P^+,I\setminus\{i\}) \}$ and continue to \ref{ch04:alg:BandB-Step4}.
    \item\label{ch04:alg:BandB-Step4} \texttt{(Stopping)} If $v^*=-\infty$ or $\mathcal{N}=\emptyset$, return $(x^*,y^*,\mu^*,v^*)$. Otherwise, go back to \ref{ch04:alg:BandB-Step1}.
\end{enumerate}
\end{formulation}

The idea of the Branch-and-Bound method is to explore all the possible (exponentially many) combinations of how to decouple $\mu\odot (A_fx+B_fy-b_f) = 0$ in $m$ linear constraints: for each $i\in [m]$, either impose $\mu_i = 0$ or $(A_fx+B_fy-b_f)_i = 0$. 

The Branch-and-Bound algorithm sequentially includes these alternatives as a binary tree, on the process known as \emph{Branching} (see \ref{ch04:alg:BandB-Step3}). If no pruning is done, the branching process would construct a binary tree of depth $m$, and within its $2^m$ leafs we would recover all the possible combinations of decoupling $\mu\odot (A_fx+B_fy-b_f) = 0$. However, the cleverness of Branch-and-Bound is to \emph{prune} some nodes of the tree before the branching process ends. The four types of pruning are: 
\begin{enumerate}[label={(\alph*)}]
    \item \textbf{By infeasibility:} If at a given moment, the partial problem $P$ is infeasible, there is no need to continue branching. All the subsequent problems will be infeasible as well.
    \item \textbf{By bound:} If at a given moment the algorithm has a current optimal value $v^*$ and the solution of $P$ has optimal value $v>v^*$, there is no need to continue branching. Indeed, all the subsequent problems will have a worse optimal value, since $P$ is a relaxation of all its descendants (it has the least constraints). Thus, any solution we might find over subsequent nodes is worse than the current solution and there is no need to explore it.
    \item \textbf{By SOS1-feasibility:} If a current solution $(x,y,\mu,v)$ of a relaxation $P$ is already feasible for the complete problem \ref{ch04:eq:LMPCC}, then this is a candidate of optimal solution. If the $v>v`*$, then we prune the node by bound. Otherwise, $(x,y,\mu,v)$ is a better than the current solution $(x^*,y^*,\mu^*,v^*)$, and we update $(x^*,y^*,\mu^*,v^*)\leftarrow (x,y,\mu,v)$. After updating, we can prune the node as well since the criteria of pruning by bound is met.
\end{enumerate}
It is not hard to convince that the Branch-and-Bound algorithm ends with an optimal solution of \eqref{ch04:eq:LMPCC}, if it exists (if the problem if infeasible, it return $v^*=+\infty$, and if the problem is unbounded, the algorithm verifies it at a leaf of the branching tree). The price to pay is that the algorithm has exponential worst-case, since it might need to explore all the branching tree before stopping.
\begin{note}{Branch-and-Bound for MIP}
    The SOS1 Branch-and-Bound method is a direct extension of the classic Branch-and-Bound method for Mixed-Integer programming. In fact, the method for MIP problems can be deduced from the one presented here by noting that
    \[
        z \in \{0,1\} \iff z\in [0,1] \text{ and }z\cdot(1-z) = 0.
    \]
    However, solvers like \texttt{Gurobi} have the MIP version much more developed. In fact, solvers use Branch-and-Bound as a base, but run many heuristics to try to accelarate the process. 
\end{note}

\begin{note}{Big-M reformulations vs SOS1 Branch-and-Bound}

The main motivation to develop SOS1 methods was the belief that Big-M reformulations were impractical in general. In \citep{Pineda2019Glitters}, a counterexample was provided showing that natural incremental heuristics can fail to find correct big-M's. Moreover, in \citep{Kleinert2020HardnessBig-M}, it was shown that verifying if a constant $M$ is large enough or not for the validity of Theorem~\ref{ch04:thm:Big-M-LBP-Equivalence} is very hard. These results led the community of bilevel programming that efficient Big-M reformulations were only possible for particular cases.\\

However, very recently, in \citep{Buchheim2023BilevelInNP} it was shown that valid constants can be found efficiently and so, Big-M reformulations and SOS1 Branch-and-Bound are both applicable in practice. While some arguments have been made in favor of SOS1 Branch-and-Bound (see, e.g., \citep{KleinertSchmidt2023NoNeed-BigM}), there is no categorical answer to which method is better.  
\end{note}

\subsection{Linear bilevel programming is NP-hard}\label{subsec:NLP-NP-hard}

A natural question is whether the algorithms we have seen so far are optimal for (optimistic) linear bilevel programming or not. That is, it is possible to do better than using MIP solvers or SOS1 Branch-and-Bound? The answer is: not really, unless $P=NP$.

\begin{note}{Complexity classes $P$ and $NP$ (Informal definitions)}
    Computational complexity is the study of how difficult is a computational decision problem, which is a class $\mathcal{I}$ of problems such that for any instance $I\in\mathcal{I}$, the answer to $I$ is either \texttt{YES} or \texttt{NO}. Two fundamental classes of problems are defined:
    \begin{enumerate}
        \item \textbf{Complexity class P:} A class of problems $\mathcal{I}$ is said to belong to P if there is a polynomial-time algorithm $\mathrm{Alg}_{\mathcal{I}}$ that solves it. That is, for an instance $I\in \mathcal{I}$, the algorithm $\mathrm{Alg}_{\mathcal{I}}$ receives an encoding of $I$ and returns the solution after $p(|I|)$ many iterations, where $p(\cdot)$ is a polynomial and $|I|$ is the length on the encoding.
        \item \textbf{Complexity class NP:} A class of problems $\mathcal{I}$ is said to belong to NP if there is a polynomial-time algorithm $\mathrm{Alg}_{\mathcal{I}}$ that verifies solutions for it using polynomial certificates. That is, for an instance $I\in \mathcal{I}$ there exists a certificate $V$ of size $q(|I|)$, where $q(\cdot)$ is a polynomial, such that the algorithm $\mathrm{Alg}_{\mathcal{I}}$ receives the encoding of $(I,V)$ and decides the solution of $I$ after $p(|I|)$ many iterations.
    \end{enumerate}
 The class NP is an abbreviation of ``non-deterministic polynomial'', which intuitively means that it contains the problems that can be decided in polynomial time if we are allowed to ``guess'' the solution and verify it afterwards.  It is clear that ${\rm P}\subset{\rm NP}$ since, for problems in P, we can use $\emptyset$ as certificate and just run $P$. Probably the most celebrated open problem in Computers Science is whether P$=$NP or not.\\
 
 For a formal exposition on complexity classes, we refer the reader to \citep{AroraBarak2009Complexity}. 
\end{note}

\begin{note}{NP-hardness}
    A computational class of (not necessary decision) problems $\mathcal{I}$ is said to be \textbf{NP-hard} if for any class $\mathcal{J}$ in NP, there exist an algorithm $R$ (called \textit{polynomial-time reduction}) that receives an instance $J\in\mathcal{J}$ and
    \begin{enumerate}
        \item $R$ computes an instance $I\in\mathcal{I}$ associated to $J$ in polynomial time (with respect to the size $|J|$).
        \item If $x$ is a solution of $I$ (of polynomial size with respect to $|J|$), then $R$ can use $x$ and $I$ to solve $J$ in polynomial time (with respect to $|J|$).
    \end{enumerate}
    A problem is NP-hard if it is at least as hard as any problem in NP. Indeed, if one could solve the instances of $\mathcal{I}$ in polynomial time (and so, finding solutions of polynomial size), then $R$ could be used to solve any instance of $\mathcal{J}$ in polynomial time as well.\\

     To show NP-hardness of a given class $\mathcal{I}$, it is enough to construct a polynomial-time reduction from another NP-hard class $\mathcal{J}$ to the class $\mathcal{I}$.
\end{note}

We will show that the class of Optimistic bilevel linear programming problems \texttt{BLP} is NP-hard. To do so, we will follow the strategy of \citep{Ben-Ayed1990Difficulties}: we will show that the class \texttt{KNAPSACK} can be reduced in polynomial-time to \texttt{BLP}.

\begin{formulation}{KNAPSACK problem}
    For a list of numbers $A:=\{a_1,\ldots,a_n\}\in \N$ and a threshold $\delta\in \N$, find the subset $S\subset A$ that maximizes $\sum_{a\in S} a$ without exceeding the threshold $\delta$. That is, to solve the integer programming problem
    \begin{equation}\label{ch04:eq:Knapsack-IP}
    \left\{\begin{array}{cl}
       \displaystyle\max_{x}  & \sum_{i=1}^{n} a_ix_i  \\
         s.t. & \begin{cases}
             \sum_{i=1}^n a_ix_i \leq \delta,\\
             x\in\{0,1\}^n.
         \end{cases} 
    \end{array}\right.
    \end{equation}
    Note that an instance $I\in \texttt{KNAPSACK}$ is given by $I = (a_1,\ldots,a_n,\delta)$ and problem \eqref{ch04:eq:Knapsack-IP} has polynomial-size encoding with respect to $|I|$. It is well-known that \texttt{KNAPSACK} is NP-hard. 
\end{formulation}

\begin{warning}
    Here, we are assuming that encodings are \textbf{binary encodings}. If one uses other type of encodings (such as unary, of example), the complexity results presented here are not longer valid.
\end{warning}

Take an instance $I=(a_1,\ldots,a_n,\delta)$. Without lose of generality, we will assume that $\max\{a_i\,:\, i\in [n]\}\geq 2$. Indeed, the case where all $a_i$'s are equal to 1 is trivial to solve. In the IP formulation \eqref{ch04:eq:Knapsack-IP}, we relax the constraint $x\in\{0,1\}^n$ by introducing a new variable $y\in[0,1]^n$ that penalizes the values that are not integer. Thus, we want to impose
\[
x\in\{0,1\}^n\xrightarrow{\hspace{1cm}} \begin{cases}
    x\in [0,1]^n,\\
    y_i = \min\{x_i,1-x_i\},\quad \forall i\in [n].
\end{cases}
\]
Then, we use the new variables $y$ to penalize choosing fractional elements in the Knapsack problem.
\begin{equation}\label{ch04:eq:knapsack-to-blp-partial}
    \left\{\begin{array}{cl}
       \displaystyle\max_{x,y}  & \sum_{i=1}^{n} a_ix_i  - M\sum_{i=1}^n y_i \\
         s.t. & \begin{cases}
             \sum_{i=1}^n a_ix_i \leq \delta,\\
             x\in [0,1]^n,\\
             y_i = \min\{x_i,1-x_i\},\quad \forall i\in [n].
         \end{cases} 
    \end{array}\right.
\end{equation}
Finally, we can rewrite the definition $ y_i = \min\{x_i,1-x_i\}$ by solving a lower-level problem.
\begin{equation}\label{ch04:eq:knapsack-to-blp}
    \left\{\begin{array}{cl}
       \displaystyle\max_{x,y}  & \sum_{i=1}^{n} a_ix_i  - M\sum_{i=1}^n y_i \\
         s.t. & \begin{cases}
             \sum_{i=1}^n a_ix_i \leq \delta,\\
             x\in [0,1]^n,\\
             y\text{ solves }\left\{\begin{array}{cl}
                        \displaystyle\max_{y}  & \sum_{i=1}^{n} y_i  \\
                         s.t. & \begin{cases}
                             y\leq x, y\leq 1-x,\\
                             y\geq 0.
                         \end{cases} 
                    \end{array}\right.
         \end{cases} 
    \end{array}\right.
\end{equation}

Note that if $M\in\N$ has polynomial size with respect to $I$, then the encoding of \eqref{ch04:eq:knapsack-to-blp} is of polynomial size with respect to the encoding of $I$.

Thus, in order to show that \eqref{ch04:eq:knapsack-to-blp} is a polynomial-time reduction of $I$, it is enough to show that $M$ can be chosen of polynomal-size such that for any solution $(x,y)$ of \eqref{ch04:eq:knapsack-to-blp}, $x$ is a solution of \eqref{ch04:eq:Knapsack-IP}.

\begin{theorem} Let $I = (a_1,\ldots,a_n,\delta)\in\texttt{\em KNAPSACK}$ with $\beta =\max\{a_i\,:\, i\in [n]\}\geq 2$. Choose $M>\beta^2$. Then, if $(x,y)$ is a feasible point of \eqref{ch04:eq:knapsack-to-blp} with $x$ not integer, there exists $z\in\{0,1\}^n$ such that $(z,0)$ is also feasible for \eqref{ch04:eq:knapsack-to-blp} and
\[
\sum_{i=1}^{n} a_ix_i  - M\sum_{i=1}^n y_i < \sum_{i=1}^{n} a_iz_i.
\]
In particular, every optimal solution $(x,y)$ of \eqref{ch04:eq:knapsack-to-blp} is integer, and therefore its first coordinate is also a solution of \eqref{ch04:eq:Knapsack-IP}. Therefore, choosing $M = \beta^2+1$ induces a polynomial-time reduction of $I$ to \eqref{ch04:eq:knapsack-to-blp}.   
\end{theorem}
\begin{proof}
     Set $Q = 1 - \frac{1}{\beta}$ and consider the functions
     \[
     f(x) = \sum_{i=1}^n a_ix_i \quad\text{ and }\quad g(x) = \sum_{i=1}^n \min\{x_i,1-x_i\}. 
     \]
     Since $\beta\geq 2$, $Q\geq 1/2$. We produce $z$ by modifying $x$ as follows:
     \begin{enumerate}[label=(\arabic*)]
         \item Choose $x_i\in (0,Q]$, set $x_i = 0$ and go to (2). If this is not possible, simply go to (2).
         \item Choose $x_i,x_j\in (Q,1)$ (If this is not possible, go to (3)). Replace them by $x_i',x_j'$ so $a_ix_i' + a_jx_j' = a_ix_i + a_jx_j$, $x_i'+x_j'\geq x_i+x_j$, and $\max\{x_i',x_j'\}=1$.
         \begin{itemize}
             \item If $a_i = a_j$ just set $x_i' = 1$ and $x_j' = x_j - (1-x_i)$. Since $Q\geq 1/2$, both new values are non-negative.
             \item If $a_i<a_j$, set $x_i' = 1$, and $x_j' = x_j - \frac{a_i}{a_j}(1-x_i)$.
         \end{itemize}
         Go to (3).
         \item If there is only one $x_j\in (Q,1)$ and all the other entries are integer, set $x_j = 1$. Otherwise, go back to (1).
     \end{enumerate}
     Note that at each modification, the number of integer variables increases. Thus, the procedure ends after at most $n$ modification with an integer vector $z$. Let $x^0,x^1,\ldots,x^m$ be the sequence of modifications obtained using this method, with $x^0 = x$ and $x^m = z$.

     Take $k\leq m-1$ and suppose that $x^k$ is feasible for \eqref{ch04:eq:knapsack-to-blp}. We will show that $x^{k+1}$ is also feasible \eqref{ch04:eq:knapsack-to-blp} and that $f(x^k)-Mg(x^k) \leq f(x^{k+1}) - Mg(x^{k+1})$, with the exception of the case $k=m$, where the inequality is strict.
     \paragraph{Case (1):} Suppose that the modification from $x^k$ to $x^{k+1}$ happens at (1). Then, there is $j\in [n]$ such that $x_j^k\in (0,Q]$ and $x_j^{k+1} = 0$. Clearly, $x^{k+1}$ is feasible. Moreover, 
     \begin{itemize}
         \item If $x_j^k\leq 1/2$, we have that
         \begin{align*}
             f(x^{k+1}) - M g(x^{k+1}) &= (f(x^{k}) - a_jx_j)  - M(g(x^k)-x_j)\\
             &= f(x^{k}) - Mg(x^k) + (M-a_j)x_j > f(x^{k}) - Mg(x^k).
         \end{align*}
         \item If $x_j \in (1/2,Q]$ we have that
         \begin{align*}
             f(x^{k+1}) - M g(x^{k+1}) &= (f(x^{k}) - a_jx_j)  - M(g(x^k)-(1-x_j))\\
             &\geq f(x^{k}) - Mg(x^k) + M(1-Q) - a_jQ\\
             &\geq f(x^{k}) - Mg(x^k) + M/\beta - 1 > f(x^{k}) - Mg(x^k).
         \end{align*}
     \end{itemize}
     In either case, the desired inequality $f(x^k)-Mg(x^k) < f(x^{k+1}) - Mg(x^{k+1})$ holds.
     \paragraph{Case (2):} Suppose that the modification from $x^k$ to $x^{k+1}$ happens at (2). Then, by construction $x^{k+1}$ is feasible. Moreover, $f(x^{k+1}) = f(x^k)$. Finally, let $i,j$ the involved coordinates, and suppose that $a_i\geq a_j$. Then,
     \begin{align*}
         g(x^{k+1}) &= g(x^{k}) - (2 - (x_i^k+x_j^k)) + \min\{x_j^{k+1},1-x_j^{k+1}\}\\
         &\leq g(x^{k}) - (2 - (x_i^{k+1}+x_j^{k+1})) + \min\{x_j^{k+1},1-x_j^{k+1}\}\\
         &=  g(x^{k}) - (1 - x_j^{k+1}) + \min\{x_j^{k+1},1-x_j^{k+1}\} \leq g(x^k).
     \end{align*}
     This yields that the inequality $f(x^k)-Mg(x^k) \leq f(x^{k+1}) - Mg(x^{k+1})$ holds. Moreover, note that if $k = m-1$, then $x_j^{k+1} = 0$, and so, in this case, $g(x^{m})<g(x^{m-1})$. Then, if $k=m-1$, we get that $f(x^{m-1})-Mg(x^{m-1}) < f(x^{m}) - Mg(x^{m})$.
     \paragraph{Case (3):} Suppose that the modification from $x^k$ to $x^{k+1}$ happens at (3). This can only happen if $k=m-1$, and we trivially have that $f(x^{m-1})-Mg(x^{m-1}) < f(x^{m}) - Mg(x^{m})$. Thus, we only need to show that $z=x^m$ is feasible, that is, $\sum_{i=1}^n a_iz_i \leq \beta$. But in this case, we have that
     \[
     \sum_{i=1}^n a_iz_i < 1 + \sum_{i=1}^n a_ix_i^{m-1} \leq 1+\beta.
     \]
     since both, $\sum_{i=1}^n a_iz_i$ and $1+\beta$ are integers, we conclude that $\sum_{i=1}^n a_iz_i \leq \beta$, finishing the construction.

     Now, the conclusion follows applying induction, deducing that $x=x^0,x^1,\ldots,x^m=z$ is a feasible sequence, verifying that
     \begin{align*}
     f(x) - Mg(x) = f(x^0) - Mg(x^0)&\leq f(x^1) - Mg(x^1)\\
     & \vdots\\
     &\leq f(x^{m-1}) - Mg(x^{m-1}) < f(x^{m}) - Mg(x^{m}) = f(z).
     \end{align*}
     The proof is then finished.
\end{proof}

\begin{note}{Bilevel programming in NP}
 It has been recently showed in \citep{Buchheim2023BilevelInNP} that the decision version of the optimistic linear bilevel programming belongs to NP.\\
 
 The decision version is as follows: For the problem \eqref{ch04:eq:LBP} with rational data, and a rational value $V$, to decide whether there is a feasible point $(x,y)$ with value smaller than $V$ or not.  Here, rational entries is paramount to provide a proper analysis. See, e.g., \citep{Conforti2014Integer} for the nuances of rational entries in linear (and mixed-integer) programming.  
\end{note}



\newpage
\section{Problems}\label{ch04:sec:Problems}

\begin{problem}[Correspondence of local optima] Consider the problem
\[
(P)=\left\{\begin{array}{cl}
    \displaystyle\min_{x,y} & (x-1)^2 + (y-1)^2 \\
    s.t. & y\text{ solves } \left\{\begin{array}{cl}
    \displaystyle\min_{y} & -y \\
    s.t. & \begin{cases}
        y+x\leq 1,\\
        y-x\leq 1.
    \end{cases} 
\end{array}\right.
\end{array}\right.
\]
\begin{enumerate}[label=(\alph*)]
    \item Show that, for every $x\in \R$, the lower level problem is constraint qualified.
    \item Write the MPCC reformulation of $(P)$.
    \item Compute all the critical points of MPCC.
    \item Compute the global optima of $(P)$.
    \item Show that the point $(x,y) = (0,1)$ admits multipliers $(\mu_1,\mu_2)$ such that $(0,1,\mu_1,\mu_2)$ is a local optima for the MPCC reformulation and that $(0,1)$ fail to be local optima of $(P)$.
\end{enumerate} 
\end{problem}
\vspace{0.5cm}

\begin{problem}[Reformulations in linear programming] Consider the linear bilevel programming problem 
    \[
\begin{array}{cl}
    \displaystyle\min & 7x - 2y_1 + 3y_2 \\
    s.t. & \begin{cases}
     x\geq 0,\\   
    y\text{ solves } \left\{\begin{array}{cl}
    \displaystyle\min_{y} & 2y_1 - y_2 \\
    s.t. & \begin{cases}
        4x + y_1 + y_2 \leq 3,\\
        2x - 2y_1 + 5y_2 \leq 5,\\
        3x - y_1 - 2y_2 \leq 1,\\
        y_1,y_2\geq 0.
    \end{cases} 
\end{array}\right.
\end{cases}
\end{array}
\]
\begin{enumerate}[label=(\alph*)]
    \item For every $x\geq 0$, compute the dual problem of the lower-level problem.
    \item Identify all the vertices of the dual problem. Detail your procedure and computations. If you use a software, include the code.
    \item Write the MPCC formulation and a valid Big-M reformulation.
    \item Solve the Big-M formulation (MIP formulation) and the MPCC formulation (with SOS1-type constraints) using a solver. \textbf{Recommended:} \texttt{Julia} + \texttt{Gurobi} with academic license.
\end{enumerate} 
\end{problem}
\chapter{Extensions}
\label{Chapter05:Extensions}

\section{Mixed-Integer bilevel programming}
\label{ch05:Mixed-Integer}

A natural extension of linear bilevel programming is to admit integrality constraints for variables of the leader and/or the follower. A general formulation of a Mixed-Integer bilevel programming problem is given by

\begin{equation}\label{ch05:eq:MIBP}
        \left\{\begin{array}{rl}
           \displaystyle\min_{x,y}  & c_l^{\top}x + d_l^{\top}y  \\
            s.t. & \left\{\begin{array}{l}
                 A_lx\leq b_l,  \\
                 x\in X,\\
                 y \text{ solves } \left\{\begin{array}{rl}
           \displaystyle\min_{y}  & c_f^{\top}y \\
            s.t. & \begin{cases}A_fx + B_fy \leq b_f,\\ y\in Y,\end{cases}
        \end{array}\right.
            \end{array}\right.
        \end{array}\right.
        \tag{MIBP}
    \end{equation} 
where $X$ and $Y$ are the integrality constraints. That is, for some subsets $I\subset [p]$ and $J\subset [q]$,
\begin{align*}
    X&= \{x\in \R^p\ :\ x_i \in Z,\, \forall i\in I\},\\
    Y&= \{y\in \R^q\ :\ y_j \in Z,\, \forall j\in J\}.
\end{align*}
In general, \eqref{ch05:eq:MIBP} might fail to have solutions, as the following example shows.

\begin{example}\label{ch05:ex:MIBP-without-solutions} Consider the problem
\[
\left\{ \begin{array}{cl}
   \displaystyle \min_{x,y}  &x-y  \\
    s.t. & \begin{cases}
        x\in [0,1]\\
        y\text{ solves }\left\{\begin{array}{cl}
             \displaystyle\min_y& y  \\
             s.t. & y\in \{0,1\},\, y\geq x. 
        \end{array}\right.
    \end{cases} 
\end{array}\right.
\]
For a given $x\in [0,1]$, the unique solution of the lower level problem is given by $y = \lceil x\rceil$. Thus, the MIBP problem can be reduced to
\[
\min_{x\in [0,1]} x - \lceil x\rceil,
\]
which fails to have solution: the optimal value is $-1$ but is never reached since $x-\lceil x\rceil$ is discontinuous at $0$.
\end{example}

\begin{proposition} Consider problem \eqref{ch05:eq:MIBP} and suppose that one of the following holds:
\begin{enumerate}
    \item[(i)] Only the leader has integer variables (i.e. $J = \emptyset$).
    \item[(ii)] All the variables of the leader are integer ones (i.e. $I=[p]$).
\end{enumerate}
    Suppose than the set
    \[
    D = \left\{ (x,y)\, :\ , \begin{array}{l}
         A_lx\leq b_l\\
         A_fx+B_fy\leq b_f
    \end{array} \right\}
    \]
    is compact. Then, \eqref{ch05:eq:MIBP} is either infeasible or it admits a solution.
\end{proposition}
\begin{proof} Assume first that $(i)$ holds, and let us write $x_I$ to denote the vector of integer variables of the leader. Since $D$ is compact, the variables $x_I$ are bounded, and so they can take finitely many values. Let $\{v_1,\ldots,v_N\}\in \Z^I$ be all the possible values of $x_I$. Then, for each $k\in [N]$, we can define the problem 
    \[
    (\text{MIBP}_k) = \left\{\begin{array}{rl}
           \displaystyle\min_{x,y}  & c_l^{\top}x + d_l^{\top}y  \\
            s.t. & \left\{\begin{array}{l}
                 A_lx\leq b_l,  \\
                 x_I = v_k,\\
                 y \text{ solves } \left\{\begin{array}{rl}
           \displaystyle\min_{y}  & c_f^{\top}y \\
            s.t. & A_fx + B_fy \leq b_f.
        \end{array}\right.
            \end{array}\right.
        \end{array}\right.
    \]
    Then, $(\text{MIBP}_k)$ is an optimistic bilevel linear programming problem. Thus, either it is infeasible, or it has a solution $(x^*_k,y^*_k)$. The conclusion follows by noting that either all problems $(\text{MIBP}_k)$ are infeasible making \eqref{ch05:eq:MIBP} infeasible as well, or \eqref{ch05:eq:MIBP} admits a solution, which is obtained as
    \[
    \min\{ c_l^{\top}x^*_k + d_l^{\top}y^*_k\, :\, (x_k^*,y_k^*)\text{ solves }(\text{MIBP}_k) \}.
    \]

    Assume first that $(ii)$ holds. Again, the integer variable $x$ can only have finitely many values. Let $\{v_1,\ldots,v_N\}\in \Z^p$ be all the possible values of $x$. For each value, we can consdier the follower's problem given by 
    \[
    (F_k) = \left\{\begin{array}{rl}
           \displaystyle\min_{y}  & c_f^{\top}y \\
            s.t. & \begin{cases}
                B_fy \leq b_f-A_fv_k,\\
                \forall j\in J,\, y_j\in \Z.
            \end{cases}
        \end{array}\right.
    \]
    Each problem $(F_k)$ is a bounded Mixed-Integer Linear Programming problem. Thus, it is either infeasible, or it admits a solution $y_k$. The conclusion follows exactly as in the previous case: either all follower's problems are infeasible making \eqref{ch05:eq:MIBP} infeasible as well, or \eqref{ch05:eq:MIBP} admits a solution which is given by the best pair $(v_k,y_k)$ among those where $y_k$ solves $(F_k)$.
\end{proof}

\begin{note}{The challenges of MIBP}
Mixed-Integer Bilevel programming is an extremely challenging field of research. In particular, it is well-known that 
\begin{itemize}
    \item Check feasibility of a pair $(x,y)$ in \eqref{ch05:eq:MIBP} is NP-hard.
    \item Decision versions of \eqref{ch05:eq:MIBP} are in fact $\Sigma_2^p$-hard \citep{Jeroslow1985}.
\end{itemize}
In particular, this means that there is no hope for efficient algorithms, unless $P=NP$. However, methods to deal with this family of problems is an active field of research in computational optimization \citep{Kleinert2021SurveyMIBP}.
\end{note}
\section{Bilevel games}
\label{ch05:Games}

A natural extension of optimization problems are Nash games, also known as Nash Equilibrium problems. They are an extension of the 2-players games studied by Cournot (See Chapter~\ref{Chapter02:Model}), and they were heavily developed during the 20th century, probably with the highest points in the 50's, with the contributions \citep{Nash1950} and \citep{ArrowDebreu1954}. Nowadays, one can access to well-curated monographs such as \citep{Tadelis2013GameTheory} for a primer in the field.
\begin{formulation}{Nash game (or Nash Equilibrium Problem, NEP)}
The model of a Nash game is the following. We consider $N$ players, each of them controlling a decision vector $x_i\in X_i\subseteq\R^{n_i}$, with $i\in [N]$. We set $n= \sum_{i=1}^N n_i$ and we call a vector $x = (x_i\, :\, i\in [N])$ a decision profile of the game.\\ 

For a player $i\in [N]$ we usually write $x=(x_i,x_{-i})$ to emphasize the decision vector $x_i$ within the profile $x$, while $x_{-i}$ denotes the decision vectors of all other players, that is,  $x_{-i}=(x_j\,:\, j\in[N]\setminus\{i\})$. We extend this notation to also write $\R^{n_{-i}} := \prod_{j\neq i} \R^{n_j}$ and $X_{-i}:= \prod_{j\neq i} X_j$.\\ 

Each player $i\in [N]$ seeks to minimize a cost function $\theta_i:X_i\times X_{-i}\to \R$, depending on the decision vector $x_i\in X_i$, but also on the other players' profile $x_{-i}\in X_{-i}$.\\ 

Then, for a given profile $x_{-i}\in X_{-i}$, player $i$ aims to solve
\[
P_i(x_{-i}) = \left\{\begin{array}{cl}
    \displaystyle\min_{x_i} & \theta_i(x_i,x_{-i})  \\
     s.t. & x_i\in X_i. 
\end{array}\right.
\]
A profile $x^*\in X = \prod_{i=1}^N X_i$ is said to be an \textbf{equilibrium} if no player has incentives to deviate unilaterally of it. That is,
\[
\forall i\in [N],\,x_i^*\text{ solves }P_i(x_{-i}^*).
\]
\end{formulation}

\begin{example}[Cournot Duopoly revisited]\label{ch05:ex:CournotRevisited}
   The Cournot Duopoly we presented in Chapter~\ref{Chapter02:Model} is a good example of a Nash Equilibrium problem. Recall that the formulation is given by

     \begin{equation}\label{ch05:eq:CournotFormulation}
    \begin{array}{|c|c|}
        \hline
        &\\
        \begin{array}{cl}
            \displaystyle\max_{q_1}   & (p_0 -\alpha(q_1+q_2)-c)q_1\\
            \text{s.t.} & q_1\geq 0.
        \end{array}
        &
        \begin{array}{cl}
            \displaystyle\max_{q_2}   & (p_0 -\alpha(q_1+q_2)-c)q_2\\
            \text{s.t.} & q_2\geq 0.
        \end{array}\\
        
        &\\
        \hline
        \text{Problem }P_1(q_2)&\text{Problem }P_2(q_1)\\
        \hline
    \end{array}
\end{equation}
Recall that the unique equilibrium was given by $q^*=(q_1^*,q_2^*)$ with
\[
q_1^* = q_2^* = \frac{p_0 - c}{3\alpha}.
\]
Thus, for example, for $p_0 = 10$, $c=1$ and $\alpha=1$, the equilibrium is given by $q^*=(3,3)$.
\end{example}

In a Nash game, the action of players affect each other on the cost function only, and each player decides her actions over a static set. An extension of this model, are the so-called Generalized games, where the actions of players also affect the feasible sets of each other.
\begin{formulation}{Generalized Nash Equilibrium Problem (GNEP)} We consider $N$ players and for each $i\in [N]$, the set $X_i\in\R^{n_i}$ and the cost function $\theta_i$ as in the NEP description. For each player $i\in [N]$ we also consider the constraint map $K_i:X_{-i}\tto X_i$, where $K_i(x_{-i})$ represents the feasible set for player $i$ given a decision profile $x_{-i}$.

Then, for a given profile $x_{-i}\in X_{-i}$, player $i$ aims to solve
\[
P_i(x_{-i}) = \left\{\begin{array}{cl}
    \displaystyle\min_{x_i} & \theta_i(x_i,x_{-i})  \\
     s.t. & x_i\in K_i(x_{-i}). 
\end{array}\right.
\]
Let $K:X\tto X$ given by $K(x):= \prod_{i=1}^N K_i(x_{-i})$. A profile $x^*\in X = \prod_{i=1}^N X_i$ is said to be an \textbf{equilibrium} of the GNEP if $x^* \in K(x^*)$  and
\[
\forall i\in [N],\,x_i^*\text{ solves }P_i(x_{-i}^*).
\]
\end{formulation}

\begin{example}[Cournot Duopoly with market constraints]\label{ch05:ex:CournotMarketConstraints}
   Consider the Cournot Duopoly \eqref{ch05:eq:CournotFormulation} with an extra constraint of a maximum Market supply of $K>0$: that is, the total amount of the produced good must satisfy $q_1+q_2\leq K$. Then, the formulation is given by  

     \begin{equation}\label{ch05:eq:CournotFormulation-GNEP}
    \begin{array}{|c|c|}
        \hline
        &\\
        \begin{array}{cl}
            \displaystyle\max_{q_1}   & (p_0 -\alpha(q_1+q_2)-c)q_1\\
            \text{s.t.} & \begin{cases}
                q_1 + q_2\leq K,\\
                q_1\geq 0.
            \end{cases}
        \end{array}
        &
        \begin{array}{cl}
            \displaystyle\max_{q_2}   & (p_0 -\alpha(q_1+q_2)-c)q_2\\
            \text{s.t.} & \begin{cases}
                q_1 + q_2\leq K,\\
                q_2\geq 0.
            \end{cases}
        \end{array}\\
        
        &\\
        \hline
        \text{Problem }P_1(q_2)&\text{Problem }P_2(q_1)\\
        \hline
    \end{array}
\end{equation}
Then, the new formulation is a GNEP. For $p_0=10$, $\alpha = 1$, $c=1$ and $K=5$, the unique equilibrium $q^*=(3,3)$ is no longer feasible. In this case, the new set of equilibria are given by the segment $[(1,4),(4,1)]$ in $\R^2$. 

Indeed, for $q_2\geq 0$ fixed, the problem $P_1(q_2)$ is constraint qualified and its $KKT$ equations are given by
\begin{align*}
    9 - q_2 - 2q_1 - \mu_1 + \mu_2 &= 0,\\
    \mu_1(q_1 + q_2 - 5) &= 0,\\
    \mu_2(-q_1) &= 0,\\
    q_1 + q_2 - 5 &\leq 0,\\
    -q_1 &\leq 0,\\
    \mu_1,\mu_2&\geq 0.
\end{align*}

If $q_2 > 5$, the problem is infeasible, let us assume that $q_2\leq 5$. Thus, $q_1\in [0, 5-q_2]$. Then, we distinguish the following cases.
\begin{itemize}
    \item If $\mu_2 = 0$ and $\mu_1>0$: Then $q_1 = 5-q_2$.
    \item If $\mu_2 = 0$ and $\mu_1=0$: Then $q_1 = (9-q_2)/2$. But this is only possible if $q_2\leq 1$.
    \item If $\mu_2>0$, then $q_1 = 0$. This yields that either $q_2 = 5$ or $\mu_1 = 0$. In the first case, $q_1 = 5-q_2$. In the second case, we obtain that $9-q_2 = - \mu_2 \leq 0$. Thus, $q_2\geq 9$ which is not possible.
\end{itemize}
We deduce that the best response  function of player 1 is given by
\[
B_1(q_2) := \begin{cases}
    \frac{9-q_2}{2}&\text{ if }q_2\in [0,1),\\
    5-q_2 &\text{ if }q_2\in [1,5].
\end{cases}
\]
By symmetry,
\[
B_2(q_1) := \begin{cases}
    \frac{9-q_1}{2}&\text{ if }q_1\in [0,1),\\
    5-q_1 &\text{ if }q_1\in [1,5].
\end{cases}
\]
Then, we deduce that the set of equilibria are given by 
\[
\mathcal{E}=\{ (q_1,q_2) \,:\, q_1+q_2 = 5,\text{ and } q_1,q_2\in [1,4] \}.
\]
Indeed, on the one hand, every pair $(q_1,q_2)\in\mathcal{E}$ verifies that $q_i\in [1,5]$ and $B_{i}(q_{-i}) = 5-q_{-i} = q_i$ for $i=1,2$. On the other hand:

If $q_1<1$, then $q_2$ must coincide with $B_2(q_1) = \frac{9-q_1}{2}>4$. But then
    \[
    B_1(q_2) = 5 - \frac{9 - q_1}{2} = \frac{1}{2} + \frac{q_1}{2} > q_1. 
    \]
    Thus, $(q_1,q_2)$ with $q_1<1$ cannot be an equilibrium (and by symmetry, the same holds for $q_2$).
If $q_1>4$, then $q_2$ must coincide with $B_2(q_1) =5-q_1<1$. But then
    \[
    B_1(q_2) = \frac{9 - (5-q_1)}{2} = \frac{q_1 + 4}{2} < q_1.
    \]
Thus, $(q_1,q_2)$ with $q_1>4$ cannot be an equilibrium (and by symmetry, the same holds for $q_2$).

We conclude that the set $\mathcal{E}$, which is exactly the interval $[(1,4),(4,1)]$, is the set of all equilibria of the problem.
\end{example}

Probably the most celebrated result in terms of existence of equilibria for GNEP is the Arrow-Debreu theorem, presented in \citep{ArrowDebreu1954}.

\begin{theorem}[Arrow, Debreu, 1954] Suppose that for each $i\in [N]$ we have that
\begin{enumerate}
    \item[(i)] $X_i$ is nonempty, convex and compact.
    \item[(iii)] $K_i:X_{-i}\tto X_i$ is both upper semicontinuous and lower semicontinuous, with convex closed values.
    \item[(iii)] The function $\theta_i$ is (jointly) continuous and for each $x_{-i}\in X_{-i}$, the function $\theta_i(\cdot, x_{-i})$ is convex on $K_i(x_{-i})$.
\end{enumerate}
Then the GNEP described by $\{X_i\}_{i\in [N]}$, $\{\theta_i\}_{i\in [N]}$ and $\{K_i\}_{i\in [N]}$, admits an equilibrium.   
\end{theorem}

Now, the idea is to use the concept of GNEP to define a bilevel game. The structure will be given a by a set of players, the leaders and the followers The leaders will produce a profile $x$ that parametrized the game of the followers. The followers then will react with an equilibrium $y\in \mathcal{E}(x)$.

\begin{figure}[ht]
    \centering
   \definecolor{zzttqq}{rgb}{0.26666666666666666,0.26666666666666666,0.26666666666666666}
   \begin{tikzpicture}[line cap=round,line join=round,>=triangle 45,x=1.0cm,y=1.0cm,scale=0.85]
	
	\fill[fill=black,fill opacity=0.10000000149011612] (-3.,3.+3) -- (-3.,2.+3) -- (-2.,2.+3) -- (-2.,3.+3) -- cycle;
	\fill[fill=black,fill opacity=0.10000000149011612] (-1.,3.+3) -- (-1.,2.+3) -- (0.,2.+3) -- (0.,3.+3) -- cycle;
	\fill[fill=black,fill opacity=0.10000000149011612] (3.,3.+3) -- (3.,2.+3) -- (4.,2.+3) -- (4.,3.+3) -- cycle;
	\draw[dashed,color=zzttqq] (-3.6,3.4+3) -- (-3.6,1.56+3) -- (4.44,1.56+3) -- (4.44,3.4+3) -- cycle;
	
	\draw (-3.,3.+3)-- (-3.,2.+3)-- (-2.,2.+3) -- (-2.,3.+3)-- cycle;
	\draw (-1.,3.+3)-- (-1.,2.+3)-- (0.,2.+3)-- (0.,3.+3)--cycle;
	\draw (3.,3.+3)-- (3.,2.+3)-- (4.,2.+3)-- (4.,3.+3)--cycle;
	
	\draw  (-1.,2.44+3)-- (-2.,2.44+3);
	\draw (0.,2.38+3)-- (0.7,2.38+3);
	\draw [dashed] (0.7,2.38+3)-- (2.3,2.38+3);
	\draw (2.3,2.38+3)-- (3.,2.38+3);
	
	\draw[color=black] (-2.5,2.5+3) node {\scriptsize $L_1$};
	\draw[color=black] (-0.5,2.5+3) node {\scriptsize $L_2$};
	\draw[color=black] (3.5,2.5+3) node {\scriptsize $L_n$};
	\draw[color=black] (-5.5,5.5) node {\small Leaders};
	

	\fill[fill=black,fill opacity=0.10000000149011612] (-3.,3.) -- (-3.,2.) -- (-2.,2.) -- (-2.,3.) -- cycle;
	\fill[fill=black,fill opacity=0.10000000149011612] (-1.,3.) -- (-1.,2.) -- (0.,2.) -- (0.,3.) -- cycle;
	\fill[fill=black,fill opacity=0.10000000149011612] (3.,3.) -- (3.,2.) -- (4.,2.) -- (4.,3.) -- cycle;
	\draw[dashed,color=zzttqq] (-3.6,3.4) -- (-3.6,1.56) -- (4.44,1.56) -- (4.44,3.4) -- cycle;
	
	\draw (-3.,3.)-- (-3.,2.)-- (-2.,2.) -- (-2.,3.)-- cycle;
	\draw (-1.,3.)-- (-1.,2.)-- (0.,2.)-- (0.,3.)--cycle;
	\draw (3.,3.)-- (3.,2.)-- (4.,2.)-- (4.,3.)--cycle;
	
	\draw  (-1.,2.44)-- (-2.,2.44);
	\draw (0.,2.38)-- (0.7,2.38);
	\draw [dashed] (0.7,2.38)-- (2.3,2.38);
	\draw (2.3,2.38)-- (3.,2.38);
    
	\draw[color=black] (-2.5,2.5) node {\scriptsize $F_1$};
	\draw[color=black] (-0.5,2.5) node {\scriptsize $F_2$};
	\draw[color=black] (3.5,2.5) node {\scriptsize $F_m$};
	
	\draw[color=black] (-5.5,2.5) node {\small Followers};
    
    
	\draw [->] (4.44/2-3.6/2-.5,1.56+3)-- (4.44/2-3.6/2-.5,3.4);
	\draw[color=black,left] (4.44/2-3.6/2-.7,1.56/2+6.4/2) node {\scriptsize $x$};

	\draw [<-] (4.44/2-3.6/2+.5,1.56+3)-- (4.44/2-3.6/2+.5,3.4);
	\draw[color=black,right] (4.44/2-3.6/2+.7,1.56/2+6.4/2) node {\scriptsize $y\in \mathcal{E}(x)$};
	
	\end{tikzpicture}
    \caption{Structure of Multi-Leader-Follower games.}\label{ch05:fig:MLF}
\end{figure}
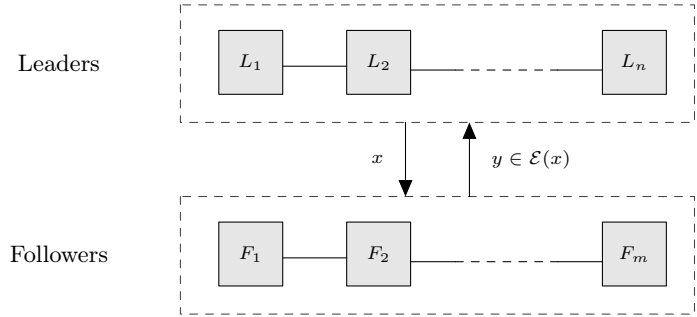 

\begin{formulation}{Bilevel Games}
    We consider a set of leaders $I$ and a set of followers $J$. Each leader $i\in I$ controls a decision vector $x_i\in X_i\subset\R^{p_i}$, and each follower $j\in J$ controls a decision vector $y_j\in Y_j\subset\R^{q_j}$. We set $X = \prod_{i\in I}X_i$ and $Y=\prod_{j\in J}Y_j$.\\ 
    
    Each leader $i\in I$ aims to minimize a cost function $\theta_i: X_i\times X_{-i}\times Y\to \R$ and each follower $j\in J$ aims to minimize a function $f_j: Y_j\times Y_{-j}\times X\to \R$. For each follower, we consider a constraint map $K_j:X\times Y_{-j}\tto Y_j$.\\

    Then, for a given profile $x\in X$, the followers aim to solve the Generalized Nash Equilibrium problem $GNEP(x)$ given by the problems
    \[
    P_j(x,y_{-j}) = \left\{\begin{array}{cl}
    \displaystyle\min_{y_j} & f_j(y_j,y_{-j},x) \\
     s.t. & y_j\in K_j(x,y_{-j}). 
\end{array}\right.
    \]

    Then, each leader aims to solve 
    \[
    P_i(x_{-i}) = \left\{\begin{array}{cl}
    \displaystyle\min_{x_i} & \theta_i(x_i,x_{-i},y) \\
     s.t. & x_i\in X_i, y\in\mathcal{E}(x), 
\end{array}\right.
    \]
    where $\mathcal{E}(x)$ denotes the equilibrium set of $GNEP(x)$. A profile $x^*\in X$ is said to be a leaders' equilibrium if for each  $i\in I$, $x_i^*$ solves $P_i(x_{-i}^*)$.
\end{formulation}

The concept of bilevel games seems to be introduced in \citep{PangFukushima2005MLF}, under the name Multi-Leader-Follower games. Indeed, this name is very useful since we can distinguish two particular cases:
\begin{enumerate}
    \item When there is only one leader and multiple followers, we refer the bilevel game as a \textbf{Single-Leader-Multi-Follower (SLMF)} game.
    \item When there is only one follower and multiple leaders, we refer the bilevel game as a \textbf{Multi-Leader-Follower-Follower (MLSF)} game.
\end{enumerate}

Of course, the formulation of a Bilevel game is ill-posed, since there is no clear choice of $y\in \mathcal{E}(x)$, when the set of equilibria is not a singleton. One could describe optimistic/pessimistic leaders, but however such definition would induce inconsistencies. How to deal with this is a very delicate matter that is still subject of discussion in the literature (see, e.g., \citep{AusselSvensson2020Short}). 

However, for the particular case of SLMF games, it is easy to talk about optimistic and pessimistic approaches, since there is only one leader. For the case of an optimistic leader, the formulation can be written as
\begin{equation}\label{ch05:eq:SLMF}
    \left\{\begin{array}{cl}
      \displaystyle\min_{x,y}   & \theta(x,y)  \\
       s.t.  & \begin{cases}
           x\in X,\\
           y\in \mathcal{E}(x).
       \end{cases} 
    \end{array}\right.
    \tag{SLMF}
\end{equation}
With this formulation, existence of solutions for \eqref{ch05:eq:SLMF} can be reduced, as the single-follower counterpart, to the closedness of the set-valued map $\mathcal{E}:X\tto Y$. We present here the existence result of \citep{AusselSvensson2018EPCC}.

\begin{theorem} Consider \eqref{ch05:eq:SLMF} and suppose that:
    \begin{enumerate}
        \item[(i)] $\theta$ is lower semicontinuous and $X$ is compact.
        \item[(ii)] For each $j\in J$, $Y_j$ is compact and $f_j$ is continuous.
        \item [(iii)] For each $j\in J$, $K_j:X\times Y_{-j}\tto Y_j$ is lower semicontinuous and has closed graph.
    \end{enumerate}
   Then,  either \eqref{ch05:eq:SLMF} is infeasible or it admits a solution.
\end{theorem}

\begin{proof}
    Let us assume that \eqref{ch05:eq:SLMF}  is feasible. We only need to show that in this case, there is a solution. Let us show first that the set-valued map GNEP has closed graph, thus defining a closed constraint set for the leader. Let us observe that we can write 
    \[
    \mathcal{E}(x) = \bigcap\limits_{j\in J} S_{j}(x)
    \]
    with 
    \[
    S_j(x) := \big\{(y_j,y_{-j})  \;\mid\; y_j\in \text{argmin}_{z} \{f_j (x,z,y_{-j}) \;\mid\; z_j \in K_j (x,y_{-j})\} \big\}.
    \]
    
    It is sufficient to prove that of each $j\in J$, the map $S_j:X\tto Y$ has closed graph. Let us
    fix $j\in J$ and take sequences $(x_k)_k$  in $X$ and $(y_k)_k$ in $Y$ converging respectively to $x$ and $y$, and such that $y_k  \in S_j(x_k)$ for all $k \in \N$. We want to prove that $y \in S_j(x)$. Note that
    \[
    (x_k,y_k) \in \gph S_j \implies y_{j,k}\in K_{j}(x_k,y_{-j,k})\implies (x_k,y_k) \in \gph K_j.
    \]
    Thus, since $K_j$ has closed graph, we get that $y_j \in K_j(x,y_{-j})$. Take $z_j \in  K_j(x,y_{-j})$. By lower semicontinuity of the set-valued map $K_j$, we know that there exists $z_{j,k} \in K_j(x_k,y_{-j,k})$ such that $z_{j,k} \rightarrow z_j$.
    Since $y_k \in S_j(x_k)$ then
    \[
    f_j(x_k,y_{j,k},y_{-j,k})\leq f_j(x_k,z_{j,k},y_{-j,k}), \forall k \in \mathbb{N}.
    \]
    Taking limit as $k\to\infty$,  continuity yields $f_j(x,y_{j},y_{-j})\leq f_j(x,z_{j},y_{-j}).$ Since $z_{j}$ was arbitrarily chosen from $K_j(x, y_{-j})$, we conclude that $ y \in S_j(x)$. Thus $S_j$ is closed and hence,  $\mathcal{E}$ has closed graph..
    
    Observe also that $\gph \mathcal{E}$ is also bounded, since 
    \[
    \gph \mathcal{E} \subset X\times Y,
    \]
    and the right-hand set is compact.  Since the objective function $\theta$ is lower semicontinuous, it follows, by the Weierstrass theorem, that the optimization problem of the leader in \eqref{ch05:eq:SLMF} has a solution.
\end{proof}

\begin{note}{The challenges of Bilevel Games}
Literature with respect to Bilevel games is scarce with respect to its counterpart with only one leader and one follower. However, SLMF games have proven to be a very well-structured family of problems. The MPCC reformulations are also valid in this setting (under suitable hypotheses for the GNEP of the followers), and therefore they can be attacked algorithmically. When multiple leaders are to be considered, the scenario is way more complicated. It is worth to mention that for multiple leaders, the shared constraint $y\in \mathcal{E}(x)$ makes the problem a GNEP itself. This field, popularized in the last 10 years, is still wide open to contributions.
\end{note}

\section{The Bayesian Approach for bilevel optimization}
\label{ch05:Bayessian}

Let us consider the bilevel programming problem
\begin{equation}\label{ch05:eq:GeneralBP}
    \begin{array}{cl}
     \displaystyle \min_{x}    & \theta(x,y)  \\
       s.t.  & x\in X,\, y\in S(x),
    \end{array}
\end{equation}
where $S:X\tto Y$ is the set-valued map of solutions of the follower's problem. The ill-posedness of  \eqref{ch05:eq:GeneralBP} is solved by considering a belief of the leader on how the follower will select $y\in S(x)$:
\begin{itemize}
    \item For the optimistic approach, the leader \textbf{believes} that the follower will select $y\in S(x)$ to favor her the most.
    \item For the pessimistic approach, the leader \textbf{believes} that the follower will select $y\in S(x)$ to harm her the most.
\end{itemize}

However, these are not the only options. We can model beliefs using probability distributions. This idea was first introduced by \citep{Mallozzi1996}, and rediscovered in \citep{SalasSvensson2023}. For the rest of the chapter, we will assume that $X$ and $Y$ are nonempty compact sets.
\begin{note}{Borel probabilities and weak-convergence}
 Let $Y$ be a nonempty compact subset of $\R^q$. We denote by $\mathcal{B}/Y)$ the borel $\sigma$-algebra of $Y$ and by $\mathscr{P}(Y)$ the set of all Borel probability measures over $Y$. We endow  $\mathscr{P}(Y)$ with the topology of weak-convergence: that is, for a sequence $(\mu_k)\subset \mathscr{P}(Y)$ and $\mu\in \mathscr{P}(Y)$, we say that $\mu_k$ weak-convergences to $\mu$ if
\[
\forall f\in \mathcal{C}(Y),\quad \int_Y f(y) d\mu_k(y) \to \int_Y f(y) d\mu(y).
\]
In such a case, we write $\mu_k\wto \mu$. Accordingly, a function $\beta:X\to \mathscr{P}(Y)$ will be said \textit{weak-continuous} if for every sequence $(x_n)\subset X$ converging to $x\in X$, one has that $\beta(x_n)\wto \beta(x)$.
\end{note}

In the sequel, we will use a fundamental theorem of probability, called Portemanteu Theorem. While this is very classic in literature and can be found in many monographs, we present the version contained in \citep{SalasSvensson2023}.

\begin{theorem}[Portemanteau Theorem]\label{ch05:thm:Portemanteau} Let $Y$ be a nonempty closed set of $\R^q$, $(\nu_k)$ be a sequence of probability measures of $\mathscr{P}(Y)$, and let $\nu\in \mathscr{P}(Y)$. The following assertions are equivalent:
\begin{enumerate}
    \item[(i)] $\nu_k\wto \nu$.
    \item[(ii)] For every function $f:Y\to\R$ lower semicontinuous and bounded from below, \[
    \liminf_k \int_Y f d\nu_k \geq \int_Y fd\nu.\] 
    \item [(iii)] For every closed set $C\subset Y$, $\limsup_k \nu_k(C)\leq \nu(C)$.
    \item[(iv)] For every open set $U\subset Y$, $\liminf_k\nu_k(U)\geq \nu(U)$.
\end{enumerate}   
\end{theorem}

With all these elements, we are ready to formulate the Bayesian approach for \eqref{ch05:eq:GeneralBP}. We first need to formalize what a belief should be in this setting.
\begin{definition}[Belief]\label{ch05:def:Belief} Let $S:X\tto Y$ be a set-valued map with nonempty measurable values (i.e., $S(x)\in\mathcal{B}(Y)\setminus\{\emptyset\}$ for every $x\in X$). A \textbf{belief over $\mathbf{S}$} is a mapping $\beta:X\to \mathscr{P}(Y)$ assigning to each $x\in X$ a probability measure $\beta(x) = \beta_x$ that verifies
\[
\forall x\in X,\quad \beta_x(S(x)) = 1.
\]   
\end{definition}

If the leader has a belief $\beta$ over $S$, it means that for every decision $x\in X$ she predicts that the follower will draw $y\in S(x)$ following the probability measure $\beta_x$. That is, the leader models the response of the follower as a random variable following the probability distribution $\beta_x$.

\begin{formulation}{Bayesian approach}
    For a given belief $\beta$ over the set-valued map $S$, the Bayesian approach of \eqref{ch05:eq:GeneralBP} is given by
    \begin{equation}\label{ch05:eq:BayesianApproach}
    \min_{x\in X} \mathbb{E}_{y\sim\beta_x}[ \theta(x,y)], 
    \end{equation}
    where $\mathbb{E}_{y\sim\beta_x}[ \theta(x,y)] := \int_{S(x)}\theta(x,y)d\beta_x(y)$. All the bilevel complexity is hidden now in the belief $\beta$.
\end{formulation}

A natural application of the Bayesian approach is to consider a belief based on the uniform distributions on the sets $S(x)$. This captures the idea of having no information about how the follower selects $y\in S(x)$.
\begin{formulation}{Neutral approach}
    Suppose that $S:X\tto Y$ has nonempty closed convex values. We define the neutral belief $\iota:X\to \mathscr{P}(Y)$ given by
    \[
    \iota_x(A) = \frac{\mathcal{H}^{d_x}(A\cap S(x))}{\mathcal{H}^{d_x}(S(x))},
    \]
    where $d_x$ is the Hausdorff dimension of $S(x)$. Note that under the convexity assumption, $d_x$ is always integer and $\mathcal{H}^{d_x}(S(x))>0$, making the neutral belief well-defined. The Neutral approach is then the Bayesian approach of \eqref{ch05:eq:BayesianApproach} considering the neutral belief $\iota$.
\end{formulation}

\begin{example}
		Let us consider in $\R^2$ the polygon $P$ given as the convex hull of the points $(0,4)$, $(8,0)$, $(8,8)$, $(10,1)$ and $(10,5)$. Now, consider the reaction map $S:[0,10]\tto [0,8]$ defined by the equality $\gph S = P$, and the leader's problem
		\begin{align*}
			\min_x\{ y:\, y\in S(x), x\in \dom(S)\},
		\end{align*}
		which is clearly ill-posed due to the multiplicity of optimal reactions of the follower. Figure \ref{fig:exampleneutral} compares the optimistic, pessimistic and neutral value functions. The optimistic, pessimistic and neutral solutions are $x=8$, $x=0$ and $x=10$, respectively. The optimal values are $\varphi^o(8)=0$, $\varphi^p(x) = 4$ and $\varphi^n(10) = 3$.
		
		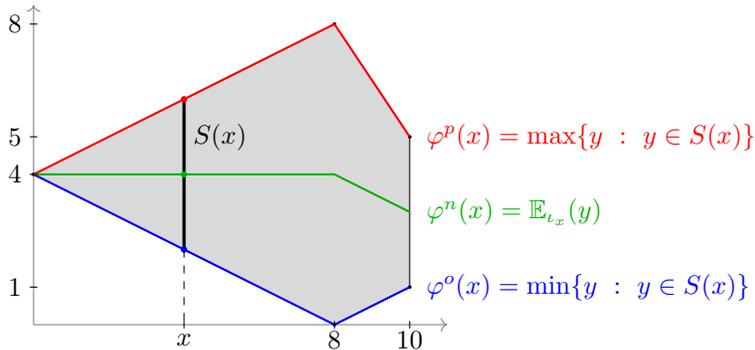
\begin{figure}[ht]
			\begin{center}
				\begin{tikzpicture}[scale=.5]
					\draw[->,gray] (0,0) -- (11,0);
					\draw[->,gray] (0,0) -- (0,8.5);
					
					\filldraw[fill=gray!30!white] (0,4) -- (8,0) -- (10,1) -- (10,5) -- (8,8) -- (0,4);
					\filldraw[black] (0,4) circle (1pt);
					\filldraw[black] (8,0) circle (1pt);
					\filldraw[black] (10,1) circle (1pt);
					\filldraw[black] (10,5) circle (1pt);
					\filldraw[black] (8,8) circle (1pt);
					
					\draw[black, very thick] (4,2)--(4,6);
					\filldraw[red] (4,6) circle (2pt);
					\filldraw[black!30!green] (4,4) circle (2pt);
					\filldraw[blue] (4,2) circle (2pt);
					\node[right] at (4,5) {$S(x)$};
					\draw[dashed] (4,2)--(4,0);
					\draw (4,-0.1) -- (4,0.1);
					\node at (4,-.4) {$x$};
					
					\draw[red, thick] (10,5) -- (8,8) -- (0,4);
					\draw[blue, thick] (0,4) -- (8,0) -- (10,1);
					\draw[black!30!green, thick] (0,4) -- (8,4) -- (10,3);
					
					\node[blue,right] at (10.2,1) {$\varphi^o(x) = \min\{ y\ :\ y\in S(x)  \}$};
					\node[red,right] at (10.2,5) {$\varphi^p(x)= \max\{ y\ :\ y\in S(x)  \}$};
					\node[black!30!green,right] at (10.2,3) {$\varphi^n(x) = \E_{\iota_x}(y)$};
					\draw (8,-0.1) -- (8,0.1);
					\draw (10,-0.1) -- (10,0.1);
					\node at (10,-.4) {10};
					\node at (8,-.4) {8};
					
					\draw (-0.1,1) -- (0.1,1);
					\draw (-0.1,4) -- (0.1,4);
					\draw (-0.1,5) -- (0.1,5);
					\draw (-0.1,8) -- (0.1,8);
					\node at (-.5,1) { 1};
					\node at (-.5,4) { 4};
					\node at (-.5,5) { 5};
					\node at (-.5,8) { 8};
				\end{tikzpicture}
			\end{center}
			\caption{Functions $\varphi^o$, $\varphi^p$ and $\varphi^n$ are the value functions that the leader aims to minimize for the optimistic, pessimistic and neutral approach, respectively.} 
			\label{fig:exampleneutral}
		\end{figure}
		\label{ex:blpapproaches}
	\end{example}

We now provide the elemental theorem of existence for the Bayesian approach: If the belief is weak-continuous, then \eqref{ch05:eq:BayesianApproach} will admit solutions. We present the exact statement and proof of \citep{SalasSvensson2023}. 

\begin{theorem}\label{ch05:thm:ExistenceSingleLeader} Let $Y$ be a nonempty compact set of $\R^q$ and let us consider \eqref{ch05:eq:BayesianApproach} with belief $\beta$ over $S:X\tto Y$. If one has that
		\begin{enumerate}
			\item[(i)]the cost function $\theta$ is lower semicontinuous,
			\item[(ii)] the constraints set $X$ is nonempty and compact, and
			\item [(iii)] the belief $\beta: X\to\mathscr{P}(Y)$ is weak continuous,
		\end{enumerate}
		then, $x\mapsto \E_{\beta_x}(\theta(x,\cdot))$ is lower semicontinuous and so, \eqref{ch05:eq:BayesianApproach}  admits at least one solution.
	\end{theorem}
	
	\begin{proof}
		We will only show the lower semicontinuity of $x\mapsto \E_{\beta_x}(\theta(x,\cdot))$ since the existence of solutions will trivially follow from the compactness of $X$. Let $(x^n)\subset X$ be a sequence converging to a point $x\in X$.
		
		Fix $\varepsilon>0$. Since $\theta$ is lower semicontinuous, for each $y\in Y$ there exists a neighborhood $U_y\times V_y$  
		of $(x,y)$ such that $\mathrm{diam}(V_{y})\leq \varepsilon$ and
		\[
		\forall (x',y')\in (U_y\times V_y) \cap (X\times Y),\quad \theta(x',y')> \theta(x,y)-\varepsilon.
		\]
		Without loss of generality, we may choose $U_{y}\times V_{y}$ to be closed. 
		Now, since $Y$ is compact, there exists a finite cover $\mathcal{V}_{\varepsilon}=\{V_{k}\}_{k=1}^m$ of $Y$, included in $\{V_{y}\ :\  y\in Y\}$. Let $\{y_{k}\}_{k=1}^m$ be a finite sequence such that $V_k=V_{y_k}$. 
		Consider the neighborhood of $x$ given by $U:=\bigcap_{k=1}^m U_{y_k}$. Note that $U$ satisfies
		\[
		\forall (x',y')\in (U\cap X)\times (V_{k}\cap Y),\quad  \theta(x',y')> \theta(x,y_k)-\varepsilon, 
		\]
		for each $k\in \{1,\ldots,m\}$. Let us define the function $\varphi_{\varepsilon}:Y\to \R$ given by
		\[
		\varphi_{\varepsilon}(y) = \min\{ \theta(x,y_k)-\varepsilon\ :\ y\in V_k\}. 
		\]
		
		By construction, $\varphi_{\varepsilon}$ is lower semicontinuous, bounded from below and satisfies $\varphi_{\varepsilon}\leq \theta(x',\cdot)$ for any $x'\in U$.
		In particular, for $n\in\N$ large enough $x^n\in U$, and hence, for every $y\in Y$, we have $\theta(x^n,y) \geq \varphi_{\varepsilon}(y)$. This yields, by the Portemanteau Theorem (see Theorem \ref{ch05:thm:Portemanteau}), that
		\begin{align*}
			\liminf_n \E_{\beta_{x^n}}(\theta(x^n,y)) &= \liminf_n\int_Y \theta(x^n,y) d\beta_{x^n}(y)\\
			&\geq \liminf_n\int_Y \varphi_{\varepsilon}(y)d\beta_{x^n}(y)\geq \int_Y \varphi_{\varepsilon}(y)d\beta_{x}(y).
		\end{align*}
		
		Now, let $\varepsilon_j = \frac{1}{j}$ and let $\varphi_j :=\varphi_{\varepsilon_j}$. We claim that $\varphi_j$ converges pointwise to $\theta(x,\cdot)$. Indeed, choose $y\in Y$. First, since $\varphi_j \leq \theta(x,\cdot)$ for every $j\in \N$, it is clear that $\limsup_j \varphi_j(y)\leq \theta(x,y)$. 
		
		By construction, for each $j\in \N $, there exists $y_j$ such that
		\[
		\varphi_j(y) = \theta(x,y_j)-\frac{1}{j}\quad\text{and}\quad\|y_j-y\|\leq \frac{1}{j}.
		\]
		Then, we can write
		\begin{align*}
			\liminf_j \varphi_j(y) &= \liminf_j \left(\theta(x,y_j)-\frac{1}{j}\right)\geq \theta(x,y),
		\end{align*}
		where the last inequality follows from the fact that $\theta$ is lower semicontinuous and that $y_j\to y$. Thus, we have shown that
		\[
		\theta(x,y) \leq \liminf_j \varphi_j(y) \leq \limsup_j \varphi_j(y)\leq \theta(x,y),
		\]
		proving that $\varphi_j(y)\to \theta(x,y)$, and, by arbitrariness of $y\in Y$, proving our claim. 
		
		Now, since $\theta$ attains its minimum over $X\times Y$, we have that for every $j\in \N$, $\varphi_j \geq \min_{X\times Y} \theta -1$, and so, applying the Fatou Lemma for the measure $\beta_x$, we can write that
		\begin{align*}
			\liminf_n \E_{\beta_{x^n}}(\theta(x^n,\cdot)) &\geq \liminf_j \int_Y \varphi_j(y)d\beta_x(y)\\ 
			&\geq \int_Y \theta(x,y)d\beta_{x}(y) = \E_{\beta_{x}}(\theta(x,\cdot)).
		\end{align*}
		Since $(x^n)$ was an arbitrary sequence, we have shown the lower semicontinuity of the mapping $x\mapsto \E_{\beta_x}(\theta(x,\cdot))$, which finishes the proof. 
	\end{proof}

Note that the statement of Theorem \ref{ch05:thm:ExistenceSingleLeader} does not consider any hypothesis on the set-valued map $S$. This yields the question on how easy or hard is to verify weak-continuity of a belief. In general, this can be very challenging.

   \begin{example}\label{ch05:example:triangletosegment}
    It is not hard to show that if the neutral belief over a set-valued map $S$ with compact and convex values were weak-continuous, then the centroid map $x\mapsto \mathfrak{c}(S(x))$ would be continuous as well. Consider $S:[0,1]\rightrightarrows [0,1]^2$ defined by 
	\[
	S(x)=\{(y_1,y_2)\in[0,1]^2: y_2\leq x\cdot y_1 \}
	\]
    
    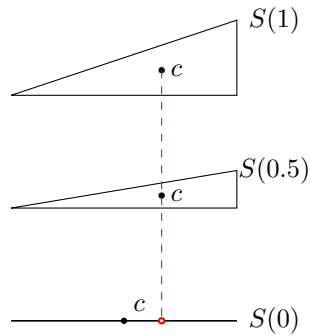
\begin{figure}[ht]
    \centering
			\begin{tikzpicture}
				\draw (0,0) -- (3,0) -- (3,1) --     (0,0);
				\node at (3.5,1) {$S(1)$};
				\filldraw[black] (2,0.333) circle (1pt) node[anchor=west] {$c$};

				\draw (0,0-1.5) -- (3,0-1.5) -- (3,.5-1.5) --     (0,0-1.5);
				\node at (3.5,.5-1.5) {$S(0.5)$};
				\filldraw[black] (2,0.166-1.5) circle (1pt) node[anchor=west] {$c$};

				\draw (0,0-3) -- (3,0-3) -- (3,0-3) --     (0,0-3);
				\node at (3.5,0-3) {$S(0)$};
				\filldraw[black] (1.5,0-3) circle (1pt) node[above right] {$c$};
				\draw[red,dashed](2,0.333)--(2,-3);
				\filldraw[fill=white, draw=red,thick] (2,-3) circle (1pt);
		\end{tikzpicture}
        \caption{Illustration of the variation of the centroid of $S$.}\label{ch05:fig:TriangleToSegment}
  \end{figure}

  In this case, one can compute explicitly the centroid map, as illustrated in Figure~\ref{ch05:fig:TriangleToSegment}. Indeed
  \[
  \mathfrak{c}(S(x))=\begin{cases}
      (2/3, x/3)\quad& \text{ if } x\in (0,1]\\
      (1/2,0)\quad &\text{ if }x=0.
  \end{cases}
  \]
  We deduce that in this example, the neutral belief can not be continuous, even though $S$ is a continuous set-valued map. 
 \end{example}

The main issue in Example \ref{ch05:example:triangletosegment} is the change of dimension. This changes of dimension can happen even in the fundamental case of linear bilevel programming. Nevertheless, it is possible to prove that for linear bilevel programming, the neutral belief is indeed weak-continuous, leading to the main  and last result of this section.

\begin{theorem}[Salas, Svensson, 2023]\label{ch05:thm:BayesianForLinear} Suppose that
\[
S(x) = \argmin_{y}\{ c^{\top}y\;\mid\; Ax+By \leq b  \},
\]
and that $X\subset\dom S$. Then \eqref{ch05:eq:BayesianApproach} admits a solution under the Neutral approach (i.e. Bayesian approach with $\beta = \iota$), provided that
		\begin{enumerate}
			\item[(i)] The decision set $X$ of the leader is  nonempty and compact.
            \item[(ii)] The cost function $\theta$ is lower semicontinuous.
			\item[(iii)] The constraint set $K(x):=\{y\in \R^q: Ax+By\leq b\}$ is bounded for each $x\in X$.
		\end{enumerate}
	\end{theorem}

The proof of Theorem~\ref{ch05:thm:BayesianForLinear} is highly technical, and it is the main result of \citep{SalasSvensson2023}. It can be extended to other types of functions and it is based on an enhanced notion of continuity of set-valued maps, called rectangular continuity. The key point is that the solution sets of parametric linear problems might change of dimension, but they do in a balanced way. 

\begin{note}{The challenges of the Bayesian approach}

Up to this moment, the Bayesian approach is in its early stages. Besides the original work of \citep{Mallozzi1996} where the concept was introduced, no developments can be found in the literature up to \citep{SalasSvensson2023}. Probably because working with beliefs requires some heavy mathematics. Nevertheless, in \citep{SalasSvensson2023} it was shown that the approach is well-possed for several fundamental cases in the theory of Bilevel programming. After this, a first contribution in algorithms to solve these problems was provided in \citep{MunozSalasSvensson2023FullLength} for the particular case of vertex-supported beliefs in linear programming. However, there is a lot to do here: algorithms, applications, variational properties. Everything is open. It is worth to mention that the name ``Bayesian'' was chosen because of the potential of beliefs $\beta$ to be refined by learning. This is also something to be done. Note also that the Bayesian approach opens the door to solve inconsistencies in Multi-Leader-Follower games. 
\end{note}

\newpage
\section{Problems}\label{ch05:sec:Problems}

\begin{problem}[Failing existence of solutions in MLSF games] Consider the Bilevel game with two leaders and one follower given by
\[
\begin{array}{|c|c||c|}
        \hline
        &&\\
        \begin{array}{cl}
            \displaystyle\min_{x_1}   & \frac{1}{2}x_1+y\\
            \text{s.t.} & x_1\in [0,1].
        \end{array}
        &
        \begin{array}{cl}
            \displaystyle\min_{x_2}   & -\frac{1}{2}x_2 - y\\
            \text{s.t.} & x_2\in [0,1].
        \end{array}
        &
        \begin{array}{cl}
            \displaystyle\min_{y}   & y(x_1+x_2-1) + \frac{1}{2}y^2\\
            \text{s.t.} & y\geq 0.
        \end{array}\\
        
        &&\\
        \hline
        \mathrm{Leader}_1&\mathrm{Leader}_2&\mathrm{Follower}\\
        \hline
    \end{array}
\]  
Show that the follower's best response is single-valued, and show that the game doesn't admit any equilibrium.
\end{problem}
\vspace{0.5cm}

\begin{problem}[Continuity of the centroid as necessary condition] Let $K\subset\R^q$ be a nonempty convex compact set. The centroid of $K$ is given by 
\[
\mathfrak{c}(K) := \frac{1}{\mathcal{H}^{d}(K)}\int_K y d\mathcal{H}^{d}(y),
\]
where $d$ is the Hausdorff dimension of $K$ and $\mathcal{H}^{d}$ is the $d$-dimensional Hausdorff measure. Note that $\mathcal{H}^{d}$ coincides with the Lebesgue measure over the affine space generated by $K$.\\

Show that if $S:X\tto Y\subset\R^q$ is a set-valued map with nonmepty convex and compact values such that the neutral belief over $S$ is weak-continuous, then the centroid map $x\mapsto \mathfrak{c}(S(x))$ must be continuous as well.
\end{problem}

\cleardoublepage

\phantomsection

\addcontentsline{toc}{chapter}{Bibliography}
\nocite{*}
\bibliographystyle{abbrvnat}
\bibliography{chapters/References}
\end{document}